\documentclass[11pt,b5paper,notitlepage]{article}
\usepackage[b5paper, margin={0.5in,0.65in}]{geometry}
\usepackage{amsmath,amssymb,amsthm,amsfonts}
\usepackage{mathtools,array}
\usepackage{mathrsfs,indentfirst,mathabx}
\usepackage{palatino}

\let\OLDthebibliography\thebibliography
\renewcommand\thebibliography[1]{%
  \OLDthebibliography{#1}%
  \setlength{\parskip}{0pt}%
  \setlength{\itemsep}{2pt}%
}

\allowdisplaybreaks
\usepackage[colorlinks,linkcolor=blue,anchorcolor=blue]{hyperref}
\hypersetup{
citecolor=[rgb]{0,0.5,0},
pdftitle={Twisted Jacquet modules and induction from Speh representations},
pdfauthor={Alan Xuelun Hou}
}
\theoremstyle{plain}
\newtheorem{theorem}{Theorem}[section]
\newtheorem{lemma}[theorem]{Lemma}
\newtheorem{proposition}[theorem]{Proposition}
\newtheorem{corollary}[theorem]{Corollary}
\newtheorem{maintheorem}{Theorem}

\theoremstyle{definition}
\newtheorem{remark}[theorem]{Remark}

\numberwithin{equation}{section}

\DeclareMathOperator{\Sp}{Sp}
\DeclareMathOperator{\SO}{SO}
\DeclareMathOperator{\GL}{GL}
\DeclareMathOperator{\SL}{SL}
\DeclareMathOperator{\Mat}{Mat}
\DeclareMathOperator{\Ind}{Ind}
\DeclareMathOperator{\ind}{ind}
\DeclareMathOperator{\WF}{WF}
\DeclareMathOperator{\Hom}{Hom}
\DeclareMathOperator{\diag}{diag}
\DeclareMathOperator{\tr}{tr}
\DeclareMathOperator{\Lie}{Lie}
\newcommand{\calE}{\mathcal{E}}
\newcommand{\calH}{\mathcal{H}}
\newcommand{\bbC}{\mathbb{C}}
\newcommand{\bfS}{\mathbf{S}}
\newcommand{\one}{\mathbf{1}}

\title{Twisted Jacquet modules and induction from Speh representations}
\author{\scshape Alan Xuelun Hou}
\date{}

\begin{document}\sloppy
\pagenumbering{arabic}

\maketitle
\begin{abstract}
Fourier coefficients of Eisenstein series induced from Speh representations are the kernels of the generalized doubling method of Cai, Friedberg, Ginzburg and Kaplan. Ginzburg and Soudry integrated these coefficients against cusp forms to obtain Eisenstein series induced from shorter Speh representations and arbitrary cuspidal data. We determine the corresponding local modules over a non-archimedean local field $F$ of characteristic zero. Let $q_F$ be the cardinality of its residue field. Let $J_s$ be the twisted Jacquet module of the representation induced from $\Delta(\tau,m+i)|\det|^s$. It carries an action of $G\times H$, where $G$ and $H$ are split symplectic or special orthogonal groups. Let $\sigma$ be an admissible representation of $G$ of finite length. We prove that $(J_s\otimes\sigma)_G$ is induced from $\Delta(\tau,i)|\det|^s$ and a fixed conjugate of $\sigma$, outside a finite set of values of $q_F^{-s}$ depending on $\tau$ and $\sigma$. No genericity assumption on $\sigma$ is needed. For irreducible $\sigma$, this determines all irreducible quotients of $J_s$ of the form $\sigma^\vee\boxtimes\pi$. The proof computes every orbit contribution, including those that vanish globally by cuspidality. It also gives an explicit finite set containing the exceptional parameters. This set is empty for irreducible supercuspidal $\sigma$ unless $G$ is the split group $\SO_2$. We prove the analogous result for the symplectic double cover. Over finite fields of large characteristic, we give the complete decomposition of the twisted Jacquet module when $\tau$ is cuspidal on $\GL_n$ and $n$ is larger than the rank of $G$.
\end{abstract}

\section{Introduction}

Twisted Jacquet modules of representations induced from Speh representations occur in the local theory of the generalized doubling method of Cai, Friedberg, Ginzburg and Kaplan \cite{CFGK19,CFK22}. Let $G$ be a split classical group. This method gives integral representations of the tensor product $L$-functions of $G\times\GL_n$ for every cuspidal representation of $G$, without a genericity assumption. The local theory defines the corresponding $\gamma$, $\epsilon$ and $L$ factors for every irreducible representation of $G$ \cite{CFK22}. The required uniqueness theorem is due to Gourevitch and Kaplan \cite{GK23}. In the doubling setting, a twisted Jacquet module of a representation induced from a Speh representation is a representation of $G\times G$. Let $\sigma$ and $\pi$ be irreducible representations of $G$. Gourevitch and Kaplan proved that, outside a discrete set of parameters, a nonzero equivariant map from this module to $\sigma^\vee\boxtimes\pi$ exists only if $\pi$ is isomorphic to a fixed outer conjugate $\sigma^\iota$ of $\sigma$. Such a map is unique up to scalars.

Ginzburg and Soudry \cite{GS21} showed that these modules form the first level of a tower. Let $\mathbb A$ be the ring of adeles of a number field $K$. Fix integers $n,m\ge1$ and $i\ge0$. Let $\calH_r$ denote either the symplectic group $\Sp_r$, with $r$ even, or the split special orthogonal group $\SO_r$. In the symplectic case, assume that $m$ is even. Put $N=n(m+i)$, $\bfS=\calH_{2N}$, $G=\calH_m$ and $H=\calH_{2ni+m}$. Let $\tau$ be an irreducible cuspidal automorphic representation of $\GL_n(\mathbb A)$. Write $\Delta(\tau,m+i)$ for the corresponding Speh representation of $\GL_N(\mathbb A)$.

Let $E(f_s)$ be the Eisenstein series on $\bfS(\mathbb A)$ attached to a section $f_s$ of the representation induced from $\Delta(\tau,m+i)|\det|^s$ on the Siegel parabolic subgroup. Let $(U,\psi_U)$ be the unipotent subgroup and character attached to the nilpotent orbit with partition $((2n-1)^m,1^{2ni+m})$. The stabilizer of $\psi_U$ contains $G\times H$. The Fourier coefficient $E^{U,\psi_U}(f_s)$ is therefore an automorphic function on $G(\mathbb A)\times H(\mathbb A)$. Let $\sigma$ be an irreducible cuspidal automorphic representation of $G(\mathbb A)$. Ginzburg and Soudry \cite[Theorem~A]{GS21} proved that, for $\varphi_\sigma$ in the space of $\sigma$,
\begin{equation}\label{GS}
  \int_{G(K)\backslash G(\mathbb A)} E^{U,\psi_U}(f_s)(g,h)\,\varphi_\sigma(g)\,dg =E\bigl(\Lambda(f_s,\varphi_\sigma)\bigr)(h).
\end{equation}
The right hand side is the Eisenstein series on $H(\mathbb A)$ attached to an explicit section $\Lambda(f_s,\varphi_\sigma)$ of the representation induced from $\Delta(\tau,i)|\det|^s\otimes\sigma^\iota$ on the standard parabolic subgroup with Levi subgroup $\GL_{ni}\times G$. For $i=0$ we have $H=G$. In this case, the proof of (\ref{GS}) reduces to the unfolding of the generalized doubling integrals. For $n=1$, the Speh representation is the character $\tau\circ\det$. Then (\ref{GS}) is a pullback formula for Siegel Eisenstein series; see the introduction of \cite{GS21} for its history. Ginzburg and Soudry compare the tower in $i$ to towers of theta lifts, and they propose applications to the poles of Eisenstein series and to the unipotent orbits supported by their residues; see also \cite{GS22,GS24}.

Over a non-archimedean local field, let $J_s$ be the twisted Jacquet module of the representation induced from $\Delta(\tau,m+i)|\det|^s$ along $(U,\psi_U)$. It is a representation of $G\times H$. Given an irreducible representation $\sigma$ of $G$, for which irreducible representations $\pi$ of $H$ is $\sigma^\vee\boxtimes\pi$ a quotient of $J_s$? More precisely, which $H$-module occurs in the maximal $\sigma^\vee$-isotypic quotient of $J_s$? The global identity (\ref{GS}) assumes that $\sigma$ is cuspidal. In its unfolding, every double coset other than the open one contributes a constant term of $\varphi_\sigma$ along a proper parabolic subgroup of $G$. Cuspidality makes these terms vanish \cite[Corollary~3.3]{GS21}. Locally, the corresponding terms involve Jacquet modules of $\sigma$ and are in general nonzero.

We compute $(J_s\otimes\sigma)_G$ for every admissible representation $\sigma$ of $G$ of finite length. For all but finitely many values of $q_F^{-s}$, where $q_F$ is the residue field cardinality, these coinvariants are induced from $\Delta(\tau,i)|\det|^s\boxtimes\sigma^\iota$ (Theorem~\ref{main}). For irreducible $\sigma$, these coinvariants are exactly the $H$-module in the maximal $\sigma^\vee$-isotypic quotient of $J_s$. Under the hypotheses of Theorem~\ref{finite}, we also give the complete decomposition of the analogous module over a finite field. This module is multiplicity free. Its constituents define an explicit correspondence between irreducible representations of $G$ and of $H$. Both results rest on a computation of the contributions of all orbits, including those that vanish globally (Theorem~\ref{layers}).

\subsection{The main theorem}

Let $F$ be a non-archimedean local field of characteristic zero. Let $\mathcal O_F$ be its ring of integers. Fix a uniformizer $\varpi$ and put $q_F:=\#(\mathcal O_F/\varpi\mathcal O_F)$. We normalize the absolute value by $|\varpi|=q_F^{-1}$. Fix a nontrivial additive character $\psi$ of $F$. We keep the integers $n,m,i$ and the family $\calH_r$, and assume that $m$ is even in the symplectic case. We set $\calH_0(F):=\{1\}$ and $\GL_0(F):=\{1\}$, and put $N:=n(m+i)$, $\bfS:=\mathcal H_{2N}(F)$, $G:=\mathcal H_m(F)$ and $H:=\mathcal H_{2ni+m}(F)$. For a group $L=\calH_r(F)$ and $0\le k\le\lfloor r/2\rfloor$, let $Q_k^L$ be the standard parabolic subgroup with Levi subgroup $\GL_k\times\calH_{r-2k}$. Put $Q_N:=Q_N^{\bfS}$, the Siegel parabolic subgroup, and $Q_{ni}:=Q_{ni}^H$. We identify the Levi subgroup of $Q_{ni}$ with $\GL_{ni}(F)\times G$ by (\ref{Levi}). The conventions for representations, Jacquet functors, induction and the endpoint parabolic subgroups of even orthogonal groups are in Section~\ref{notation} at the end of this introduction.

Let $\tau$ be an irreducible unitary generic representation of $\GL_n(F)$. For $c\geq1$, let $\Delta(\tau,c)$ be the local Speh representation of Section~2.2 of \cite{CFGK23}. It is an $(n,c)$-representation by \cite[Theorem~4]{CFGK23}. We set $\Delta(\tau,0):=\one$. For $s\in\bbC$, define
\[
  \rho_s:=\Ind_{Q_N}^{\bfS}\bigl(\Delta(\tau,m+i)|\det|^s\bigr).
\]

Let $U$, $\psi_U$, and $t:G\times H\longrightarrow\bfS$ be the local unipotent subgroup, character, and embedding of \cite{GS21}. We recall their matrices in Section~\ref{matrices}. The image of $t$ normalizes $U$ and preserves $\psi_U$, so $J_s:=J_{U,\psi_U}(\rho_s)$ is a representation of $G\times H$.

For a smooth representation $\sigma$ of $G$, give $J_s\otimes\sigma$ the diagonal $G$-action. The action of $H$ descends to the coinvariants
\[
  \mathbf I_s(\tau,\sigma):=(J_s\otimes\sigma)_G.
\]
This is the largest quotient of $J_s\otimes\sigma$ on which $G$ acts trivially. If $\sigma$ is irreducible and admissible, the maximal $\sigma^\vee$-isotypic quotient of $J_s$ is $\sigma^\vee\boxtimes\mathbf I_s(\tau,\sigma)$; see Lemma~\ref{homlemma}. Thus, for irreducible $\sigma$, the irreducible representations $\pi$ of $H$ for which $\sigma^\vee\boxtimes\pi$ is a quotient of $J_s$ are exactly the irreducible quotients of $\mathbf I_s(\tau,\sigma)$. Let $\iota:=\iota_0$ be the conjugation automorphism of $G$ defined in (\ref{iota}), and put $\sigma^\iota(g):=\sigma(g^\iota)$.

\begin{maintheorem}\label{main}
Let $\sigma$ be an admissible representation of $G$ of finite length. There is a finite set $\mathcal B(\tau,\sigma)\subset\bbC^\times$ such that, for $q_F^{-s}\notin\mathcal B(\tau,\sigma)$, there is a natural $H$-equivariant isomorphism
\begin{equation}\label{induced}
  \mathbf I_s(\tau,\sigma)\simeq \Ind_{Q_{ni}}^H \bigl(\Delta(\tau,i)|\det|^s\boxtimes\sigma^\iota\bigr).
\end{equation}
The map from the right hand side to the left hand side is induced by inclusion of the open orbit term.
\end{maintheorem}

For every smooth representation $\sigma$ of $G$, inclusion of the open orbit term and the integral (\ref{int1}) define a natural $H$-equivariant map
\[
  \Ind_{Q_{ni}}^H \bigl(\Delta(\tau,i)|\det|^s\boxtimes\sigma^\iota\bigr) \longrightarrow \mathbf I_s(\tau,\sigma).
\]
For each fixed admissible representation $\sigma$ of finite length, Theorem~\ref{main} shows that this map is an isomorphism whenever $q_F^{-s}\notin\mathcal B(\tau,\sigma)$.

\paragraph{The case $i=0$.}
Here $H=G$, and $J_s$ is the module of the generalized doubling method studied in \cite{GK23}, with $k=n$ and $c=m$. More precisely, for even $m$ the matrices of Section~\ref{matrices} are those of \cite[Section~1.5]{GK23}, and $J_s$ is the module of \cite[(2.1)]{GK23} at the parameter $s+1/2$, with $\psi$ replaced by $\psi^{-1}$. For odd $m$, the two characters are conjugate by a similitude of ratio $-1$ of the form defining $\bfS$, which fixes $\rho_s$, followed by an element of the Levi subgroup of the parabolic subgroup normalizing $U$. The open orbit term of $J_s$ is the regular representation of $G\times G$ on $C_c^\infty(G)$, with the second factor acting through $\iota$. Outside $\mathcal B(\tau,\sigma)$, all other orbit terms have zero $G$-coinvariants after tensoring with $\sigma$. Hence Theorem~\ref{main} gives $\mathbf I_s(\tau,\sigma)\simeq\sigma^\iota$. For irreducible $\sigma$ and $\pi$, it follows that
\[
  \dim\Hom_{G\times H}(J_s,\sigma^\vee\boxtimes\pi)
  =\begin{cases}
    1,&\pi\simeq\sigma^\iota,\\
    0,&\pi\not\simeq\sigma^\iota.
  \end{cases}
\]
For irreducible representations, this is the multiplicity one theorem of Gourevitch and Kaplan \cite[Theorem~2.1 and Corollary~2.3]{GK23}. For $\sigma$ of finite length, they bound $\dim\Hom_{G\times H}(J_s,\sigma^\vee\boxtimes\pi)$ by $\dim\Hom_G(\sigma^\iota,\pi)$ outside a discrete set of $s$ that depends on $\sigma$ and $\pi$. For the groups considered here, Theorem~\ref{main} gives equality for every smooth $\pi$. It identifies $(J_s\otimes\sigma)_G$ outside an explicit finite set depending only on $\tau$ and $\sigma$.

\paragraph{The case $i\ge1$.}
Now $\mathbf I_s(\tau,\sigma)$ is infinite dimensional, and Theorem~\ref{main} determines it as an $H$-module. Let $\sigma$ be admissible of finite length and let $\Pi$ be any smooth $H$-module. If $q_F^{-s}\notin\mathcal B(\tau,\sigma)$, then
\begin{equation}\label{hom1}
  \begin{aligned}
    &\Hom_{G\times H}(J_s,\sigma^\vee\boxtimes\Pi)\\
    &\qquad\simeq\Hom_H\!\left( \Ind_{Q_{ni}}^H \bigl(\Delta(\tau,i)|\det|^s\boxtimes\sigma^\iota\bigr),\Pi\right),
  \end{aligned}
\end{equation}
naturally in $\Pi$, by Lemma~\ref{homlemma}. In particular, for irreducible admissible $\sigma$ and $\pi$, the representation $\sigma^\vee\boxtimes\pi$ is a quotient of $J_s$ if and only if $\pi$ is a quotient of the induced representation in (\ref{induced}).

Suppose that $G$ is symplectic and $i=1$, that $\tau$ and $\sigma$ are irreducible and tempered, and that $s>0$ is real and satisfies the parameter condition. The induced representation in (\ref{induced}) is then a standard module. It has a unique irreducible quotient by the Langlands classification \cite[Theorem~1.1]{BJ13}. By (\ref{hom1}), this Langlands quotient is the unique irreducible representation $\pi$ of $H$ with $\Hom_{G\times H}(J_s,\sigma^\vee\boxtimes\pi)\ne0$. The $\Hom$ space is one dimensional by \cite[Remark~4.5]{BJ13}. Distinct $\sigma$ give nonisomorphic quotients. Thus $J_s$ realizes an injective map that sends $\sigma$ to the Langlands quotient of the induced representation in (\ref{induced}).

The theorem also determines geometric wavefront sets. The wavefront set of a parabolically induced representation is obtained by inducing nilpotent orbits from the Levi subgroup; see \cite[(2.1)]{CK25}. For $i\ge1$, $\Delta(\tau,i)$ has the rectangular wavefront orbit $(n^i)$ by \cite[Theorem~4]{CFGK23}. Under the parameter condition, the geometric wavefront set of $\mathbf I_s(\tau,\sigma)$ is therefore determined by $(n^i)$ and the geometric wavefront set of $\sigma^\iota$. The resulting geometric wavefront set bounds those of all irreducible constituents of $\mathbf I_s(\tau,\sigma)$.

Theorem~\ref{main} identifies the local representation underlying the right hand side of (\ref{GS}). At a non-archimedean place $v$, this representation is the largest quotient of $J_s\otimes\sigma_v$ on which $G$ acts trivially, provided that $\sigma_v$ satisfies the hypotheses and parameter condition of the theorem.

The set $\mathcal B(\tau,\sigma)$ is given explicitly in (\ref{badset}). A parameter is exceptional only if a central character of a Jacquet module of $\sigma$ along a maximal parabolic subgroup of $G$ matches one of finitely many characters attached to $\Delta(\tau,m+i)$. Remark~\ref{exception} gives an exceptional parameter at which (\ref{induced}) fails. If $\sigma$ is irreducible and supercuspidal and $G\ne\SO_2$, the exceptional set is empty. In this case, (\ref{induced}) holds for every $s$ (Corollary~\ref{cusp}). If $\tau$ is supercuspidal and $n>\lfloor m/2\rfloor$, with $n\ge2$, then (\ref{induced}) holds for every $s$ and every smooth $\sigma$ (Corollary~\ref{taucusp}). Theorem~\ref{metmain} proves the identity for genuine representations of the symplectic double cover.

\paragraph{Comparison with the theta correspondence.}
Ginzburg and Soudry compare (\ref{GS}) to towers of theta lifts. The local construction has the same quotient structure: one starts with a representation of $G\times H$ and takes its maximal $\sigma^\vee$-isotypic quotient. In the theta correspondence, this representation is the Weil representation, and the resulting $H$-module is the big theta lift. Here the representation is $J_s$. Theorem~\ref{main} identifies the resulting module $\mathbf I_s(\tau,\sigma)$ with an explicit induced representation for all but finitely many values of $q_F^{-s}$.

\subsection{Correspondence over finite fields}

Let $k$ be a finite field of odd characteristic $p$. Keep the family $\mathcal H_r$, and assume $m\ge2$ and $n>\lfloor m/2\rfloor$. Put $G_k:=\mathcal H_m(k)$ and $H_{i,k}:=\mathcal H_{2ni+m}(k)$. Let $\tau_k$ be an irreducible cuspidal representation of $\GL_n(k)$. Let $\Delta(\tau_k,r)$ be the finite Speh representation defined in Section~\ref{Fq}, with $\Delta(\tau_k,0):=\one$. Fix a nontrivial additive character of $k$. The same matrices define $U$, $\psi_U$, and $t$ over $k$. Let $Q_N$ denote the Siegel parabolic subgroup over $k$. Induction in this subsection is ordinary induction of finite groups. Define
\[
J_i:=J_{U,\psi_U}\!\left(\Ind_{Q_N}^{\mathcal H_{2N}(k)}\Delta(\tau_k,m+i)\right).
\]
Let $Q_{ni}^{H_{i,k}}$ be the standard parabolic subgroup with Levi subgroup $\GL_{ni}(k)\times G_k$. The automorphism $\iota$ is defined by the same matrices over $k$. For a representation $\sigma$ of $G_k$, put
\[
I_i(\sigma):=\Ind_{Q_{ni}^{H_{i,k}}}^{H_{i,k}}\bigl(\Delta(\tau_k,i)\boxtimes\sigma^\iota\bigr).
\]
Let $\operatorname{Irr}(G_k)$ denote the set of isomorphism classes of irreducible complex representations of $G_k$.

\begin{maintheorem}\label{finite}
If $p>2n(m+i)$, then
\[
J_i\simeq\bigoplus_{\sigma\in\operatorname{Irr}(G_k)} \sigma^\vee\boxtimes I_i(\sigma).
\]
If $\tau_k\not\simeq\tau_k^\vee$, each $I_i(\sigma)$ is irreducible. If $\tau_k\simeq\tau_k^\vee$, each $I_i(\sigma)$ is a direct sum of exactly $i+1$ pairwise nonisomorphic irreducible representations. Distinct $\sigma$ give representations $I_i(\sigma)$ with no common irreducible constituent.
\end{maintheorem}

Thus $J_i$ is multiplicity free as a representation of $G_k\times H_{i,k}$. Its constituents define a correspondence between $\operatorname{Irr}(G_k)$ and a subset of $\operatorname{Irr}(H_{i,k})$. A representation $\sigma$ corresponds to one irreducible representation of $H_{i,k}$ if $\tau_k$ is not selfdual, and to exactly $i+1$ if it is selfdual. Each of these irreducible representations determines $\sigma$ through an ordinary Jacquet module. Let $U_Q$ be the unipotent radical of $Q_{ni}^{H_{i,k}}$. If $\pi$ is an irreducible constituent of $I_i(\sigma)$, then
\[
  \Hom_{\GL_{ni}(k)} \bigl(\Delta(\tau_k,i),J_{U_Q,1}(\pi)\bigr) \simeq\sigma^\iota
\]
as $G_k$-modules. By Frobenius reciprocity, the multiplicity of $(\sigma')^\iota$ on the left equals the multiplicity of $\pi$ in $I_i(\sigma')$. Theorem~\ref{finite} gives multiplicity one for $\sigma'\simeq\sigma$ and zero otherwise.

As in Corollary~\ref{taucusp}, the condition $n>\lfloor m/2\rfloor$ makes every orbit other than the open one contribute zero. This vanishing holds before taking coinvariants with a representation of $G_k$. The decomposition of $J_i$ is proved in Theorem~\ref{finitecase}. The assertions about $I_i(\sigma)$ are proved in Proposition~\ref{constituents}, which requires only odd characteristic.

\subsection{The orbit contributions}

The proof of Theorem~\ref{main} starts from the double coset decomposition used by Ginzburg and Soudry. Global cuspidality eliminates all terms other than the open orbit. Locally, the other terms need not vanish. We compute them and prove that their $G$-coinvariants after tensoring with $\sigma$ vanish outside $\mathcal B(\tau,\sigma)$. The geometric lemma of Bernstein and Zelevinsky \cite{BZ77}, applied with the double coset representatives of \cite{GS21}, gives a filtration of $J_s$ by $G\times H$-modules. The matrix calculations of \cite{GS21}, together with the wavefront bound on $\Delta(\tau,m+i)$ from \cite{CFGK23}, eliminate most double cosets. The remaining orbits are indexed by an integer $e$ with $0\le e\le d:=\lfloor m/2\rfloor$, with two signed terms at $e=d$ when $G=\SO_{2d}$. For $0\leq e\leq d$, we write $P_e:=Q_e^G$, $R_e:=Q_{ni+e}^H$ and $G_e:=\mathcal H_{m-2e}(F)$. Let $U_e^G$ and $U_{R_e}$ be the unipotent radicals of $P_e$ and $R_e$, respectively. When $G=\SO_{2d}$, the signs of $Q_d^{G,\varepsilon}$ and $Q_{ni+d}^{H,\varepsilon}$ are matched as in Proposition~\ref{evenSO}. Except when $G=\SO_{2d}$ and $e=d$, the stabilizer of the orbit indexed by $e$ is
\begin{equation}\label{Se}
S_e:=(U_e^G\times U_{R_e})\rtimes (\GL_e\times\GL_{ni+e}\times G_e).
\end{equation}
The two projections of $G_e$ lie in the classical factors of the Levi subgroups of $P_e$ and $R_e$; see Proposition~\ref{stab}. The stabilizer acts, through a projection $\ell_e$ to the Siegel Levi subgroup, on the twisted Jacquet module $J_{V_{ni+m,m^{n-1}},\psi_e}(\Delta(\tau,m+i))$ defined in (\ref{Ce}). Put
\begin{equation}\label{ab}
  \alpha_e:=\frac{ni-m(n-1)+e}{2},\qquad \beta_e:=\frac{e-m(n-1)}{2}.
\end{equation}
Let $\calE_{e,s}$ denote this twisted Jacquet module with the following action of $\GL_e\times\GL_{ni+e}\times G_e$:
\begin{equation}\label{Eaction}
  \begin{aligned}
    \calE_{e,s}(a,A,b) &:={}|\det a|^{s+\alpha_e}|\det A|^{s+\beta_e}\\
    &\qquad\cdot \bigl[J_{V_{ni+m,m^{n-1}},\psi_e}(\Delta(\tau,m+i))\bigr] \bigl(\ell_e(a,A,b)\bigr).
  \end{aligned}
\end{equation}

\begin{maintheorem}\label{layers}
Let $W_{e,s}$ be the representation of $S_e$ defined in (\ref{rhoe}). The $G\times H$-module $J_s$ has a finite filtration whose first term corresponds to the open orbit. Its successive quotients are
\[
J_{e,s}\simeq\ind_{S_e}^{G\times H}W_{e,s},
\]
with $e=0,\ldots,d$ when $G\ne\SO_{2d}$; when $G=\SO_{2d}$, two signed indices replace $e=d$. For $e>0$, the unipotent radical $U_e^G$ acts trivially on $W_{e,s}$. For every smooth $G$-module $\sigma$ and every $s$, there are $H$-equivariant isomorphisms
\begin{align}
  (J_{0,s}\otimes\sigma)_G &\simeq \Ind_{Q_{ni}}^H \bigl(\Delta(\tau,i)|\det|^s\boxtimes\sigma^\iota\bigr), \notag\\
  (J_{e,s}\otimes\sigma)_G &\simeq \Ind_{R_e}^H \bigl(\calE_{e,s}\otimes r_{U_e^G}(\sigma)\bigr)_{\GL_e}, \qquad e>0. \label{bterm}
\end{align}
If $G=\SO_{2d}$, the two endpoint stabilizers are $S_{d,\varepsilon}:=Q_d^{G,\varepsilon}\times Q_{ni+d}^{H,\varepsilon}$, for $\varepsilon\in\{+,-\}$, and (\ref{bterm}) holds for each signed term.
\end{maintheorem}

Theorem~\ref{layers} follows from Theorem~\ref{geom}, Proposition~\ref{integral} and Theorem~\ref{boundary}. Theorem~\ref{boundary} also specifies the action of $G_e$ on the coefficient in (\ref{bterm}). For $e>0$, formula (\ref{bterm}) expresses the orbit contribution in terms of the ordinary Jacquet module of $\sigma$ along $P_e$ and a twisted Jacquet module of the Speh representation. If $P_e$ is proper and $\sigma$ is supercuspidal, the ordinary Jacquet module vanishes. This gives Corollary~\ref{cusp}.

Theorem~\ref{triviality} proves that $U_e^G$ acts trivially on $W_{e,s}$. This makes the $G$-coinvariants of the term indexed by $e$ factor through $r_{U_e^G}(\sigma)$. The corresponding global calculation is where cuspidality of $\varphi_\sigma$ is used. The proof of Theorem~\ref{triviality} combines root exchange \cite{GRS11,GS22} with the vanishing criterion for degenerate Whittaker models of \cite{GGS17} and the wavefront bound on $\Delta(\tau,m+i)$. The unipotent groups that occur are not unipotent radicals of parabolic subgroups. We therefore check the hypotheses of the criterion by constructing explicit gradings and computing the radicals of the associated symplectic forms (Lemma~\ref{D2D3}).

To deduce Theorem~\ref{main}, we determine how the center of $\GL_e$ acts on $\calE_{e,s}$. Further root exchanges reduce $\calE_{e,s}$ to an ordinary Jacquet module of $\Delta(\tau,m+i)$ followed by rectangular twisted Jacquet modules (Theorem~\ref{coarse}). The rectangular quotient of each constituent is one dimensional. We prove this by comparing a neutral and a degenerate model for the same rectangular orbit (Lemma~\ref{derivative}). On the open orbit, the same calculation produces $\Delta(\tau,i)$ (Proposition~\ref{openSpeh}). Hence the center of $\GL_e$ acts on $\calE_{e,s}$ through finitely many explicit characters (Proposition~\ref{spectrum}). Comparing them with the central characters of the Jacquet modules of $\sigma$, as in \cite[Lemmas~2.15--2.16]{GK23}, shows that the boundary terms have no $G$-coinvariants outside $\mathcal B(\tau,\sigma)$. This only shows that the open orbit term maps onto $\mathbf I_s(\tau,\sigma)$. We obtain injectivity by applying the same argument to the first smooth group homology of the boundary terms (Lemma~\ref{exactness}).

\subsection{Further directions}

Filtrations of the same kind as Theorem~\ref{layers} occur in the doubling method \cite{PSR87,HKS96,C21}. For unramified data, Ginzburg and Soudry compute a full local twisted Jacquet module in the setting of double descent \cite[Theorem~12.2]{GS22}. Questions about the irreducible quotients of twisted Jacquet modules also arise in local descent; for instance, Jiang and Zhang \cite{JZ18} determine the spectrum of local descents at the first occurrence for irreducible representations of $p$-adic special orthogonal groups with generic Langlands parameters.

For depth zero inducing data, one may ask how Theorems~\ref{main} and~\ref{finite} are related. Liu, Ma, and Shi study local descent for unitary groups using finite Fourier--Jacobi models and depth zero types \cite{LMS25}. For the modules considered here, the question is whether parahoric restriction relates the two theorems.

In the symplectic case with $\tau$ on $\GL_2$, another question concerns the level with no shorter Speh factor. At this level, one may replace the ambient induced representation by a residual quotient. Related residual Fourier coefficients appear in the Siegel--Weil type identities proposed in \cite[Section~6.2]{GS24}. This question requires understanding how Theorem~\ref{main} interacts with the intertwining operators that define the residual quotients.

\paragraph{Organization.}
Section~\ref{orbits} constructs the orbit filtration, determines the stabilizers, and proves Theorem~\ref{triviality}. Section~\ref{terms} computes the contribution of each orbit and proves Theorem~\ref{layers}. Section~\ref{generic} proves Theorem~\ref{main} and the identity for the symplectic double cover. Section~\ref{Fq} proves Theorem~\ref{finite} and describes the constituents. The finite field vanishing criterion is proved in Appendix~\ref{comparison}.

\paragraph{Acknowledgments.}
The author thanks his advisors, Professor Solomon Friedberg and Professor Omer Offen, for suggesting this problem and for their help and guidance. This work is part of the author's Ph.D. thesis.

\subsection{Notation and conventions}\label{notation}

In even orthogonal type, we fix one of the two special orthogonal orbits of maximal isotropic subspaces. We distinguish the two endpoint parabolics below. For $L=\SO_{2a}(F)$, let $Q_a^{L,+}$ and $Q_a^{L,-}$ be the stabilizers of the two standard maximal isotropic subspaces stabilized by the fixed Borel subgroup. These parabolics are distinct for $a\geq2$. For $a=1$, both are $L$. At a signed endpoint, we identify the general linear Levi factor by its action on the corresponding maximal isotropic subspace. The group $Q_N$ is the Siegel parabolic stabilizing the chosen standard maximal isotropic subspace. Its Levi factor is $\GL_N(F)$. For $i=0$, we have $Q_{ni}=H$.

All representations are complex and smooth. We use the conventions for smooth representations in Section~1 of \cite{BZ77}. We omit $(F)$ from general linear groups when the field is clear. We denote by $\Mat_{a\times b}(F)$ the space of $a\times b$ matrices over $F$. We write $\Mat_a(F):=\Mat_{a\times a}(F)$. The symbols $I_a$ and ${}^tg$ denote the identity matrix of size $a$ and the transpose of $g$, respectively. We write $0_{a\times b}$ for the $a\times b$ zero matrix. Put $0_a:=0_{a\times a}$. A plain $0$ in a block matrix has the size determined by its row and column blocks.

We denote by $\pi^\vee$ the smooth contragredient of a representation $\pi$. The symbol $\one_L$ denotes the trivial representation of a group $L$; we omit the subscript when the group is clear. We use $\boxtimes$ for the external tensor product of representations of different group factors, and $\otimes$ with the diagonal action when both factors are representations of the same group. For a character $\chi$, the twist $\pi\chi$ has action $(\pi\chi)(g)=\chi(g)\pi(g)$.

For a parabolic subgroup $P$, we denote its unipotent radical by $U_P$ and its modulus character by $\delta_P$, with convention $\delta_P(m)=|\det(\operatorname{Ad}(m)|_{\Lie U_P})|$ on its Levi subgroup.

For an ordered block decomposition with dimensions $a_1,\ldots,a_t$, let $V_{a_1,\ldots,a_t}$ be the group of upper block unipotent matrices. We omit blocks of dimension zero. In a subscript, $a^t$ means $t$ consecutive blocks of size $a$. An empty sum is zero, an empty product is one, and a matrix of size zero has determinant one and trace zero.

For $P=MU_P\subset L$ and a representation $\pi$ of $M$, we inflate $\pi$ to $P$ by letting $U_P$ act trivially. Our normalized induction $\Ind_P^L\pi$ has right translation action and consists of smooth functions satisfying $f(pg)=\delta_P(p)^{1/2}\pi(p)f(g)$ for $p\in P$ and $g\in L$. The symbol $\ind_K^L\rho$ denotes unnormalized compact induction: its functions satisfy $f(kg)=\rho(k)f(g)$ and have compact support modulo the closed subgroup $K$, again with right translation action.

As in Section~1.2 of \cite{GK23}, $J_{V,\chi}$ denotes the unnormalized twisted Jacquet functor. For a unipotent subgroup $V\subset L$, a character $\chi$ of $V$, and a smooth representation $\pi$ of $L$ on $E$, our convention is $J_{V,\chi}(\pi):=E/\langle\pi(v)\xi-\chi(v)\xi:v\in V,\ \xi\in E\rangle$. The stabilizer of $\chi$ in the normalizer of $V$ acts on this quotient. We abbreviate $J_{V,1}$ to $J_V$. The normalized ordinary Jacquet module for $P=MU_P$ is $r_{U_P}(\pi):=\delta_P^{-1/2}\otimes J_{U_P,1}(\pi)$ as an $M$-module. For a smooth representation $V$ of a locally profinite group $L$, define its space of coinvariants by $V_L:=V/\langle\ell v-v:\ell\in L,\ v\in V\rangle$.

We use the self-dual measure for $\psi$ on each root coordinate. We choose compatible quotient measures.

\section{The orbit filtration}\label{orbits}

\subsection{The unipotent subgroup and its character}\label{matrices}

We recall the matrices in (1.20)--(1.27) of \cite{GS21} over $F$. Let $w_r$ be the antidiagonal permutation matrix of size $r$. For a nonsingular symmetric matrix $A$, we write $\SO(A):=\{g\in\SL_r(F):{}^tgAg=A\}$, where $r$ is the size of $A$. The split source orthogonal groups and all even orthogonal groups are realized by the forms $w_r$. The symplectic group $\Sp_{2a}(F)$ preserves the alternating form $\bigl(\begin{smallmatrix}0&w_a\\-w_a&0\end{smallmatrix}\bigr)$. For the odd-dimensional target orthogonal group, we use the form $w'_{2k+1}$ specified below. For $g\in\GL_r(F)$, we write $g^*:=w_r\,{}^tg^{-1}w_r$.

For $n\ge2$, let $U$ be the unipotent radical of the standard parabolic of $\bfS$ with Levi subgroup $\GL_m^{\,n-1}\times\mathcal H_{2(ni+m)}$. Let $d:=\lfloor m/2\rfloor$ and $b:=2(ni+\lceil m/2\rceil)$.

With block sizes $m,\ldots,m,d,b,d,m,\ldots,m$, an element of $U$ has the form
\[
  \setlength{\arraycolsep}{4pt}
u=\begin{pmatrix}
I_m&x_1&\star&\star&\star&\star&\star&\star&\star\\
    &\ddots&x_{n-2}&\star&\star&\star&\star&\star&\star\\
    &&I_m&y_1&y_2&y_3&\star&\star&\star\\
    &&&I_d&0&0&y'_3&\star&\star\\
    &&&&I_b&0&y'_2&\star&\star\\
    &&&&&I_d&y'_1&\star&\star\\
    &&&&&&I_m&x'_{n-2}&\star\\
    &&&&&&&\ddots&x'_1\\
    &&&&&&&&I_m
  \end{pmatrix}\in\bfS.
\]
There are $n-1$ diagonal blocks of size $m$ on each side of the central block. Blank entries are zero. The primed entries are determined by preservation of the defining form. The symbols $\star$ denote the remaining entries allowed by that condition. For $n=2$, there is one block of size $m$ on each side. No $x_j$ occurs in this case. The blocks $y_1,y_3$ belong to $\Mat_{m\times d}(F)$. The block $y_2$ belongs to $\Mat_{m\times b}(F)$.

We define $A_{\bfS}\in\Mat_{2(ni+m)\times m}(F)$ as follows. For $m=2d$, we put
\[
A_{\bfS}:=\begin{pmatrix}
I_d&0\\
    0_{(2ni+m)\times d}&0_{(2ni+m)\times d}\\
    0&I_d
  \end{pmatrix}.
\]
For $m=2d+1$, let $k:=ni+d$. With row sizes $d,k,1,1,k,d$ and column sizes $d,1,d$, we put
\[
A_{\bfS}:=\begin{pmatrix}
I_d&0&0\\
    0&0&0\\
    0&1&0\\
    0&\tfrac12&0\\
    0&0&0\\
    0&0&I_d
  \end{pmatrix}.
\]
The character of $U$ is defined by
\begin{equation}
  \psi_U(u):=\psi\left(\sum_{j=1}^{n-2}\tr x_j+ \tr\bigl((y_1,y_2,y_3)A_{\bfS}\bigr)\right). \label{psiU}
\end{equation}
For $n=1$, $U:=\{1\}$ and $\psi_U:=1$.

For a square matrix $g$, we write $g^{\Delta_r}:=\diag(g,\ldots,g)$ with $r$ copies. For $g\in G$ and $h\in H$, the matrix $j(g,h)$ is defined as follows. For $m=2d$, we write $g=\bigl(\begin{smallmatrix}a&b_1\\c&d_1\end{smallmatrix}\bigr)$ in $d\times d$ blocks.  In block sizes $d,2ni+m,d$,
\begin{equation}
j(g,h):=\begin{pmatrix}
a&0&b_1\\
    0&h&0\\
c&0&d_1
  \end{pmatrix}.
  \label{jm}
\end{equation}
For $m=2d+1$ the groups are orthogonal.  As in (1.26) of \cite{GS21}, we realize the split group $H$ using the form $w'_{2k+1}:=\bigl(\begin{smallmatrix}0&0&w_k\\0&-1&0\\w_k&0&0\end{smallmatrix}\bigr)$, where $k=ni+d$. We write
\[
g=\begin{pmatrix}a_1&b_1&c_1\\a_2&b_2&c_2\\a_3&b_3&c_3\end{pmatrix},
  \qquad
h=\begin{pmatrix}A_1&B_1&C_1\\A_2&B_2&C_2\\A_3&B_3&C_3\end{pmatrix},
\]
in block sizes $d,1,d$ and $k,1,k$, respectively. With block sizes $d,k,1,1,k,d$, we put
\[
  \setlength{\arraycolsep}{4pt}
  \renewcommand{\arraystretch}{1.2}
j(g,h):=\begin{pmatrix}
a_1&0&\tfrac12b_1&b_1&0&c_1\\
    0&A_1&\tfrac12B_1&-B_1&C_1&0\\
a_2&A_2&\tfrac12(b_2+B_2)&b_2-B_2&C_2&c_2\\
    \tfrac12a_2&-\tfrac12A_2&\tfrac14(b_2-B_2)&\tfrac12(b_2+B_2)&-\tfrac12C_2&\tfrac12c_2\\
    0&A_3&\tfrac12B_3&-B_3&C_3&0\\
a_3&0&\tfrac12b_3&b_3&0&c_3
  \end{pmatrix}.
\]
For every odd $c=2a+1$, let $w'_c:=\bigl(\begin{smallmatrix}0&0&w_a\\0&-1&0\\w_a&0&0\end{smallmatrix}\bigr)$. Put $K_c:=\diag(I_a,1,-I_a)$. The identity ${}^tK_cw'_cK_c=-w_c$ identifies $\SO(w_c)$ with $\SO(w'_c)$ by $b\mapsto K_cbK_c^{-1}$. For symplectic groups and even orthogonal groups, we take $K_c:=I_c$. These conventions identify every residual classical factor in a Levi subgroup of $H$ with the corresponding group in the source family.

Let $P\subset G$ and $R\subset H$ have residual dimension $c$. For $b\in\mathcal H_c(F)$ and general linear elements $a,A$ of the corresponding sizes, we use
\begin{equation}
m_P(a,b):=\diag(a,b,a^*),\qquad m_R(A,b):=\diag(A,K_cbK_c^{-1},A^*). \label{Levi}
\end{equation}
For odd orthogonal groups, the two embeddings use the forms $w_c$ and $w'_c$, respectively.

The embedding $t:G\times H\longrightarrow\bfS$ is
\begin{equation}
t(g,h):=\diag\bigl(g^{\Delta_{n-1}},j(g,h), (g^*)^{\Delta_{n-1}}\bigr), \qquad g\in G,\quad h\in H, \label{tgh}
\end{equation}
Blocks of dimension zero are omitted. The matrices satisfy $j(g,h)A_{\bfS}=A_{\bfS}g$. Since trace is invariant under conjugation, this identity shows that $t(g,h)$ preserves the character in (\ref{psiU}). When $n=1$, (\ref{tgh}) means $t(g,h)=j(g,h)$.

\subsection{Double cosets and stabilizers}

Let $\delta_{\Sp}:=-1$ and $\delta_{\SO}:=1$. For a standard parabolic $P$ with general linear factor $\GL_k$, let $m_P(a,h)$ denote its Levi embedding. For the source and target groups, we use (\ref{Levi}). When the ambient group is denoted by $\calH_M$, we also write $m_k(a,h)$. At a signed endpoint, $m_P(a):=m_P(a,1)$ denotes the Levi element whose action on the chosen maximal isotropic subspace is $a$. The standard parabolic $Q_k^{\calH_M}$ has modulus
\begin{equation}
  \delta_{Q_k^{\calH_M}}(m_k(a,h)) =|\det a|^{M-k-\delta_{\calH}}.                  \label{deltaQ}
\end{equation}

For a partition $\lambda$ of $N$, let $\mathcal O_\lambda$ be the corresponding nilpotent orbit in $\mathfrak{gl}_N$. For an admissible representation $\pi$ of $\GL_N(F)$ of finite length, let $\WF(\pi)$ denote its wavefront set. By Theorem~4 of \cite{CFGK23},
\begin{equation}
  \lambda_1\le n \qquad\bigl(\mathcal O_\lambda\in\WF(\Delta(\tau,m+i))\bigr).                     \label{WFn}
\end{equation}
We use the vanishing criterion for generalized Whittaker models in Theorems~A and~E of \cite{GGS17} and \cite{MW87} in the following form. For the unipotent groups and characters specified below, a nonzero twisted Jacquet module implies a nonzero generalized Whittaker model for the corresponding nilpotent matrix. The nilpotent orbit of this matrix must therefore lie in the closure of an orbit in the wavefront set. Consequently, (\ref{WFn}) forces the twisted Jacquet module to vanish whenever that matrix has a Jordan block of size greater than $n$.

For a unipotent matrix group $V$ and a character $\chi$ occurring below, we choose the trace-dual matrix $N_\chi$ with convention $\chi(\exp X)=\psi^{-1}\bigl(\tr(N_\chi X)\bigr)$ for $X\in\Lie V$. Entries outside the root blocks dual to those occurring in $\chi$ are taken to be zero. If $\chi$ is supported on adjacent blocks of a standard unipotent radical, $N_\chi$ has nonzero entries only in the corresponding blocks immediately below the diagonal. If $N_\chi^n\ne0$, its Jordan partition has first part greater than $n$. The vanishing criterion then gives a zero twisted Jacquet module. When the group of integration is not a standard unipotent radical, we verify the hypotheses of the same criterion below. We use root exchange as in Section~7.1 of \cite{GRS11} and the proof of Corollary~12.7 of \cite{GS22}.

We apply Mackey theory to the successive double coset decompositions, as in the proof of Theorem~12.4 of \cite{GS22}. We use the representatives in (2.31)--(2.43), (2.53)--(2.54) of \cite{GS21}. The resulting filtrations are given by the geometric lemma of \cite{BZ76,BZ77}. For $n\ge2$, the first index is an integer $r$ with $0\leq r\leq m(n-1)$. The second index is a tuple $\bar r:=(r_1,\ldots,r_{n-1})$ with $0\leq r_j\leq m$ and $\sum r_j=r$. Put $\ell_j:=m-r_j$. Let $P_{a,b}$ denote the standard parabolic of $\GL_{a+b}$ with Levi subgroup $\GL_a\times\GL_b$. Choose representatives $\eta_j\in\GL_m(F)$ in $P_{r_j,\ell_j}\backslash\GL_m$. Let $\bar\eta:=(\eta_1,\ldots,\eta_{n-1})$.

The remaining orbits are those of $G\times H$ on the maximal isotropic subspaces in the chosen orbit under $\mathcal H_{2(ni+m)}$. They are indexed by $0\leq e\leq d=\lfloor m/2\rfloor$. Here $e$ is the dimension of the intersection with the $m$-dimensional space on which $G$ acts. Denote this intersection by $E$. The quotient $E^\perp/E$ in the space for $G$ has dimension $m-2e$. The split even orthogonal endpoint is described separately below. For $n=1$, only these orbits occur.

Assume $n\ge2$ for the transition matrices below. For $1\le j<n-1$, the rows of $\eta_{j+1}\eta_j^{-1}$ are divided into sizes $r_{j+1},\ell_{j+1}$ and its columns into sizes $r_j,\ell_j$. The rows of $\eta_{n-1}^{-1}$ are divided into sizes $e,m-e$ and its columns into sizes $r_{n-1},\ell_{n-1}$. Write $\eta_{j+1}\eta_j^{-1}=\bigl(\begin{smallmatrix}*&*\\\kappa_j&*\end{smallmatrix}\bigr)$ and $\eta_{n-1}^{-1}=\bigl(\begin{smallmatrix}*&*\\\kappa_\infty&*\end{smallmatrix}\bigr)$. Here $\kappa_j\in\Mat_{\ell_{j+1}\times r_j}(F)$ and $\kappa_\infty\in\Mat_{(m-e)\times r_{n-1}}(F)$.

\begin{lemma}\label{intersection}
Assume $n\ge2$. The representative $\bar\eta$ does not contribute to the twisted Jacquet module unless
\begin{equation}
  \kappa_1=\cdots=\kappa_{n-2}=\kappa_\infty=0,\quad r_1\leq\cdots\leq r_{n-1}\leq e.                    \label{ij}
\end{equation}
In particular, the term indexed by $(\bar r,e)$ is zero if the inequalities in (\ref{ij}) fail.
\end{lemma}

\begin{proof}
This follows from the matrix calculation in (2.31)--(2.43), in the proof of Proposition~2.3 of \cite{GS21}, together with the injectivity argument in (12.41) of \cite{GS22}.
\end{proof}

We use the injective map on twisted Jacquet quotients in (12.41) of \cite{GS22}, from the proof of Theorem~12.4. The same construction applies here with middle block of size $ni+m$ instead of $m$. It follows that a compactly induced term is zero if every projected Jacquet quotient of its inducing representation is zero.

\begin{proposition}\label{rzero}
Under (\ref{WFn}), only the double cosets with $r=0$ contribute to $J_s$ when $n\ge2$.
\end{proposition}

\begin{proof}
The result follows from the matrix calculation in (2.53)--(2.54), in the proof of Proposition~2.4 of \cite{GS21}, together with the vanishing criterion following (\ref{WFn}) and the injective local quotient map recalled above.
\end{proof}

For $n\ge2$, we henceforth take $r=0$. Let $\alpha^{(0)}$ and $\gamma_e$ be the matrices in (3.5) and (2.25) of \cite{GS21}, respectively. Let $h_{\gamma_e}:=\diag(I_{m(n-1)},\gamma_e,I_{m(n-1)})$. Put $w_e^{(0)}:=\alpha^{(0)}h_{\gamma_e}$. For $n=1$, the outer blocks of $w_e^{(0)}$ are omitted.

Let $c_e:=m-2e$, so that $G_e=\calH_{c_e}(F)$. Let $W$ and $Y$ be the defining spaces of $G$ and $H$, respectively. Let $\epsilon:=m-2d$. We define
\[
  \beta_e^{(0)}:=
  \begin{pmatrix}
    0_{(d-e)\times(ni+e)}&I_{d-e}\\
I_{ni+e}&0_{(ni+e)\times(d-e)}
  \end{pmatrix},\qquad
  \widehat\beta_e:= \diag\bigl(\beta_e^{(0)},I_\epsilon, (\beta_e^{(0)})^*\bigr)\in H.
\]
The matrix $\widehat\beta_e$ is the one in (3.11) of \cite{GS21}. It lies in the Levi subgroup of the maximal isotropic parabolic of $H$.

Let $w_{\mathrm{GL}}\in\GL_N(F)$ be the matrix in (3.14)--(3.15) of \cite{GS21}. For even $m$, it is
\[
w_{\mathrm{GL}}:=\diag\left(
  \begin{pmatrix}0&I_{ni}\\I_m&0\end{pmatrix},I_{m(n-1)}\right),
\]
where the first matrix has row sizes $ni,m$ and column sizes $m,ni$. For odd $m=2d+1$, it is
\[
w_{\mathrm{GL}}:=\diag\left(
  \begin{pmatrix}
    0&0&I_{ni}&0\\
I_d&0&0&0\\
    0&0&0&\tfrac12\\
    0&I_d&0&0
  \end{pmatrix},I_{m(n-1)}\right),
\]
where the first matrix has row sizes $ni,d,1,d$ and column sizes $d,d,ni,1$. Let $\widehat w_{\mathrm{GL}}:=\diag(w_{\mathrm{GL}},w_{\mathrm{GL}}^*)\in Q_N$. We use the representative $w_e:=\widehat w_{\mathrm{GL}}\,w_e^{(0)}\,t(1,\widehat\beta_e)$. Right multiplication by $t(1,\widehat\beta_e)$ places the projection of the stabilizer to $H$ in the standard parabolic $R_e$. Left multiplication by $\widehat w_{\mathrm{GL}}$ gives the coordinates in (3.16)--(3.17) of \cite{GS21}. All stabilizer actions below use this representative.

The central maximal isotropic subspace represented by $w_e$ has three summands. The first is an $e$-dimensional isotropic subspace of the space for $G$. The second is an $(ni+e)$-dimensional isotropic subspace of the space for $H$. The third projects isomorphically to each of the two remaining nondegenerate spaces of dimension $c_e$. Let $\zeta:=(-1)^{n-1}$. The matrices in (3.9)--(3.10) of \cite{GS21} are
\[
J_e:=
  \begin{cases}
    \begin{pmatrix}0&I_{d-e}\\I_{d-e}&0\end{pmatrix},
    &G=\Sp_{2d},\\[6pt]
    \begin{pmatrix}0&I_{d-e}\\-I_{d-e}&0\end{pmatrix},
    &G=\SO_{2d},\\[6pt]
    \begin{pmatrix}0&0&\zeta I_{d-e}\\0&1&0\\-\zeta I_{d-e}&0&0\end{pmatrix},
    &G=\SO_{2d+1}.
  \end{cases}
\]
The third summand is the graph of a scalar multiple of $J_e$ from the residual source space to the residual target space. This scalar does not affect the stabilizer relation. We define $\iota_e$ on $G_e$ by
\begin{equation}
b^{\iota_e}:=C_{c_e}bC_{c_e}^{-1},\qquad C_{c_e}:=J_e^{-1}K_{c_e}. \label{iota}
\end{equation}
The matrix $J_e$ identifies the two residual forms up to a nonzero scalar. Together with the identity for $K_{c_e}$, this shows that $C_{c_e}$ normalizes the fixed matrix group $G_e$. For odd $c_e=2(d-e)+1$, the formula is explicitly
\[
C_{c_e}=
  \begin{pmatrix}0&0&\zeta I_{d-e}\\0&1&0\\\zeta I_{d-e}&0&0\end{pmatrix},
  \qquad {}^tC_{c_e}w_{c_e}C_{c_e}=w_{c_e}.
\]
We write $\iota:=\iota_0$. The parameter $b$ always belongs to the source model $G_e$. The target residual block is $K_{c_e}bK_{c_e}^{-1}$. Let $\operatorname{pr}_N:Q_N\to\GL_N(F)$ be the Siegel Levi projection. For an element $s=(g,h)$ of this stabilizer, let $\ell_e(s):=\operatorname{pr}_N\bigl(w_et(g,h)w_e^{-1}\bigr)$.

\begin{proposition}\label{stab}
Except when $G=\SO_{2d}$ and $e=d$, the stabilizer at $w_e$ is
\[
S_e=(U_e^G\times U_{R_e})\rtimes (\GL_e\times\GL_{ni+e}\times G_e).
\]
For $a\in\GL_e(F)$, $A\in\GL_{ni+e}(F)$, and $b\in G_e$, the stabilizer elements have the form
\[
  \left(
  \begin{pmatrix}a&*&*\\0&b^{\iota_e}&*\\0&0&a^*\end{pmatrix},
  \begin{pmatrix}A&*&*\\0&K_{c_e}bK_{c_e}^{-1}&*\\0&0&A^*\end{pmatrix}
  \right).
\]
Setting the stars to zero gives the Levi element attached to $(a,A,b)$. Let $\ell_e(a,A,b)$ denote its image under $\ell_e$. Then
\begin{equation}
  \det\ell_e(a,A,b)=\det a\det A.                       \label{detell}
\end{equation}
Thus the source residual coordinate is $b^{\iota_e}$ when the target residual coordinate is parametrized by $b$ as in (\ref{Levi}).
\end{proposition}

\begin{proof}
The stabilizer calculation is (3.7)--(3.11) of \cite{GS21}. Right translation by $t(1,\widehat\beta_e)$ places its target projection in $R_e$. If $g_0$ and $h_0$ are the residual classical blocks, that calculation gives $h_0J_e=J_eg_0$. In the coordinates of (\ref{Levi}), we have $h_0=K_{c_e}bK_{c_e}^{-1}$. Hence $g_0=b^{\iota_e}$.

On the central maximal isotropic space, the three diagonal blocks have determinants $\det a$, $\det A$, and $\det g_0=1$. Each of the other $n-1$ blocks carries a defining $G$-action, whose determinant is one. Thus $\det\ell_e(a,A,b)=\det a\det A$. Left multiplication by $\widehat w_{\mathrm{GL}}$ conjugates the projected action and does not change its determinant.
\end{proof}

Assume now that $G=\SO_{2d}$. Let $d_H:=ni+d$. In a split space of dimension $2a$, fix a maximal isotropic subspace $L^+$ in the chosen special orthogonal orbit. For a maximal isotropic subspace $L$, we define $\varepsilon(L):=\operatorname{codim}_L(L\cap L^+)\pmod2$. The signs $+$ and $-$ mean $\varepsilon(L)=0$ and $1$, respectively. Choose the reference spaces $L_W^+\subset W$ and $L_Y^+\subset Y$ so that $L_W^+\oplus L_Y^+$ belongs to the chosen orbit under the ambient special orthogonal group.

\begin{proposition}\label{evenSO}
Among the maximal isotropic subspaces in the chosen ambient orbit, those with $e=d$ form two $\SO_{2d}\times\SO_{2d_H}$-orbits.  Their stabilizers are
\begin{equation}
S_{d,+}=Q_d^{G,+}\times Q_{d_H}^{H,+},\qquad S_{d,-}=Q_d^{G,-}\times Q_{d_H}^{H,-}.                \label{signs}
\end{equation}
The action of elements of determinant $-1$ in both orthogonal factors exchanges these orbits. The product of the two special orthogonal groups preserves each orbit. For each $e<d$, the subspaces with this value of $e$ form one orbit.
\end{proposition}

\begin{proof}
At $e=d$, the residual nondegenerate space of dimension $m-2e$ is zero. The subspace is therefore $L_W\oplus L_Y$, where $L_W\subset W$ and $L_Y\subset Y$ are maximal isotropic subspaces. Since the subspace lies in the chosen ambient orbit, $\varepsilon(L_W)+\varepsilon(L_Y)=0$ in $\mathbf Z/2$. The possible pairs of labels are therefore $(+,+)$ and $(-,-)$. Each special orthogonal factor is transitive on the subspaces with a fixed label. Choose a reflection of determinant $-1$ in each orthogonal factor. These reflections switch both labels. Their product has determinant $1$ in the ambient space. The stabilizer of each maximal isotropic subspace is the corresponding signed Siegel parabolic. Their product gives (\ref{signs}).

When $e<d$, the remaining nondegenerate spaces have positive dimension. Witt extension first gives an element of $\mathrm O(W)\times\mathrm O(Y)$ carrying one such subspace to another. Since both subspaces lie in the chosen ambient orbit, the product of its two determinants is one. If both determinants are $-1$, choose a reflection $r$ on the residual source space. The corresponding reflection on the residual target space is $J_erJ_e^{-1}$. The pair of reflections preserves the graph. We extend each reflection by the identity on the two hyperbolic summands of a Witt decomposition. Composing with this pair changes both determinants to $1$. Thus the required element lies in $\SO(W)\times\SO(Y)$, and $e$ is the complete invariant.
\end{proof}

At the even orthogonal endpoint, we write $w_{d,+}:=w_d$. Choose involutive reflections $r_G$ of $W$ and $r_H$ of $Y$ that carry the standard positive isotropic subspaces to the standard negative ones. Each reflection has determinant $-1$. We extend the matrix formulas (\ref{jm}) and (\ref{tgh}) to this pair and put $c:=t(r_G,r_H)$. The product of the two determinants is one, so $c\in\bfS$. The same matrix identities show that $c$ normalizes $U$ and preserves $\psi_U$. Let $w_{d,-}:=w_{d,+}c$. The stabilizer and its Levi action are obtained by conjugation through $(g,h)\mapsto(r_Ggr_G,r_Hhr_H)$. The projected Jacquet group and character are unchanged. Thus the two signed terms use the same Jacquet coefficient below, with the Levi identifications on their respective isotropic subspaces.

\subsection{The Jacquet modules}

Let $W_0:=F^{ni}$. Let $q:=m-e$ and $p:=e$. For $1\le j\le n$, fix a space $M_j:=Q_j\oplus P_j$ with $Q_j:=F^q$ and $P_j:=F^p$. The symbols $Q_j$ and $P_j$ in this block decomposition denote vector spaces. Consider the parabolic of $\GL_N$ whose Levi subgroup acts on the successive summands $W_0\oplus M_1,\ M_2,\ldots,M_n$. Let $V_{ni+m,m^{n-1}}$ be its unipotent radical. For $n\ge2$, we write its elements in blocks corresponding to $W_0,M_1,\ldots,M_n$:
\[
v=\begin{pmatrix}
I_{ni}&0&y_2&y_3&\cdots&y_n\\
    &I_m&x_1&\star&\cdots&\star\\
    &&I_m&x_2&\cdots&\star\\
    &&&\ddots&\ddots&\vdots\\
    &&&&I_m&x_{n-1}\\
    &&&&&I_m
  \end{pmatrix},\qquad
  \begin{gathered}
y_j\in\Mat_{ni\times m}(F),\\
x_j\in\Mat_m(F).
  \end{gathered}
\]
The $(W_0,M_1)$ block is zero because these two spaces form a single Levi block. For $n=2$, the display has just three diagonal blocks. For $2\le j<n$, denote the adjacent block $M_{j+1}\to M_j$ by $x_j$. The $M_1$-component of the first adjacent block is denoted by $x_1:M_2\to M_1$. Let $a_e:=\diag(I_q,0_p)$. Define the character $\psi_e$ by
\begin{equation}
  \psi_e(v):=\psi^{-1}\!\left( \operatorname{tr}(a_ex_1)+\sum_{j=2}^{n-1}\operatorname{tr}x_j\right).
  \label{Ce}
\end{equation}
Its twisted Jacquet module is $J_{V_{ni+m,m^{n-1}},\psi_e}(\Delta(\tau,m+i))$. The character $\psi_e$ is trivial on the $(W_0,M_2)$ block. For $n=1$, the group $V_{ni+m,m^{n-1}}$ is trivial and we set $\psi_e:=1$. Its Jacquet module is then $\Delta(\tau,m+i)$.

Let $\psi'_e$ be the projected character for $w_e^{(0)}$ in (3.2)--(3.4) of \cite{GS21}. For $n\ge2$, the change of basis defining $w_e$ gives $\psi_e(v)=\psi'_e(w_{\mathrm{GL}}^{-1}vw_{\mathrm{GL}})$. The isomorphism of Jacquet quotients is induced by $\xi\mapsto\Delta(\tau,m+i)(w_{\mathrm{GL}})\xi$ on the space of $\Delta(\tau,m+i)$. The stabilizer action is therefore the action through $\ell_e$ defined in Proposition~\ref{stab}.

\begin{proposition}\label{fibre}
Let $\lambda_{\calH}:=(ni+m(2-n)-\delta_{\calH})/2$. Let $W_{e,s}$ denote the representation of $S_e$ that occurs in the compact induction attached to $w_e$. The compact induction is unnormalized. The representation $W_{e,s}$ corresponds to the stabilizer representation $\sigma_{[\tau',e]}$ in (12.47)--(12.53) of \cite{GS22}. The notation distinguishes it from the given representation $\sigma$ of $G$. The underlying space of $W_{e,s}$ is $J_{V_{ni+m,m^{n-1}},\psi_e}(\Delta(\tau,m+i))$. Its action on the Levi subgroup is
\begin{equation}
  \begin{aligned}
W_{e,s}(a,A,b) &={}|\det a\det A|^{s+\lambda_{\calH}}\\
    &\qquad\cdot \bigl[J_{V_{ni+m,m^{n-1}},\psi_e}(\Delta(\tau,m+i))\bigr] \bigl(\ell_e(a,A,b)\bigr).
  \end{aligned}
  \label{rhoe}
\end{equation}
For a general $s_0\in S_e$, the same formula defines $W_{e,s}(s_0)$, with $\ell_e(a,A,b)$ replaced by $\ell_e(s_0)$ and $\det a\det A$ by $\det\ell_e(s_0)$. In particular, the unipotent radical acts through the operators induced by $\Delta(\tau,m+i)(\ell_e(s_0))$ on the Jacquet quotient. The change of variables in the quotient measure contributes the character
\begin{equation}
  \mathfrak j_e(a,A,b) :=|\det a|^{-m(n-1)}|\det A|^{-m(n-1)}.                \label{measure}
\end{equation}
This character is trivial on $G_e$. The formulas apply to both signs at $e=d$ for the even orthogonal group.
\end{proposition}

\begin{proof}
The projected group and character are (3.2)--(3.4) of \cite{GS21}, after the change of coordinates in (3.16)--(3.17) there. In particular, the $(W_0,M_1)$ block does not occur in this Jacquet functor. The local coefficient calculation for $i=0$ is (12.53) of \cite{GS22}. We compute the determinant factors for the present block sizes.

Let $U_{e,0}:=U\cap w_e^{-1}Q_Nw_e$. Let $X$ be the sum of the first $n-1$ blocks of size $m$. Let $L$ be the central maximal isotropic space. Then $\dim X=m(n-1)$ and $\dim L=ni+m$. In bases adapted to $L$, the free coordinates of $U_{e,0}\backslash U$ consist of $x\in\Hom(L,X)$ and a symmetric or alternating matrix $y$ on $X$, as in (3.42) of \cite{GS21}. Let $T$ and $h_L$ be the stabilizer Levi actions on $X$ and $L$. Conjugation sends these coordinates to $x\longmapsto Txh_L^{-1}$ and $y\longmapsto Ty\,{}^tT$. We have $\det T=1$ and $\det h_L=\det a\det A$ by Proposition~\ref{stab}. The quotient measure therefore contributes $|\det T|^{\dim L}|\det h_L|^{-\dim X}=|\det a\det A|^{-m(n-1)}$. The symmetric or alternating coordinates contribute one because $\det T=1$. The quotient integral uses this Jacobian as in (3.46) of \cite{GS21}. For $n=1$, $X=0$ and this factor is one.

By (\ref{deltaQ}), normalized Siegel induction contributes $|\det a\det A|^{s+(N-\delta_{\calH})/2}$. Consequently, the total exponent is $s+\frac{N-\delta_{\calH}}2-m(n-1)=s+\frac{ni+m(2-n)-\delta_{\calH}}2$. This proves (\ref{rhoe}) and (\ref{measure}). The simultaneous reflections defining the second orthogonal endpoint preserve these determinant factors.
\end{proof}

We use the following fact for additive matrix groups; see \cite{BZ76,BZ77}. Let $D$ be such a group and let $D'$ be a vector subgroup of $D$. Let $V$ be a smooth representation of $D$. If $J_{D,\chi}(V)=0$ for every smooth character $\chi$ with $\chi|_{D'}\ne1$, then $D'$ acts trivially on $V$. This is the Fourier argument used in the proof of Proposition~12.6 of \cite{GS22}.

Write $\mathcal C_e:=J_{V_{ni+m,m^{n-1}},\psi_e}(\Delta(\tau,m+i))$. Let $D_1$ be the unipotent subgroup consisting of the following matrices, with the same block sizes:
\[
d_1(y):=\begin{pmatrix}
I_{ni}&y&0&\cdots&0\\
    &I_m&0&\cdots&0\\
    &&I_m&&0\\
    &&&\ddots&\vdots\\
    &&&&I_m
  \end{pmatrix},\qquad y\in\Mat_{ni\times m}(F).
\]
For $n\ge2$, the subgroup $D_1$ normalizes $V_{ni+m,m^{n-1}}$. Its conjugation changes only blocks on which $\psi_e$ is trivial. Thus $D_1$ acts on $\mathcal C_e$ in this case. We define the subgroup $D_Q:=\left\{d_1\!\left((y_Q,0_{ni\times p})\right): y_Q\in\Mat_{ni\times q}(F)\right\}$. As in (3.18)--(3.19) of \cite{GS21}, we write a Fourier parameter for the trace pairing with $y$ as $A=\binom{A_Q}{A_P}:W_0\longrightarrow Q_1\oplus P_1$. Here $A_Q\in\Mat_{q\times ni}(F)$ and $A_P\in\Mat_{p\times ni}(F)$. For $n\ge2$, the character $y\mapsto\psi^{-1}(\tr(Ay))$, together with $\psi_e$, is represented under the trace pairing by
\begin{equation}
  \begin{gathered}
N_A:=\begin{pmatrix}
      0&&&&&\\
A&0&&&&\\
      &a_e&0&&&\\
      &&I_m&\ddots&&\\
      &&&\ddots&0&\\
      &&&&I_m&0
    \end{pmatrix},\qquad
a_e=\begin{pmatrix}I_q&0\\0&0_p\end{pmatrix},\\[4pt]
N_A^n\big|_{W_0}=a_eA=\binom{A_Q}{0_{p\times ni}}.
  \end{gathered}
  \label{Ae}
\end{equation}
The rows and columns correspond to $W_0,M_1,\ldots,M_n$. The formula for $N_A^n|_{W_0}$ uses the chosen identification $M_n\simeq F^m$. For $n=2$, there are no $I_m$-blocks below $a_e$.

For $n=1$, the Speh construction for the character $\tau$ gives $\mathcal C_e\simeq\tau\circ\det$. Hence $D_1$ acts trivially on $\mathcal C_e$. For $n\ge2$, adjoining $D_1$ to $V_{ni+m,m^{n-1}}$ gives the unipotent radical $V_{ni,m^n}$. The character is supported on adjacent blocks, with corresponding nilpotent matrix $N_A$ in (\ref{Ae}). If $A_Q\ne0$, then $N_A^n\ne0$. The associated Jordan partition therefore has first part greater than $n$. The criterion following (\ref{WFn}) gives a zero Jacquet module for every $A$ with $A_Q\ne0$. A character indexed by $A$ is trivial on $D_Q$ precisely when $A_Q=0$. The Fourier argument above therefore shows that $D_Q$ acts trivially on $\mathcal C_e$.

The $P_1$-coordinates will be included in the block $W_0\oplus P_1$ in Section~\ref{coinvariants}.

\subsection{Unipotent invariance and the filtration}

Fix $e>0$. Consider the unipotent radical $U_e^G$ of the parabolic $P_e$. Let $c:=m-2e$. We write $M_j=E_j\oplus C_j\oplus E'_j$ with summands of dimensions $e,c,e$, respectively.  Thus $Q_j=E_j\oplus C_j$ and $P_j=E'_j$. For $Z\in\Mat_{(m-e)\times e}(F)$ and $A\in\Mat_{e\times c}(F)$, put
\[
  \nu_e(Z):=\begin{pmatrix}I_{m-e}&Z\\0&I_e\end{pmatrix},\qquad
  g(A):=\begin{pmatrix}I_e&A&0\\0&I_c&0\\0&0&I_e\end{pmatrix}.
\]
Let $D_2$ and $D_3$ be the subgroups formed by the matrices
\[
  \begin{aligned}
    D_2(Z)&:=\diag(I_{ni},I_m,\nu_e(Z)^{\Delta_{n-1}}),\\
    D_3(A)&:=\diag(I_{ni},g(A)^{\Delta_n}).
  \end{aligned}
\]
For $n\ge2$, both subgroups normalize $V_{ni+m,m^{n-1}}$. The identities $\nu_e(Z)a_e=a_e$ and $g(A)a_e=a_eg(A)$ show that they preserve $\psi_e$. Thus they act on $J_{V_{ni+m,m^{n-1}},\psi_e}(\Delta(\tau,m+i))$ in this case. For
\[
y=\begin{pmatrix}I_e&w&c_0\\0&I_c&w'\\0&0&I_e\end{pmatrix}
  \in U_e^G,
\]
we have $w\in\Mat_{e\times c}(F)$, $w'\in\Mat_{c\times e}(F)$, and $c_0\in\Mat_e(F)$. Put $Z_y:=\binom{c_0-ww'}{w'}\in\Mat_{(m-e)\times e}(F)$. Conjugation into the Siegel Levi gives
\begin{equation}
z(y):=\diag(I_{ni},g(w),y^{\Delta_{n-1}}) =D_3(w)D_2(Z_y). \label{zy}
\end{equation}
Thus $z(y)=\ell_e((y,1))$ for $y\in U_e^G$.

\begin{lemma}\label{D2D3}
Under (\ref{WFn}), $D_2$ and $D_3$ act trivially on $J_{V_{ni+m,m^{n-1}},\psi_e}(\Delta(\tau,m+i))$. The subgroup $D_3$ is trivial when $c=0$.
\end{lemma}

\begin{proof}
For $n=1$, $\Delta(\tau,m+i)\simeq\tau\circ\det$. Both subgroups have determinant one, so they act trivially. Assume $n\ge2$. Let $V:=V_{ni+m,m^{n-1}}$. The proof combines the root exchange and nilpotent orbit argument in Proposition~3.2 of \cite{GS21} with the bound (\ref{WFn}).

For $D_2$, let $V_2$ be the subgroup of $V$ in which the $(P_j,Q_{j+1})$-blocks vanish for $2\le j<n$. These blocks are adjacent, so these conditions define a subgroup. When $n=2$, we take $V_2:=V$. Fix $0\ne B\in\Mat_{e\times(m-e)}(F)$. We apply the root exchanges in (3.34)--(3.38) of \cite{GS21} to the character $D_2(Z)\mapsto\psi^{-1}(\tr(BZ))$. The additional $W_0$-rows remain in the first block throughout. The resulting subgroup is
\[
V'_2:=\left\{ \diag(I_{ni},I_m,\nu_e(Z_2),\ldots,\nu_e(Z_n))v:
  \begin{array}{l}
v\in V_2,\\[-2pt]
Z_k\in\Mat_{(m-e)\times e}(F)
  \end{array}\right\}.
\]
On the displayed element, the character is $\psi_e(v)\psi^{-1}(\tr(BZ_n))$. By root exchange, its Jacquet module is zero if and only if the Jacquet module for the chosen character of $D_2$ is zero. Put $q:=m-e$ and $p:=e$. In block sizes $q,p$, define $C_B:=\bigl(\begin{smallmatrix}0_q&0\\B&0_p\end{smallmatrix}\bigr)$. In the order $W_0,M_1,\ldots,M_n$, the trace-dual matrix of the resulting character is
\[
N_B:=\begin{pmatrix}
    0_{ni}&0&0&\cdots&0\\
    0&0_m&0&\cdots&0\\
    0&a_e&0_m&\ddots&\vdots\\
    \vdots&0&\ddots&\ddots&0\\
    0&0&\cdots&I_m&C_B
  \end{pmatrix}.
\]
The first block below the diagonal on $M_1,\ldots,M_n$ is $a_e$. Each subsequent block below the diagonal is $I_m$. For $n=2$, the last two block rows and columns form $\bigl(\begin{smallmatrix}0_m&0\\a_e&C_B\end{smallmatrix}\bigr)$. The restriction of $N_B^n$ from $M_1$ to $M_n$ has matrix $C_Ba_e=C_B\ne0$. Thus the Jordan partition of $N_B$ has first part greater than $n$.

To apply the vanishing criterion following (\ref{WFn}), we verify its hypotheses for $V'_2$. Let $\mathfrak c_2:=\Lie V'_2$. Let $\mathfrak g:=\mathfrak{gl}_N$. For a diagonal matrix $S$ with rational eigenvalues, let $\mathfrak g_a^S$ be the eigenspace of $[S,\cdot]$ with eigenvalue $a$. We also write $\mathfrak g_{>1}^S:=\bigoplus_{a>1}\mathfrak g_a^S$. For $2\le j\le n$, define
\[
  \begin{aligned}
    S_{2,j}&:=\begin{pmatrix}(2(n-j)+1)I_q&0\\0&(2(n-j)-1)I_p\end{pmatrix},\\
    S_2&:=\diag\bigl((2n-1)I_{ni+m},S_{2,2},\ldots,S_{2,n}\bigr).
  \end{aligned}
\]
The matrix $S_2$ uses the same block order. The exchanged root groups are those in (3.38) of \cite{GS21}, with the first block enlarged from $m$ to $ni+m$. The added $W_0$-rows have weights at least two for $S_2$. Thus
\begin{equation}
  \mathfrak c_2=\mathfrak g_{>1}^{S_2},\qquad \mathfrak g_1^{S_2}=0.                                \label{S2}
\end{equation}
The nonzero blocks of $N_B$ have weight $-2$ for $S_2$. Equation (\ref{S2}) verifies the hypotheses of the vanishing criterion following (\ref{WFn}). Since $N_B^n\ne0$, the Jacquet module for every nontrivial character of $D_2$ is zero. It follows that $D_2$ acts trivially on $J_{V_{ni+m,m^{n-1}},\psi_e}(\Delta(\tau,m+i))$.

If $c=0$, the subgroup $D_3$ is trivial. Assume $c>0$ for the rest of the proof. The matrices in (3.39) of \cite{GS21} give the subgroup $D_3$. We verify the corresponding local vanishing argument. All matrices below use the order $W_0,M_1,\ldots,M_n$. Within each $M_j$, the block sizes are $e,c,e$, in the order $E_j,C_j,E'_j$. Blocks of size zero are omitted. For $n=2$, the intermediate blocks in the banded matrices below are omitted. For $1\le j<n$, we write
\[
x_j=\begin{pmatrix}
x_j^{11}&x_j^{12}&x_j^{13}\\
x_j^{21}&x_j^{22}&x_j^{23}\\
x_j^{31}&x_j^{32}&x_j^{33}
  \end{pmatrix},\qquad x_j^{21}\in\Mat_{c\times e}(F).
\]
Let $V_3$ be the subgroup of $V$ obtained by replacing $x_j^{21}$ by zero for $1\le j<n$. These are adjacent blocks, so the zero conditions define a subgroup. For $C\in\Mat_{c\times e}(F)$, put $\mathsf Y(C):=\bigl(\begin{smallmatrix}0&0&0\\C&0&0\\0&0&0\end{smallmatrix}\bigr)$. For $1\le j<n$, define
\[
Y_j(C):=\begin{pmatrix}
I_{ni+(j-1)m}&0&0&0\\
    0&I_m&\mathsf Y(C)&0\\
    0&0&I_m&0\\
    0&0&0&I_{(n-j-1)m}
  \end{pmatrix}.
\]
For $1\le j<n$, let $Y_j:=\{Y_j(C):C\in\Mat_{c\times e}(F)\}$ be the corresponding unipotent subgroup. For $A\in\Mat_{e\times c}(F)$ and $1\le j\le n$, put $g_j(A):=\diag(I_{ni+(j-1)m},g(A),I_{(n-j)m})$. Let $0\ne B\in\Mat_{c\times e}(F)$ and consider the character $D_3(A)\longmapsto\psi^{-1}(\tr(BA))$. We use the convention $[x,y]:=xyx^{-1}y^{-1}$ for group commutators. Direct multiplication gives
\begin{equation}
  \setlength{\arraycolsep}{4pt}
  \begin{aligned}
    [Y_j(C),g_j(A)]&=\begin{pmatrix}
I_{ni+(j-1)m}&0&0&0\\
      0&I_m&-\diag(AC,0_{m-e})&0\\
      0&0&I_m&0\\
      0&0&0&I_{(n-j-1)m}
    \end{pmatrix},\\
    \psi_e([Y_j(C),g_j(A)])&=\psi(\tr(CA)).
  \end{aligned}
  \label{commutator}
\end{equation}
For $j=1$, the same formula holds because $a_e$ is the identity on $E_1$. The trace pairing is nondegenerate. We may therefore exchange $Y_1$ with the subgroup consisting of $g_1(A_1)$, using root exchange in the form cited after (\ref{WFn}). The first internal block is then arbitrary. The remaining $n-1$ internal blocks still have the common value $g(A)$. The character on these internal blocks is $\psi^{-1}(\tr(BA))$.

For $j=2,\ldots,n-1$, exchange $Y_j$ with the subgroup consisting of $g_j(A_j)$. Before the $j$th exchange, the blocks with indices less than $j$ are independent. The blocks with indices $j,\ldots,n$ have the same value $g(A)$. The block $x_{j-1}^{21}$ has already been removed. Thus the only remaining adjacent block that pairs with $g_j(A_j)$ is $x_j^{21}$, with pairing (\ref{commutator}).

The common subgroup is obtained by setting $x_k^{21}=0$ in $V$ for $1\le k\le j$ and adjoining the first $j-1$ independent internal blocks and the common internal blocks with indices $j,\ldots,n$. Both exchanged subgroups normalize this common subgroup. The two contributions from the common values of $g(A)$ in positions $j$ and $j+1$ cancel in the character. The remaining commutators contribute only blocks with trivial character. At each step, the exchange replaces $x_j^{21}$ by zero and allows the $j$th internal block to vary independently. At the last step, the common variable $A$ is the variable in the $n$th block. The resulting subgroup is
\[
V'_3:=\left\{ \diag(I_{ni},g(A_1),\ldots,g(A_n))v:
  \begin{array}{l}
v\in V_3,\\[-2pt]
A_k\in\Mat_{e\times c}(F)
  \end{array}\right\}.
\]
On the displayed element, the character is $\psi_e(v)\psi^{-1}(\tr(BA_n))$. Root exchange again reduces the required vanishing to the Jacquet module for this subgroup and character. Put $b_B:=\bigl(\begin{smallmatrix}0&0&0\\B&0&0\\0&0&0\end{smallmatrix}\bigr)$. The resulting character has nilpotent matrix
\[
N_B:=\begin{pmatrix}
    0&0&0&\cdots&0&0\\
    0&0&0&\cdots&0&0\\
    0&a_e&0&\cdots&0&0\\
    \vdots&0&I_m&\ddots&\vdots&\vdots\\
    0&\vdots&\ddots&\ddots&0&0\\
    0&0&\cdots&0&I_m&b_B
  \end{pmatrix}.
\]
For $n=2$, this matrix is $\left(\begin{smallmatrix}0&0&0\\0&0&0\\0&a_e&b_B\end{smallmatrix}\right)$. By matrix multiplication, the block of $N_B^n$ in row $M_n$ and column $M_1$ is
\begin{equation}
b_Ba_e=\begin{pmatrix}0&0&0\\B&0&0\\0&0&0\end{pmatrix}\ne0.
  \label{cube}
\end{equation}
Thus the Jordan partition of $N_B$ has first part greater than $n$.

To verify the hypotheses of the local vanishing criterion, we calculate the pairing on the space of weight one. Let $\mathfrak c_3:=\Lie V'_3$. For $1\le j\le n$, define $s_j:=\diag\bigl((2(n-j)+1)I_e,(2(n-j)-1)I_c,2(n-j)I_e\bigr)$. Put $S_3:=\diag\bigl((2n-2)I_{ni},s_1,\ldots,s_n\bigr)$. The nonzero blocks of $N_B$ have weight $-2$ for $S_3$. As in Definition~2.5.1 of \cite{GGS17}, the pairing on $\mathfrak g_1^{S_3}$ is defined by $\omega_{N_B}(X,Y):=\operatorname{tr}(N_B[X,Y])$, where $[X,Y]=XY-YX$. The chosen bases identify $E_j$ and $E'_j$ with $F^e$. The next four matrices use row and column block sizes $e,c,e$. For $A,U\in\Mat_e(F)$, $D\in\Mat_{e\times c}(F)$, and $V\in\Mat_{c\times e}(F)$, put
\[
  \begin{aligned}
    \mathsf A(A)&:=\begin{pmatrix}0&0&A\\0&0&0\\0&0&0\end{pmatrix},&
    \mathsf Z(D)&:=\begin{pmatrix}0&0&0\\0&0&0\\0&D&0\end{pmatrix},\\
    \mathsf U(U)&:=\begin{pmatrix}0&0&0\\0&0&0\\U&0&0\end{pmatrix},&
    \mathsf V(V)&:=\begin{pmatrix}0&0&0\\0&0&V\\0&0&0\end{pmatrix}.
  \end{aligned}
\]
For $R\in\Mat_{ni\times c}(F)$, put $\widehat R:=\bigl(0_{ni\times e},\,R,\,0_{ni\times e}\bigr)$. For $X\in\Mat_{ni\times e}(F)$, put $\widehat X:=\bigl(X,\,0_{ni\times c},\,0_{ni\times e}\bigr)$. Their column sizes are $e,c,e$. The radical of the pairing is
\begin{equation}
  \mathfrak r_{3,B}:=\left\{
  \begin{pmatrix}
    0&\widehat R&0&\cdots&0\\
    0&\mathsf Z(D)&\mathsf V(-BT)&\cdots&0\\
    0&0&\mathsf A(T)&\ddots&\vdots\\
    \vdots&\vdots&\ddots&\ddots&\mathsf V(-BT)\\
    0&0&\cdots&0&\mathsf A(T)
  \end{pmatrix}:
  \begin{array}{l}
R\in\Mat_{ni\times c}(F),\\
D\in\Mat_{e\times c}(F),\\
T\in\Mat_e(F)
  \end{array}\right\}.
  \label{rad3}
\end{equation}
Here $\mathsf V(-BT)$ occupies every block immediately above the diagonal of the $M_1\oplus\cdots\oplus M_n$ part. All other blocks above the diagonal in that part are zero. To verify this formula, let $U_j\in\Mat_e(F)$ and $V_j\in\Mat_{c\times e}(F)$ for $1\le j<n$. Put $\mathsf L_j:=\bigl(\begin{smallmatrix}0&0&0\\0&0&V_j\\U_j&0&0\end{smallmatrix}\bigr)$. Define
\[
  \mathfrak l_3:=\left\{
  \begin{pmatrix}
    0&0&\widehat X&0&\cdots&0\\
    0&0&\mathsf L_1&0&\cdots&0\\
    0&0&0&\mathsf L_2&\ddots&\vdots\\
    \vdots&\vdots&\vdots&\ddots&\ddots&0\\
    0&0&0&\cdots&0&\mathsf L_{n-1}\\
    0&0&0&\cdots&0&0
  \end{pmatrix}:
  \begin{array}{l}
X\in\Mat_{ni\times e}(F),\\
U_j\in\Mat_e(F),\\
V_j\in\Mat_{c\times e}(F),\\
    1\le j<n
  \end{array}\right\}.
\]
The weights of the adjacent blocks then give $\mathfrak c_3=\mathfrak g_{>1}^{S_3}\oplus\mathfrak l_3$. For $i=0$, the first row and column of blocks are absent. In particular, the parameter $X$ is absent. We use the following root matrices:
\[
  \begin{aligned}
A_j(A)&:=\diag(0_{ni+(j-1)m},\mathsf A(A),0_{(n-j)m}),\\
Z_j(D)&:=\diag(0_{ni+(j-1)m},\mathsf Z(D),0_{(n-j)m}) &&(1\le j\le n),\\
U_j(U)&:=\diag\left(0_{ni+(j-1)m},
    \begin{pmatrix}0&\mathsf U(U)\\0&0\end{pmatrix},
    0_{(n-j-1)m}\right),\\
V_j(V)&:=\diag\left(0_{ni+(j-1)m},
    \begin{pmatrix}0&\mathsf V(V)\\0&0\end{pmatrix},
    0_{(n-j-1)m}\right) &&(1\le j<n).
  \end{aligned}
\]
The central $2\times2$ matrices in these last two formulas have block sizes $m,m$. For $Y\in\Mat_{e\times ni}(F)$, put $\widehat Y:=\bigl(\begin{smallmatrix}Y\\0_{c\times ni}\\0_{e\times ni}\end{smallmatrix}\bigr)\in\Mat_{m\times ni}(F)$. In the order $W_0,M_1,M_2,M_3\oplus\cdots\oplus M_n$, define
\[
H(X):=\begin{pmatrix}
    0&0&\widehat X&0\\0&0&0&0\\0&0&0&0\\0&0&0&0
  \end{pmatrix},\qquad
K(Y):=\begin{pmatrix}
    0&0&0&0\\\widehat Y&0&0&0\\0&0&0&0\\0&0&0&0
  \end{pmatrix}.
\]
The last row and column of blocks are absent when $n=2$. Direct matrix multiplication gives
\begin{equation}
  \omega_{N_B}(H(X),K(Y))=-\tr(XY). \label{B0}
\end{equation}
For $1\le j<n$, it also gives
\begin{equation}
  \begin{aligned}
    \omega_{N_B}(U_j(U),A_j(A))&=-\tr(UA),\\
    \omega_{N_B}(V_j(V),Z_{j+1}(D))&=\tr(VD).
  \end{aligned}
  \label{pivots}
\end{equation}
For $2\le j<n$, the adjacent identity blocks of $N_B$ give $\omega_{N_B}(U_j(U),A_{j+1}(A))=\tr(UA)$ and $\omega_{N_B}(V_j(V),Z_j(D))=-\tr(VD)$. At $j=1$, the corresponding pairings are $\omega_{N_B}(U_1(U),A_2(A))=0$ and $\omega_{N_B}(V_1(V),Z_1(D))=0$. These equalities follow from the zero block of $a_e$ on $E'_1$. The last diagonal block $b_B$ gives $\omega_{N_B}(A_n(A),Z_n(D))=\tr(BAD)$. All other pairings between the displayed root spaces are zero, apart from those obtained by skew symmetry.

For a vector in $\mathfrak g_1^{S_3}$, let $A_j,Z_j,U_j,V_j$ denote its matrix coordinates in the four families of root spaces above. Pairing with every $U_j(U)$ and $V_j(V)$ gives $A_1=0$, $A_2=\cdots=A_n=T$, and $Z_2=\cdots=Z_n=0$ by nondegeneracy of the trace pairing. Pairing with every $A_j(A)$ then gives $U_j=0$ for $1\le j<n$. Pairing with $Z_n(D)$ gives $V_{n-1}+BT=0$. Pairing with $Z_j(D)$ gives $V_{j-1}=V_j$ for $2\le j<n$. Consequently, $U_j=0$ and $V_j=-BT$ for $1\le j<n$. For $n=2$, there are no intermediate equations $V_{j-1}=V_j$. Equation (\ref{B0}) forces the $H(X)$ and $K(Y)$ coordinates to vanish. The $(W_0,C_1)$ and $(E'_1,C_1)$ coordinates are unrestricted. Conversely, these conditions make every displayed pairing zero. Hence these conditions define the radical in (\ref{rad3}). The nondegenerate trace pairings on the quotient by the radical also show that $\mathfrak l_3+\mathfrak r_{3,B}$ is a maximal isotropic subspace of $\mathfrak g_1^{S_3}$ for $\omega_{N_B}$.

Consider the matrices
\[
  \begin{aligned}
r_0(R)&:=\diag\left(
    \begin{pmatrix}I_{ni}&\widehat R\\0&I_m\end{pmatrix},I_{(n-1)m}\right),
    &R&\in\Mat_{ni\times c}(F),\\
r_1(D)&:=\diag(I_{ni},I_m+\mathsf Z(D),I_{(n-1)m}), &D&\in\Mat_{e\times c}(F),\\
k(T)&:=\diag\bigl(I_{ni+m}, \underbrace{I_m+\mathsf A(T),\ldots,I_m+\mathsf A(T)}_{n-1\text{ blocks}}\bigr), &T&\in\Mat_e(F).
  \end{aligned}
\]
The matrices $r_0(R)$ lie in $D_Q$. They therefore act trivially before the exchanges, by the calculation following (\ref{Ae}). Let $R\in\Mat_{ni\times c}(F)$ and $C\in\Mat_{c\times e}(F)$. During the $D_3$ exchanges, direct multiplication gives
\[
  [r_0(R),Y_1(C)]=\begin{pmatrix}
I_{ni}&0&\bigl(RC,\,0_{ni\times c},\,0_{ni\times e}\bigr)&0\\
    0&I_m&0&0\\
    0&0&I_m&0\\
    0&0&0&I_{(n-2)m}
  \end{pmatrix}.
\]
The corresponding unipotent subgroup belongs to the common subgroup used in root exchange. The character is trivial on it. The matrices $r_0(R)$ commute with $Y_j(C)$ for $j>1$ and with every $g_j(A)$. Since $k(T)=D_2\!\left(\binom{T}{0_{c\times e}}\right)$, these matrices belong to $D_2$. They act trivially before the exchanges by the first part of this proof. For $1\le j<n$, direct multiplication gives
\[
  [k(T),Y_j(C)]=\begin{pmatrix}
I_{ni+(j-1)m}&0&0&0\\
    0&I_m&\mathsf V(-CT)&0\\
    0&0&I_m&0\\
    0&0&0&I_{(n-j-1)m}
  \end{pmatrix}.
\]
This unipotent subgroup also belongs to the common subgroup and has trivial character. The matrices $k(T)$ commute with every $g_j(A)$. The commutator calculations show that the root exchange maps are equivariant for $r_0(R)$ and $k(T)$. Both families therefore act trivially after root exchange.

We may adjoin $r_0(R)$ with trivial character. This adds the subspace parametrized by $R$ in (\ref{rad3}). We may also adjoin $k(T)$ with trivial character. The blocks $\mathsf V(-BT)$ in (\ref{rad3}) already belong to $\mathfrak c_3$ and have trivial character. Adjoining $k(T)$ therefore adds the subspace parametrized by $T$ in the radical.

The matrices $r_1(D)$ normalize the group obtained after adjoining $r_0(R)$ and $k(T)$. Indeed, brackets with roots of weight one lie in $\mathfrak g_{>1}^{S_3}$. Further brackets have higher weight. Equation (\ref{rad3}) shows that these conjugations preserve the character. Fix $0\ne L\in\Mat_{c\times e}(F)$. Consider the character $r_1(D)\mapsto\psi^{-1}(\tr(LD))$. Put $b_L:=\bigl(\begin{smallmatrix}0&0&0\\0&0&L\\0&0&0\end{smallmatrix}\bigr)$. The nilpotent matrix for this character is
\[
N_{B,L}:=\begin{pmatrix}
    0&0&0&\cdots&0&0\\
    0&b_L&0&\cdots&0&0\\
    0&a_e&0&\cdots&0&0\\
    \vdots&0&I_m&\ddots&\vdots&\vdots\\
    0&\vdots&\ddots&\ddots&0&0\\
    0&0&\cdots&0&I_m&b_B
  \end{pmatrix}.
\]
For $n=2$, this matrix is $\left(\begin{smallmatrix}0&0&0\\0&b_L&0\\0&a_e&b_B\end{smallmatrix}\right)$. By matrix multiplication, the block of $N_{B,L}^n$ in row $C_n$ and column $E'_1$ is $L\ne0$. Thus the Jordan partition of $N_{B,L}$ has first part greater than $n$. Let $s_{1,L}:=\diag\bigl((2n-1)I_e,(2n-3)I_c,(2n-1)I_e\bigr)$. Define $S_{3,L}:=\diag\bigl((2n-1)I_{ni},s_{1,L},s_2,\ldots,s_n\bigr)$. Relative to $S_3$, precisely the blocks $(W_0,C_1),\ (W_0,E_2),\ (E'_1,C_1),\ (E'_1,E_2)$ move from weight one to weight two. The blocks $(E_1,W_0)$ and $(E_1,E'_1)$ move from weight one to weight zero. Removing the corresponding pairings from (\ref{B0})--(\ref{pivots}) leaves the subspace parametrized by $T$ in (\ref{rad3}) as the full radical. We retain the blocks of $\mathfrak l_3$ that still have weight one. We add the subspace parametrized by $T$ in (\ref{rad3}). The remaining pairings in (\ref{pivots}) show that their sum is a maximal isotropic subspace of $\mathfrak g_1^{S_{3,L}}$. The group of integration has Lie algebra equal to this subspace plus $\mathfrak g_{>1}^{S_{3,L}}$. The vanishing criterion following (\ref{WFn}) therefore gives a zero Jacquet module for every $L\ne0$. By the Fourier argument above, the subgroup consisting of the $r_1(D)$ acts trivially. Adjoining it with trivial character therefore does not change the Jacquet module.

After adjoining this subgroup, the integration space contains the full radical in (\ref{rad3}). By Lemma~2.5.2 of \cite{GGS17}, the vanishing criterion following (\ref{WFn}) applies with $S_3$ and $N_B$. By (\ref{cube}), the associated partition has first part greater than $n$. The bound (\ref{WFn}) therefore makes the resulting Jacquet module zero for $B\ne0$. Thus every nontrivial character of $D_3$ gives a zero Jacquet module. It follows that $D_3$ acts trivially on $J_{V_{ni+m,m^{n-1}},\psi_e}(\Delta(\tau,m+i))$.
\end{proof}

\begin{theorem}\label{triviality}
For every orbit indexed by $e>0$, $U_e^G$ acts trivially on $W_{e,s}$. Equivalently, it acts trivially on $J_{V_{ni+m,m^{n-1}},\psi_e}(\Delta(\tau,m+i))$. The assertion holds for both signs at $e=d$ for the even orthogonal group.
\end{theorem}

\begin{proof}
For $n\geq2$, the factorization (\ref{zy}) and Lemma~\ref{D2D3} show that $z(y)$ acts trivially on $J_{V_{ni+m,m^{n-1}},\psi_e}(\Delta(\tau,m+i))$. The determinant character in (\ref{rhoe}) is trivial on unipotents. Thus $U_e^G$ acts trivially on $W_{e,s}$.

For $n=1$, $\Delta(\tau,m+i)\simeq\tau\circ\det$. The identity $\det z(y)=1$ proves triviality in this case as well. For the even orthogonal group, simultaneous outer conjugation carries the assertion for one endpoint to the other.
\end{proof}

We define the set of orbit indices by
\[
  \mathfrak E_G:=
  \begin{cases}
    \{0,1,\ldots,d\},&G\text{ symplectic or odd orthogonal},\\
    \{0,1,\ldots,d-1,(d,+),(d,-)\},&G=\SO_{2d}.
  \end{cases}
\]
For a signed index $e=(d,\varepsilon)$, we use the stabilizer $S_{d,\varepsilon}$ in (\ref{signs}). Let $W_{d,\varepsilon,s}$ denote the coefficient in (\ref{rhoe}) at that stabilizer. In numerical formulas, we set $e=d$ at either signed endpoint.

\begin{theorem}\label{geom}
Under (\ref{WFn}), $J_s$ has a finite filtration stable under $G\times H$. The successive quotients occur in increasing order of $e$, starting with $e=0$. Denote them by $J_{e,s}$. Some quotients may be zero. They satisfy
\[
J_{e,s}\simeq\ind_{S_e}^{G\times H}W_{e,s} \qquad(e\in\mathfrak E_G).
\]
For the even orthogonal group, the two terms at $e=d$ may be combined into their direct sum. The coefficient $W_{e,s}$ is given by (\ref{Ce}) and (\ref{rhoe}). For $e>0$, $U_e^G$ acts trivially on $W_{e,s}$.
\end{theorem}

\begin{proof}
We filter the induced representation by support on the successive double cosets. Functions supported on the open orbit form the first submodule. Let $X$ vary over the central maximal isotropic subspaces. For the remaining orbits, the condition $\dim(X\cap W)\ge e$ is closed. In coordinates, it is given by vanishing minors of the projection $X\to (W\oplus Y)/W$. Thus extending functions by zero gives the filtration in increasing order of $e$. The geometric lemma identifies its quotients with compact induction from the stabilizer of a representative; see \cite{BZ76,BZ77}. For $n\ge2$, Lemma~\ref{intersection} and Proposition~\ref{rzero} eliminate the terms with $r>0$. For $n=1$, the filtration starts with the orbits indexed by $e$. Propositions~\ref{stab}, \ref{evenSO}, and \ref{fibre} identify the stabilizers and their representations for these orbits. Exactness of twisted Jacquet functors gives the asserted filtration. The triviality of $U_e^G$ follows from Theorem~\ref{triviality}.
\end{proof}

At an endpoint for the split even orthogonal group, we write \(J_{d,+,s}\) and \(J_{d,-,s}\).

For $G=\SO_2$, the endpoint radical is trivial. Its Jacquet functor is therefore the identity. For $\SO_{2d}$ with $d\geq2$, the two endpoint parabolics are proper. We keep their Jacquet functors separate throughout.

\section{The orbit contributions}\label{terms}

\begin{lemma}\label{homlemma}
Let $L$ and $M$ be locally profinite groups. Let $V$ be a smooth $L\times M$-module. Let $\sigma$ be a smooth $L$-module and $\Pi$ a smooth $M$-module. Give $V\otimes\sigma$ the diagonal $L$-action. Then
\begin{equation}\label{quotient}
  \Hom_{L\times M}(V\otimes\sigma,\one_L\boxtimes\Pi) \simeq\Hom_M((V\otimes\sigma)_L,\Pi).
\end{equation}
If $\sigma$ is admissible, evaluation also gives
\begin{equation}\label{adjunction}
  \Hom_{L\times M}(V,\sigma^\vee\boxtimes\Pi) \simeq\Hom_M((V\otimes\sigma)_L,\Pi).
\end{equation}
Both isomorphisms are natural in $V$ and $\Pi$.
\end{lemma}

\begin{proof}
An $L$-invariant linear map on $V\otimes\sigma$ factors uniquely through $(V\otimes\sigma)_L$. This proves (\ref{quotient}). Evaluation $\sigma^\vee\otimes\sigma\to\bbC$ defines the forward map in (\ref{adjunction}).

Conversely, let $T:(V\otimes\sigma)_L\to\Pi$ be $M$-equivariant. For $v\in V$, define $T_v(u):=T([v\otimes u])$ for $u\in\sigma$. Choose a compact open subgroup $K\subset L$ fixing $v$. The coinvariant relations give $T_v(ku)=T_v(u)$ for $k\in K$. Let $e_K$ be averaging over $K$ with total mass one. Then $T_v=T_v\circ e_K$. Since $\sigma$ is admissible, $\sigma^K$ is finite dimensional. Thus $T_v$ belongs to $(\sigma^\vee)^K\otimes\Pi$. The map $v\mapsto T_v$ therefore takes values in $\sigma^\vee\boxtimes\Pi$. The coinvariant relations give $L$-equivariance. Since $T$ is $M$-equivariant, so is the map constructed above. These two constructions are inverse and commute with homomorphisms of $V$ and $\Pi$.
\end{proof}

For admissible $\sigma$, Lemma~\ref{homlemma} gives
\begin{equation}\label{hom2}
  \Hom_{G\times H}(J_s,\sigma^\vee\boxtimes\Pi) \simeq\Hom_H(\mathbf I_s(\tau,\sigma),\Pi)
\end{equation}
for every smooth $H$-module $\Pi$. This isomorphism is natural in $\Pi$. The representation identity in Theorem~\ref{main} therefore determines these $\Hom$ spaces.

We use the stabilizers and actions defined in Section~\ref{orbits}. The modulus formula (\ref{deltaQ}) gives $\delta_{P_e}(a)=|\det a|^{m-e-\delta_{\calH}}$ and $\delta_{R_e}(A)=|\det A|^{ni+m-e-\delta_{\calH}}$. Let $\delta_{S_e}$ denote the modulus character of $S_e$, with the determinant convention used for parabolic subgroups. The two unipotent radicals in (\ref{Se}) give $\delta_{S_e}(g,h)=\delta_{P_e}(g)\delta_{R_e}(h)$ for $(g,h)\in S_e$. We define the normalized coefficient by $\calE_{e,s}:=\delta_{S_e}^{-1/2}W_{e,s}$. The exponents $\alpha_e,\beta_e$ were defined in (\ref{ab}). Substitution in (\ref{rhoe}) gives
\[
  \begin{aligned}
    \calE_{e,s}(a,A,b) &={}|\det a|^{s+\alpha_e}|\det A|^{s+\beta_e}\\
    &\qquad\cdot \bigl[J_{V_{ni+m,m^{n-1}},\psi_e}(\Delta(\tau,m+i))\bigr] \bigl(\ell_e(a,A,b)\bigr).
  \end{aligned}
\]
The action of $zI_e\in\GL_e$ on $\calE_{e,s}$ differs from its action on $\calE_{e,0}$ by the factor $|z|^{es}$.

\subsection{The open orbit}

On the open orbit, the matrix $a_e$ in (\ref{Ce}) is $a_0=I_m$. For $n\ge2$, we have $D_1=D_Q$ when $e=0$. Hence $D_1$ acts trivially on $J_{V_{ni+m,m^{n-1}},\psi_0}(\Delta(\tau,m+i))$ by the argument following (\ref{Ae}). For $n=1$, the representation $\Delta(\tau,m+i)$ is a determinant character. It is therefore trivial on $D_1$.

For $a\ge0$, let $P_{a,nm}\subset\GL_{a+nm}$ be the standard parabolic with Levi subgroup $\GL_a\times\GL_{nm}$. Let $V_{a,m^n}$ be the unipotent radical for the composition $(a,m,\ldots,m)$. For $v\in V_{a,m^n}$, let $v_{j,j+1}$ be the block between the $j$th and $(j+1)$st blocks of size $m$. We use the character $\psi_{m,n}(v):=\psi^{-1}\!\left(\sum_{j=1}^{n-1}\tr v_{j,j+1}\right)$. It is trivial on $U_{P_{a,nm}}$. For $a=0$, we omit the first block and set $\GL_a:=\{1\}$. In particular, $P_{ni,nm}$ is a parabolic subgroup of $\GL_N$.

We write $D^{(r)}$ for the normalized Bernstein--Zelevinsky derivative of order $r$. For a representation of $\GL_b(F)$ and $1\le r\le b$, the derivative is a representation of $\GL_{b-r}(F)$. We use the character $u\mapsto\psi^{-1}(\sum_{j=1}^{r-1}u_{j,j+1})$ on the upper unitriangular group of the removed block of size $r$. We set $D^{(0)}=\operatorname{Id}$ and $D^{(r)}(\pi)=0$ for $r>b$. On the remaining general linear factor, the underlying unnormalized quotient is $D^{(r)}(\pi)|\det|^{r/2}$. The diagonal $\GL_m$ below is embedded by $g\mapsto\diag(I_a,g^{\Delta_n})$.

\begin{lemma}\label{derivative}
Let $\pi$ be a representation of $\GL_{a+nm}(F)$ of finite length. Assume that its wave-front orbits have first part at most $n$. There is an isomorphism of $\GL_a$-modules
\[
J_{V_{a,m^n},\psi_{m,n}}(\pi) \simeq \bigl((D^{(n)})^m\pi\bigr)|\det|^{mn/2}.
\]
The diagonal $\SL_m(F)$ acts trivially on this quotient.
\end{lemma}

\begin{proof}
Let $L:=\GL_{nm}(F)$ and $C_a:=U_{P_{a,nm}}$. Let $T_a:=J_{C_a,1}(\pi)$. Ordinary Jacquet modules preserve finite length, so $T_a$ has finite length as a $\GL_a\times L$-module. For $a=0$, we set $C_a:=1$ and $T_a:=\pi$.

Let $\kappa\boxtimes\rho$ be an irreducible constituent of $T_a$. We claim that every orbit in $\WF(\rho)$ has first part at most $n$. Suppose otherwise. Choose a maximal orbit $\mathcal O_\lambda$ of $\rho$ with $\lambda_1>n$. Choose $f_\lambda$ in lower triangular Jordan form with partition $\lambda$. On each Jordan block of size $b$, define $h_\lambda$ by $\diag(b-1,b-3,\ldots,1-b)$. Then $[h_\lambda,f_\lambda]=-2f_\lambda$. Let $U_\lambda$ and $\chi_\lambda$ be a unipotent subgroup and character defining the generalized Whittaker module for this orbit. We use the character convention $\chi_\lambda(\exp X):=\psi^{-1}(\tr(f_\lambda X))$ for $X\in\Lie U_\lambda$.

Choose an integer $c$ such that $c-\mu>1$ for every eigenvalue $\mu$ of $h_\lambda$ on $F^{nm}$. Let $h:=\diag(cI_a,h_\lambda)$ and $f:=\diag(0_a,f_\lambda)$. By the choice of $c$, all weights of $h$ on $\Lie C_a$ are greater than one. The trace character defined by $f$ is trivial on $C_a$. Embed $U_\lambda$ in the last $nm$ coordinates. The quotient defined by $h$ and $f$ is $J_{C_a\rtimes U_\lambda,\,1\boxtimes\chi_\lambda}(\pi) \simeq J_{U_\lambda,\chi_\lambda}(T_a)$. This isomorphism follows from transitivity of Jacquet modules. The corresponding twisted Jacquet module of $\rho$ is nonzero because $\mathcal O_\lambda$ is a maximal wavefront orbit. By exactness, this quotient is nonzero. The matrix $f$ has Jordan partition $(\lambda,1^a)$, whose first part is greater than $n$. The nonvanishing criterion following (\ref{WFn}) contradicts the assumed wavefront bound on $\pi$.

Let $V:=V_{0,m^n}\subset L$ and $\lambda_0:=(n^m)$. Let $N_L$ be the upper unitriangular group of $L$. We define a character of this group by $\psi_{\rm deg}(u):= \psi^{-1}\!\left(\sum_{k=0}^{m-1}\sum_{j=1}^{n-1} u_{kn+j,kn+j+1}\right)$. It is nondegenerate in each consecutive block of size $n$. It is trivial between those blocks. The permutation $(r-1)m+s\longmapsto(s-1)n+r$, with $1\le r\le n$ and $1\le s\le m$, places each of the $m$ chains of length $n$ in consecutive positions. This permutation identifies the nilpotent defining $(V,\psi_{m,n})$ with the one defining $(N_L,\psi_{\rm deg})$. The character $(V,\psi_{m,n})$ defines the neutral model for $\lambda_0$. The character $\psi_{\rm deg}$ defines a degenerate model for the same orbit. By Theorem~A of \cite{GGS17}, for every smooth $L$-module $T$ there is a natural surjection
\begin{equation}
J_{V,\psi_{m,n}}(T)\twoheadrightarrow J_{N_L,\psi_{\rm deg}}(T). \label{model}
\end{equation}
Naturality implies that this map commutes with $\GL_a$ when $T$ is a $\GL_a\times L$-module. Both Jacquet functors are exact.

The partition $\lambda_0$ dominates every partition of $nm$ with first part at most $n$. For an irreducible $\rho$ satisfying this bound, a nonzero $J_{V,\psi_{m,n}}(\rho)$ forces $\lambda_0$ to lie below its maximal wavefront partition. That partition already lies below $\lambda_0$, so the two partitions must agree. If that maximal partition differs from $\lambda_0$, the surjection (\ref{model}) also gives $J_{N_L,\psi_{\rm deg}}(\rho)=0$.

Suppose instead that the maximal partition is $\lambda_0$. For a maximal orbit, the dimension of either model equals the coefficient of that orbit in the character expansion. The dimension formula for both gradings is stated in Theorem~4.3.1(ii) of \cite{GGS17}. For an irreducible representation of $L$, this coefficient is one by Section~II.2 of \cite{MW87}. Hence $\dim J_{V,\psi_{m,n}}(\rho)=\dim J_{N_L,\psi_{\rm deg}}(\rho)=1$. The surjection (\ref{model}) is therefore an isomorphism.

We apply the two exact Jacquet functors to a composition series of $T_a$. The map (\ref{model}) is an isomorphism on every constituent $\kappa\boxtimes\rho$. Exactness therefore gives an isomorphism on $T_a$. The resulting map is $\GL_a$-equivariant. Jacquet transitivity identifies the left side with $J_{V_{a,m^n},\psi_{m,n}}(\pi)$. The right side is the quotient by $C_a\rtimes N_L$. Its character is trivial on $C_a$ and equals $\psi_{\rm deg}$ on $N_L$.

Taking one derivative of order $n$ removes the last block of size $n$ with its nondegenerate character. Repeating this $m$ times imposes precisely the character $\psi_{\rm deg}$ on $N_L$ and the trivial character on $C_a$. Each normalized derivative contributes $|\det|^{n/2}$ when expressed as an unnormalized quotient on the remaining $\GL_a$. Thus the quotient is $\bigl((D^{(n)})^m\pi\bigr)|\det|^{mn/2}$.

The diagonal $\SL_m(F)\subset L$ preserves the filtration induced by the composition series of $T_a$. Each nonzero successive quotient is $\kappa\otimes J_{V,\psi_{m,n}}(\rho)$. The second factor has dimension one. Assume $m\ge2$. The group $\SL_m(F)$ is perfect. It therefore acts trivially on each factor of dimension one. Relative to this filtration, its image consists of upper triangular block operators with identity diagonal blocks. It is therefore a nilpotent group. It is also perfect, as a quotient of $\SL_m(F)$. A perfect nilpotent group is trivial. For $m=1$, the group $\SL_m(F)$ is trivial.
\end{proof}

\begin{proposition}\label{openSpeh}
There is an isomorphism of $\GL_{ni}\times G$-modules
\[
  \delta_{P_{ni,nm}}^{-1/2} J_{V_{ni,m^n},\psi_{m,n}}(\Delta(\tau,m+i)) \simeq\Delta(\tau,i)|\det|^{-m/2}\boxtimes\one_G.
\]
\end{proposition}

\begin{proof}
For $c\ge1$, the normalized derivatives satisfy $D^{(n)}\Delta(\tau,c) \simeq\Delta(\tau,c-1)|\det|^{-1/2}$ and $D^{(r)}\Delta(\tau,c)=0$ for $r>n$, by the proof of Theorem~4 of \cite{CFGK23} and the formula for derivatives of induced representations in Lemma~4.5 of \cite{BZ77}. By the wave-front bound (\ref{WFn}), Lemma~\ref{derivative} applies to $\Delta(\tau,m+i)$. Iterating the derivative formula $m$ times gives $(D^{(n)})^m\Delta(\tau,m+i) \simeq\Delta(\tau,i)|\det|^{-m/2}$. Lemma~\ref{derivative} therefore identifies the unnormalized quotient, as a $\GL_{ni}$-module, with $\Delta(\tau,i)|\det|^{m(n-1)/2}$. On the $\GL_{ni}$ factor, $\delta_{P_{ni,nm}}^{-1/2}=|\det|^{-nm/2}$. The total exponent is $m(n-1)/2-nm/2=-m/2$.

The diagonal $\SL_m(F)$ acts trivially on the unnormalized quotient by Lemma~\ref{derivative}. The character $\delta_{P_{ni,nm}}^{-1/2}$ is trivial on this subgroup. Since $G\subset\SL_m(F)$, the normalized quotient also has trivial $G$-action.
\end{proof}

For $e=0$, Proposition~\ref{stab} gives the stabilizer $S_0=\{(b^\iota,m_{Q_{ni}}(a,b)u): a\in\GL_{ni},\ b\in G,\ u\in U_{Q_{ni}}\}$. We define the representation $\eta_s$ of $S_0$ by $\eta_s(b^\iota,m_{Q_{ni}}(a,b)u) :=\delta_{Q_{ni}}^{1/2}(a)|\det a|^s\Delta(\tau,i)(a)$. Propositions~\ref{fibre} and~\ref{openSpeh} give $J_{0,s}\simeq\ind_{S_0}^{G\times H}\eta_s$.

\begin{proposition}\label{integral}
For every smooth $G$-module $\sigma$, there is an $H$-isomorphism
\[
  (J_{0,s}\otimes\sigma)_G\simeq \Ind_{Q_{ni}}^H\bigl(\Delta(\tau,i)|\det|^s \boxtimes\sigma^\iota\bigr).
\]
\end{proposition}

\begin{proof}
For $\Phi\in\ind_{S_0}^{G\times H}\eta_s$ and $v\in\sigma$, define
\begin{equation}
  \mathcal T(\Phi\otimes v)(h) :=\int_G \Phi(g,h)\otimes\sigma(g)v\,dg. \label{int1}
\end{equation}
For fixed $h$, compact support modulo $S_0$ gives compact support in the variable $g$. Right translation in $g$ shows that $\mathcal T$ factors through $G$-coinvariants. The substitution $g=b^\iota g'$ gives, for $a\in\GL_{ni}$, $b\in G$, and $u\in U_{Q_{ni}}$,
\[
  \mathcal T(\Phi\otimes v)(m_{Q_{ni}}(a,b)uh) =\delta_{Q_{ni}}^{1/2}(a)|\det a|^s \bigl(\Delta(\tau,i)(a)\otimes\sigma(b^\iota)\bigr) \mathcal T(\Phi\otimes v)(h).
\]
Thus the $G$ factor of the Levi acts by $\sigma^\iota$.

Let $C_c^\infty(G)$ be the space of compactly supported locally constant complex functions on $G$, with the right regular action. The tensor identity in Lemma~2.8 of \cite{HS17} and Lemma~4.1 of \cite{HS17} give an isomorphism $(C_c^\infty(G)\otimes\sigma)_G\xrightarrow{\ \sim\ }\sigma$ for the trivial subgroup of $G$. It is defined by $f\otimes v\longmapsto\int_G f(g)\sigma(g)v\,dg$. Induction in stages applies this map to the coefficient over $Q_{ni}$. Since $Q_{ni}\backslash H$ is compact, taking $G$-coinvariants commutes with this induction. The equivariance calculation above therefore shows that $\mathcal T$ is an $H$-isomorphism.
\end{proof}

\subsection{The other orbits}\label{coinvariants}

Fix $e>0$ and use the notation of Section~\ref{orbits}. For $n=1$, the Speh representation is a determinant character. Let $p:=e$, $q:=m-e$, $c:=m-2e$, $d_0:=ni+e$, and $R_n:=\frac{(n-1)(n-2)}2$. For $n\ge2$, the argument following (\ref{Ae}) shows that $D_Q$ acts trivially on $J_{V_{ni+m,m^{n-1}},\psi_e}(\Delta(\tau,m+i))$. We combine $W_0$ and $P_1$ into the first diagonal block:
\begin{equation}
L_0:=W_0\oplus P_1,\qquad \dim L_0=ni+e.                \label{L0}
\end{equation}
To compute the initial coefficient, order the summands as $L_0,Q_1,M_2,\ldots,M_n$, where $M_j=Q_j\oplus P_j$. We keep the same labels for the blocks under this permutation.

For $z\in\Mat_{d_0\times q}(F)$, let
\[
n_0(z):=\begin{pmatrix}
I_{d_0}&z&0\\
    0&I_q&0\\
    0&0&I_{(n-1)m}
  \end{pmatrix}.
\]
Let $D_0:=\{n_0(z):z\in\Mat_{d_0\times q}(F)\}$. The first two block rows and columns correspond to $L_0$ and $Q_1$. The chosen basis identifies $\GL(L_0)$ with $\GL_{d_0}(F)$.

\begin{lemma}\label{endpoint}
Assume $n\ge2$. The group $D_0$ acts trivially on $J_{V_{ni+m,m^{n-1}},\psi_e}(\Delta(\tau,m+i))$. On the quotient obtained by adjoining $D_0$, $\GL(L_0)$ acts through the matrices $\diag(\alpha,I_{N-d_0})$, with $\alpha\in\GL_{d_0}(F)$.
\end{lemma}

\begin{proof}
Let $\mathcal V_e$ be the conjugate of the group in (\ref{Ce}) under the permutation just chosen. The commutator of $D_0$ with a root subgroup of $\mathcal V_e$ is either trivial or contained in a root subgroup of $\mathcal V_e$ in the row corresponding to $L_0$. The character is trivial on the resulting root subgroups and all further commutators. With rows and columns divided according to $L_0,Q_1,M_2,\ldots,M_n$, the only block missing from $\mathcal V_e$ is the $(L_0,Q_1)$-block. Hence $D_0\mathcal V_e=V_{d_0,q,m^{n-1}}$.

Let $0\ne C\in\Mat_{q\times d_0}(F)$. We associate to $C$ the character of $D_0$ given by $n_0(z)\mapsto\psi^{-1}(\tr(Cz))$. Nondegeneracy of the trace pairing shows that these are all its nontrivial characters. The resulting character of $D_0\mathcal V_e$ is supported on adjacent blocks. Let $b_q:=\binom{I_q}{0_{p\times q}}$. The corresponding nilpotent matrix is
\[
N_C:=\begin{pmatrix}
    0_{d_0}&&&&\\
C&0_q&&&\\
    &b_q&0_m&&\\
    &&I_m&\ddots&\\
    &&&\ddots&0_m
  \end{pmatrix}.
\]
For $n=2$ the matrix has three diagonal blocks. There is then no $I_m$-block below $b_q$. Matrix multiplication gives $N_C^n\big|_{L_0}=b_qC\ne0$. The argument in (12.60)--(12.61), in the proof of Proposition~12.6 of \cite{GS22}, therefore gives the trivial action of $D_0$, using the vanishing criterion following (\ref{WFn}).

After adjoining $D_0$, every block in the row corresponding to $L_0$ is present. The character is trivial on this row. The matrices $\diag(\alpha,I_{N-d_0})$ therefore normalize the group and preserve its character. Their action descends to the Jacquet module.
\end{proof}

\begin{lemma}\label{radical}
For each orbit other than the open orbit, the unipotent radical of $R_e$ acts trivially on $W_{e,s}$. The radical therefore acts trivially on $\mathcal E_{e,s}$. The assertion holds for both end nodes of type $D$.
\end{lemma}

\begin{proof}
Assume first that $n\ge2$. Let $L_G$ be the isotropic $e$-plane in the space of $G$ used in Proposition~\ref{stab}. The orbit coordinates identify the corresponding isotropic subspace in the space of $H$ with $L_0$ in (\ref{L0}). Let $L'$ be the subspace that projects isomorphically to both remaining nondegenerate spaces in Proposition~\ref{stab}. The central maximal isotropic space is $L_G\oplus L_0\oplus L'$.

Let $u\in U_{R_e}$. The element $u$ fixes $L_0$ pointwise. It acts trivially on $L_0^\perp/L_0$. The element $t(1,u)$ fixes $L_G$ pointwise. For $v\in L'$, one has $t(1,u)v-v\in L_0$. The element $t(1,u)$ also acts trivially on the $n-1$ blocks on which $G$ acts in (\ref{tgh}). We divide the columns of $Q_1$ according to $Q_1=E_1\oplus C_1$. After the change of basis by $w_{\mathrm{GL}}$ used in (\ref{Ce}), the image of $t(1,u)$ in the Siegel Levi has the form $n_0\!\left((0_{d_0\times e},z)\right)$, with $z\in\Mat_{d_0\times c}(F)$, and therefore lies in $D_0$. By Lemma~\ref{endpoint}, its action on $J_{V_{ni+m,m^{n-1}},\psi_e}(\Delta(\tau,m+i))$ is trivial.

For $n=1$ the inducing Speh representation is a determinant character. The image of $u$ is unipotent, so that determinant character is trivial on it. All modulus characters are trivial on $u$.

At either endpoint of type $D$, we have $c=0$. The matrix $n_0((0_{d_0\times e},z))$ is then the identity.
\end{proof}

Proposition~\ref{stab} identifies the kernel of $S_e\to R_e$ with $K:=P_e^0:=\GL_e\ltimes U_e^G$. For $b\in G_e$, we write $\ell(b):=m_{P_e}(I_e,b^{\iota_e})$. The Levi decomposition gives $P_e=K\rtimes\ell(G_e)$.

Let $\sigma$ be a smooth representation of $G$. The element $b\in G_e$ acts on $r_{U_e^G}(\sigma)$ through $\ell(b)$. On $\calE_{e,s}\otimes r_{U_e^G}(\sigma)$, the group $\GL_e$ acts diagonally, $\GL_{ni+e}$ acts on the first factor, and $G_e$ acts diagonally with the second action specified by $\ell$. Let $N_e:=(\calE_{e,s}\otimes r_{U_e^G}(\sigma))_{\GL_e}$. The group $\GL_{ni+e}\times G_e$ acts on this quotient. By Lemma~\ref{radical}, $U_{R_e}$ acts trivially on the coefficient. We regard $N_e$ as an $R_e$-module by inflation from its Levi subgroup.

\begin{theorem}\label{boundary}
For every smooth $G$-module $\sigma$ and every $e>0$, there is an $H$-isomorphism
\begin{equation}
  (J_{e,s}\otimes\sigma)_G\simeq \Ind_{R_e}^H \bigl(\calE_{e,s}\otimes r_{U_e^G}(\sigma)\bigr)_{\GL_e}. \label{Be}
\end{equation}
The factor $G_e$ acts diagonally on $\calE_{e,s}\otimes r_{U_e^G}(\sigma)$. At a split endpoint of type $D$, the formula holds for each sign, with the parabolics attached to that end node.
\end{theorem}

\begin{proof}
We identify the underlying vector spaces of $W_{e,s}$ and $\calE_{e,s}$. Let $\operatorname{pr}_e:W_{e,s}\otimes\sigma\to N_e$ be the composite of the quotient to $\sigma_{U_e^G}$ and the quotient by $\GL_e$. By Theorem~\ref{triviality}, $U_e^G$ acts trivially on $W_{e,s}$. The two normalizations are $W_{e,s}\simeq\delta_{S_e}^{1/2}\calE_{e,s}$ and $\sigma_{U_e^G}\simeq\delta_{P_e}^{1/2}r_{U_e^G}(\sigma)$. Since $\delta_{S_e}|_K=\delta_{P_e}|_K$, their product gives
\begin{equation}
  \operatorname{pr}_e\bigl(W_{e,s}(k)w\otimes\sigma(k)v\bigr) =\delta_{P_e}(k)\operatorname{pr}_e(w\otimes v),\qquad k\in K. \label{qaction}
\end{equation}

Choose a maximal compact subgroup $C_G$ such that $G=P_eC_G$. For $b\in G_e$, conjugation by $\ell(b)$ fixes $\GL_e$ pointwise and has modulus $\delta_{P_e}(\ell(b))=1$ on $U_e^G$. It therefore preserves Haar measure on $K$. We choose compatible Haar measures on $C_G$ and $G_e$. The quotient integral for $K\backslash G$ is then
\[
f\longmapsto\int_{C_G}\int_{G_e}f(\ell(b)c)\,db\,dc, \qquad f(kg)=\delta_{P_e}(k)f(g).
\]
It is defined for locally constant functions compactly supported modulo $K$. It is invariant under right translation by $G$. This is the quotient integral in Lemma~4.1 of \cite{HS17}, with relative modulus $\delta_{P_e}^{-1}$.

For $\Phi\in\ind_{S_e}^{G\times H}W_{e,s}$ and $v\in\sigma$, define
\begin{equation}
  \mathcal T_e(\Phi\otimes v)(h) :=\int_{C_G}\int_{G_e} \operatorname{pr}_e\bigl(\Phi(\ell(b)c,h)\otimes\sigma(\ell(b)c)v\bigr)\,db\,dc. \label{int2}
\end{equation}
For fixed $h$, compact support modulo $S_e$ gives compact support in $K\backslash G$. In the coordinates $g=\ell(b)c$, the integrand therefore has compact support in $G_e\times C_G$. Equation~(\ref{qaction}) gives the required transformation under left translation by $K$. Right $G$-invariance of the quotient integral implies that $\mathcal T_e$ factors through $G$-coinvariants.

Let $r\in R_e$. Write $b_0$ for its image in $G_e$. We set $s_r:=(\ell(b_0),r)\in S_e$. The defining equivariance of compact induction gives $\Phi(g,rh)=W_{e,s}(s_r)\Phi(\ell(b_0)^{-1}g,h)$. The element $\ell(b_0)$ has trivial $P_e$-modulus. Since $\delta_{S_e}(s_r)=\delta_{R_e}(r)$, the action on the coefficient satisfies
\[
  \operatorname{pr}_e\bigl(W_{e,s}(s_r)w\otimes\sigma(\ell(b_0))v\bigr) =\delta_{R_e}(r)^{1/2}N_e(r)\operatorname{pr}_e(w\otimes v).
\]
Substitution $b=b_0b'$ in (\ref{int2}) gives $\mathcal T_e(\Phi\otimes v)(rh) =\delta_{R_e}(r)^{1/2}N_e(r)\mathcal T_e(\Phi\otimes v)(h)$. The group $U_{R_e}$ acts trivially on $W_{e,s}$ by Lemma~\ref{radical}. Hence $\mathcal T_e$ takes values in $\Ind_{R_e}^H N_e$. The map $\mathcal T_e$ intertwines the right $H$-actions by construction.

The tensor identity in Lemma~2.8 and Frobenius reciprocity in Lemma~4.1 of \cite{HS17} give
\[
  (\ind_K^G W_{e,s}\otimes\sigma)_G \simeq(\delta_{P_e}^{-1}\otimes W_{e,s}\otimes\sigma)_K \simeq N_e.
\]
For the last isomorphism, we first take $U_e^G$-coinvariants. Equation~(\ref{qaction}) identifies the remaining $\GL_e$-action with the one defining $N_e$. Induction in stages through $G\times R_e$ gives $\mathcal T_e$. The compactness of $R_e\backslash H$ permits taking $G$-coinvariants on the inducing coefficient. Thus $\mathcal T_e$ is bijective. This proves (\ref{Be}).

At either endpoint of type $D$, the group $G_e$ is trivial. The integral calculation above still applies. Finite central kernels have modulus one.
\end{proof}

\subsection{The remaining Jacquet modules}\label{boundaryterms}

We use the notation of Section~\ref{coinvariants}. Assume $n\ge2$. We use the permutation matrix that puts the summands in the order $L_0,Q_1,\ldots,Q_n,P_2,\ldots,P_n$. Let $V_{d_0,q^n,p^{n-1}}$ be the unipotent radical for these blocks. Its matrices have the form
\begin{equation}
  \setlength{\arraycolsep}{5pt}
v=\begin{pmatrix}
I_{d_0}&z_1&\cdots&z_n&y_2&\cdots&y_n\\
    &I_q&u_1&\star&\star&\cdots&\star\\
    &&\ddots&u_{n-1}&\vdots&&\vdots\\
    &&&I_q&\star&\cdots&\star\\
    &&&&I_p&v_2&\star\\
    &&&&&\ddots&v_{n-1}\\
    &&&&&&I_p
  \end{pmatrix}.
  \label{block}
\end{equation}
The block sizes are $d_0,q,\ldots,q,p,\ldots,p$. There are $n$ blocks of size $q$ and $n-1$ blocks of size $p$. The entries satisfy $z_j\in\Mat_{d_0\times q}(F)$ for $1\le j\le n$. We have $y_j\in\Mat_{d_0\times p}(F)$ for $2\le j\le n$. The adjacent blocks are $u_j:=v_{Q_j,Q_{j+1}}\in\Mat_q(F)$ and $v_j:=v_{P_j,P_{j+1}}\in\Mat_p(F)$. We define
\begin{equation}
  \theta_e(v):=\psi^{-1}\!\left( \sum_{j=1}^{n-1}\tr u_j+ \sum_{j=2}^{n-1}\tr v_j\right).                    \label{theta}
\end{equation}
For $n=2$, there is one block of size $p$, so the second sum is empty.

The character is trivial on the blocks $z_j$ and $y_j$. It is also trivial on every block between a $Q_j$ space and a $P_k$ space. When $i=0$, the first block is $L_0=P_1$. We retain this block of size $e$.

Let $\psi_{q,n}$ and $\psi_{p,n-1}$ be the restrictions of $\theta_e$ to $V_{q^n}$ and $V_{p^{n-1}}$, respectively. Let $U_{(d_0,nq,(n-1)p)}$ be the standard unipotent radical in $\GL_N$ with Levi subgroup $\GL_{d_0}\times\GL_{nq}\times\GL_{(n-1)p}$. For $(\alpha,\beta,\gamma)\in\GL_{d_0}\times\GL_q\times\GL_p$, let $j'(\alpha,\beta,\gamma)$ act by $\alpha$ on $L_0$, by $\beta$ on each $Q_j$, and by $\gamma$ on each $P_j$ with $j\ge2$. In the block order of (\ref{block}), its matrix is $j'(\alpha,\beta,\gamma) :=\diag(\alpha,\beta^{\Delta_n},\gamma^{\Delta_{n-1}})$. After adjoining $D_0$, these matrices normalize the unipotent group defining $J_{V_{ni+m,m^{n-1}},\psi_e}(\Delta(\tau,m+i))$ and preserve its character. They also normalize $V_{d_0,q^n,p^{n-1}}$ and preserve $\theta_e$. Let $\pi_e$ and $\pi'_e$ be their actions on $J_{V_{ni+m,m^{n-1}},\psi_e}(\Delta(\tau,m+i))$ and $J_{V_{d_0,q^n,p^{n-1}},\theta_e}(\Delta(\tau,m+i))$, respectively. The $\GL(L_0)$-action on $J_{V_{ni+m,m^{n-1}},\psi_e}(\Delta(\tau,m+i))$ is the one in Lemma~\ref{endpoint}.

\begin{theorem}\label{coarse}
Assume $n\ge2$. There is an isomorphism $\Phi_e$
\begin{equation}
J_{V_{ni+m,m^{n-1}},\psi_e}(\Delta(\tau,m+i))\simeq J_{V_{d_0,q^n,p^{n-1}},\theta_e} (\Delta(\tau,m+i)).                              \label{Je}
\end{equation}
Jacquet transitivity gives the equivalent expression
\begin{equation}
  \begin{aligned}
    &J_{V_{ni+m,m^{n-1}},\psi_e}(\Delta(\tau,m+i))\\
    &\qquad\simeq J_{V_{q^n}\times V_{p^{n-1}}, \psi_{q,n}\boxtimes\psi_{p,n-1}} \left(J_{U_{(d_0,nq,(n-1)p)},1} (\Delta(\tau,m+i))\right).
  \end{aligned}
  \label{transitivity}
\end{equation}
The inner Jacquet module is a representation of $\GL_{d_0}\times\GL_{nq}\times\GL_{(n-1)p}$. Its $\GL_{d_0}$-action commutes with the outer Jacquet functor. The induced action agrees with the $\GL(L_0)$-action on $J_{V_{ni+m,m^{n-1}},\psi_e}(\Delta(\tau,m+i))$. For $(\alpha,\beta,\gamma)\in\GL_{d_0}\times\GL_q\times\GL_p$, the isomorphism satisfies
\begin{equation}
  \begin{aligned}
    \Phi_e\pi_e(j'(\alpha,\beta,\gamma)) &=\bigl(|\det\gamma|^q|\det\beta|^{-p}\bigr)^{R_n}\\
    &\qquad\cdot\pi'_e(j'(\alpha,\beta,\gamma))\Phi_e.
  \end{aligned}                                      \label{jacobian}
\end{equation}
\end{theorem}

\begin{proof}
By Lemmas~\ref{endpoint} and~\ref{D2D3}, we may adjoin $D_0$ and $D_2$ with trivial character. All subsequent exchanges involve only the $Q_j$ and $P_j$ blocks. They therefore commute with $\GL(L_0)$. On each intermediate unipotent group, we use $\theta_e$ for the character given by the same adjacent trace terms as (\ref{theta}). Before the exchanges, its restriction is the character $\psi_e$ of (\ref{Ce}).

We apply the root exchanges in (12.63)--(12.66) and the proof of Corollary~12.7 of \cite{GS22}, with block dimensions $q$ and $p$. Lemma~\ref{D2D3} supplies the required trivial action of $D_2$. The exchanged blocks are paired by the nondegenerate trace pairing. These exchanges leave $L_0$ fixed.

We next apply the exchanges in the proof of Proposition~12.9 of \cite{GS22}, with block dimensions $d_0,q,p$; see (12.86)--(12.87) there. The trace pairings are perfect for these block dimensions. The resulting group ${}^0V$ consists of matrices
\[
v=\begin{pmatrix}U&X\\0&S\end{pmatrix},\qquad
U\in V_{d_0,q^n},\quad S\in V_{p^{n-1}}.
\]
The block rows of $X$ correspond to $L_0,Q_1,\ldots,Q_n$. Its block columns correspond to $P_2,\ldots,P_n$. Thus $X\in\Mat_{(d_0+nq)\times(n-1)p}(F)$. All blocks of $X$ are arbitrary except $X_{Q_n,P_j}=0$ for $2\le j<n$.

To adjoin the missing blocks, we use the last paragraph of the proof of Proposition~12.9. In its notation, let ${}^1V:=\{v\in{}^0V:S=I\}$ and ${}^2V:=\{v\in{}^0V:U=I,\ X=0\}$. For $r=0,1,2$, let $\theta_r$ be the restriction of $\theta_e$ to ${}^rV$. The subgroup ${}^2V$ normalizes ${}^1V$ and preserves $\theta_1$. Since ${}^0V={}^1V\rtimes{}^2V$, Jacquet transitivity gives
\[
  J_{{}^0V,\theta_0}(\Delta(\tau,m+i))\simeq
  J_{{}^2V,\theta_2}\bigl(J_{{}^1V,\theta_1}(\Delta(\tau,m+i))\bigr).
\]
Let $Z_j\in\Mat_{q\times p}(F)$ for $2\le j\le n$. Put $Z:=(Z_2,\ldots,Z_n)\in\Mat_{q\times(n-1)p}(F)$. Group the summands as $(L_0\oplus Q_1\oplus\cdots\oplus Q_{n-1}),Q_n, (P_2\oplus\cdots\oplus P_n)$. Let $R$ be the subgroup formed by the matrices
\[
r(Z):=\begin{pmatrix}
I_{d_0+(n-1)q}&0&0\\
    0&I_q&Z\\
    0&0&I_{(n-1)p}
  \end{pmatrix}.
\]
The group $R$ normalizes ${}^1V$ and preserves $\theta_1$. Its intersection with ${}^1V$ consists of the $(Q_n,P_n)$-block, on which $\theta_1$ is trivial. We have ${}^1V R=V_{d_0,q^n,(n-1)p}$. The vanishing argument of Proposition~12.9 applies to these adjacent blocks with the bound (\ref{WFn}). It shows that $R$ acts trivially on $J_{{}^1V,\theta_1}(\Delta(\tau,m+i))$. Adjoining $R$ with trivial character and then taking the ${}^2V$-Jacquet module gives $J_{V_{d_0,q^n,p^{n-1}},\theta_e}(\Delta(\tau,m+i))$. This proves (\ref{Je}). Since $\theta_e$ is trivial on $U_{(d_0,nq,(n-1)p)}$, Jacquet transitivity gives (\ref{transitivity}).

The two sequences contribute $n-2$ and $(n-3)(n-2)/2$ root exchanges, respectively; see (12.72) and (12.85) of \cite{GS22}. Their sum is $R_n$. In each exchange, conjugation by $j'(\alpha,\beta,\gamma)$ sends $Y\in\Mat_{p\times q}(F)$ to $\gamma Y\beta^{-1}$. Its Jacobian is $|\det\gamma|^q|\det\beta|^{-p}$. Formula (12.73) of \cite{GS22} therefore gives (\ref{jacobian}).
\end{proof}

\section{The local identity}\label{generic}

We use the notation of Section~\ref{boundaryterms}.

\subsection{The generic local identity}

Assume $n\geq2$ and $e>0$. For $z\in F^\times$, put $b_e(z):=\diag(zI_e,I_{m-2e})\in\GL_{m-e}$. In the coordinates of Section~\ref{boundaryterms}, the scalar $zI_e\in\GL_e$ fixes $L_0$ pointwise. It acts by $b_e(z)$ on each $Q_j$ and by $z^{-1}I_e$ on each $P_j$ for $j\ge2$. Define $d_e(z):=\diag\left(I_{ni+e},\,b_e(z)^{\Delta_n},\, z^{-1}I_{(n-1)e}\right)$. By Theorem~\ref{coarse}, $zI_e\in\GL_e$ acts on $J_{V_{ni+m,m^{n-1}},\psi_e}(\Delta(\tau,m+i))$ through $d_e(z)$ with the additional factor $|z|^{-emR_n}$. This factor follows from (\ref{jacobian}) with $\beta=b_e(z)$ and $\gamma=z^{-1}I_e$.

Consider the unnormalized Jacquet module of $\Delta(\tau,m+i)$ for the composition $(ni+e,n(m-e),(n-1)e)$. We write its irreducible constituents as $\pi_{0,r}\boxtimes\pi_{q,r}\boxtimes\pi_{p,r}$. We retain only the constituents for which both twisted Jacquet modules in (\ref{transitivity}) are nonzero. Let $\omega_{p,r}$ be the central character of $\pi_{p,r}$. We regard $J_{V_{(m-e)^n},\psi_{m-e,n}}(\pi_{q,r})$ as a representation of $F^\times$ through $z\mapsto b_e(z)$. Let $\xi_r$ run over the characters on its irreducible subquotients. We include the choice of $\xi_r$ in the index $r$. The following proof shows that this representation of $F^\times$ is finite dimensional.

\begin{proposition}\label{spectrum}
The representation $\calE_{e,0}$ has a finite filtration stable under $\GL_e\times\GL_{ni+e}\times G_e$. On each successive quotient, $zI_e$ acts by a character. For $n\ge2$, these characters are
\begin{equation}
  \eta_{e,r}(z):=|z|^{e\alpha_e-emR_n} \xi_r(z)\omega_{p,r}(z^{-1}). \label{eta}
\end{equation}
The corresponding characters on $\calE_{e,s}$ are $\eta_{e,r}|\cdot|^{es}$. If $n=1$ and $\tau=\chi$, the list consists of
\begin{equation}
  \eta_e(z):=|z|^{e(i+e)/2}\chi(z)^e. \label{eta1}
\end{equation}
For $n=1$, we write this singleton list as $\eta_{e,1}:=\eta_e$. The assertion applies separately to the two signed endpoint orbits in type $D$.
\end{proposition}

\begin{proof}
Assume first that $n\geq2$. The ordinary Jacquet module for $(ni+e,n(m-e),(n-1)e)$ has finite length. The argument for ordinary Jacquet constituents in the proof of Lemma~\ref{derivative} applies to its middle Levi factor. It shows that every wavefront partition of $\pi_{q,r}$ has first part at most $n$. The dimension calculation in the same proof, with $m$ replaced by $m-e$, gives $\dim J_{V_{(m-e)^n},\psi_{m-e,n}}(\pi_{q,r})=1$.

We apply the two twisted Jacquet functors in (\ref{transitivity}) to a composition series of the ordinary Jacquet module. Both functors are exact. On each resulting quotient, $d_e(z)$ acts on the middle factor through $b_e(z)$. It acts on the last factor by $\omega_{p,r}(z^{-1})$. Since the middle quotient is finite dimensional, we can refine the filtration so that $b_e(z)$ acts on each quotient by a character $\xi_r(z)$. Parabolic normalization contributes $|z|^{e\alpha_e}$. The changes of measure in the root exchanges contribute $|z|^{-emR_n}$. Multiplying these factors gives (\ref{eta}).

The matrices $zI_e$ belong to the center of $\GL_e\times\GL_{ni+e}\times G_e$. The refinement can therefore be chosen stable under this group. The unipotent radical of $R_e$ acts trivially by Lemma~\ref{radical}. Hence the refinement is also stable under the full $R_e$-action. We may therefore induce this filtration to $H$.

The parameter contributes $|\det(zI_e)|^s=|z|^{es}$. For $n=1$ there are no root exchanges. The determinant character of $\Delta(\chi,m+i)$ and $\alpha_e=(i+e)/2$ give (\ref{eta1}).
\end{proof}

For a smooth representation $V$ of a locally profinite group $L$, we write $H_j(L,V)$ for smooth group homology. It is the $j$th left derived functor of coinvariants in the category of smooth complex $L$-representations. Thus $H_0(L,V)\simeq V_L$. We use $H_1$ to prove injectivity of the map induced by the open orbit inclusion.

\begin{lemma}
\label{exactness}
Fix $s$ and a smooth $G$-module $\sigma$. Suppose that, for every $e>0$, the coefficient $\calE_{e,s}\otimes r_{U_e^G}(\sigma)$ has a finite filtration stable under $\GL_e$ and the commuting $R_e$-action. Suppose that on each nonzero successive quotient, a subgroup of $Z(\GL_e)$ acts through a nontrivial character. The subgroup and its character may depend on the successive quotient. In type $D$, impose this condition on both signed endpoint terms. Then inclusion of the open orbit induces an $H$-isomorphism
\[
  (J_{0,s}\otimes\sigma)_G \xrightarrow{\ \sim\ }(J_s\otimes\sigma)_G.
\]
If $\sigma$ is admissible, restriction therefore gives an isomorphism
\begin{equation}\label{res}
  \begin{aligned}
    \Hom_{G\times H}(J_s,\sigma^\vee\boxtimes\Pi) &\longrightarrow \Hom_{G\times H}(J_{0,s},\sigma^\vee\boxtimes\Pi),\\
T&\longmapsto T|_{J_{0,s}}
  \end{aligned}
\end{equation}
for every smooth $H$-module $\Pi$.
\end{lemma}

\begin{proof}
Fix $e>0$. Put $K:=P_e^0=\GL_e\ltimes U_e^G$ and $V_e:=(W_{e,s}\otimes\sigma)|_K$. The tensor identity and induction in stages in Theorem~\ref{boundary} reduce the calculation over $G$ to $\ind_K^G V_e$. The proof of Proposition~4.4 of \cite{HS17} gives $H_j(G,\ind_K^G V_e) \simeq H_j(K,\delta_{P_e}^{-1}\otimes V_e)$. The character $\delta_{P_e}^{-1}$ is the relative modulus for $K\backslash G$. The passage between homology and the Ext formulation in that proposition is Lemma~2.15 of \cite{HS17}.

Taking $U_e^G$-coinvariants is exact, as recalled in the proof of Lemma~2.16 of \cite{HS17}. Lemma~2.18 there therefore permits taking these coinvariants before computing homology. By Theorem~\ref{triviality}, $U_e^G$ acts trivially on $W_{e,s}$. Taking these coinvariants therefore replaces $\sigma$ by $\sigma_{U_e^G}$. The identities $W_{e,s}\simeq \delta_{S_e}^{1/2}\calE_{e,s}$, $\sigma_{U_e^G}\simeq \delta_{P_e}^{1/2}r_{U_e^G}(\sigma)$, and $\delta_{S_e}=\delta_{P_e}\delta_{R_e}$ give the modulus cancellation for normalized induction to $H$: $\delta_{R_e}^{-1/2}\delta_{S_e}^{1/2} \delta_{P_e}^{1/2}\delta_{P_e}^{-1}=1$. Parabolic induction commutes with this homology calculation. Indeed, its compact realization expresses compact open invariants as finite sums of images of averaging operators. These operators commute with $G$. Smooth $G$-homology commutes with these operators and with filtered unions. We therefore obtain
\begin{equation}
H_j(G,J_{e,s}\otimes\sigma) \simeq\Ind_{R_e}^H H_j\bigl(\GL_e,\calE_{e,s}\otimes r_{U_e^G}(\sigma)\bigr), \qquad j=0,1. \label{H1}
\end{equation}
These maps commute with homomorphisms of the coefficient representations. The remaining actions are those computed in Theorem~\ref{boundary}. Thus the isomorphism is $H$-equivariant. The same calculation applies to each signed endpoint.

Let $V$ be a nonzero successive quotient of a coefficient filtration. Let $Z'\subset Z(\GL_e)$ be the subgroup specified for $V$. Let $\chi$ be its nontrivial character on $V$. Choose $z\in Z'$ with $\chi(z)\ne1$.

Choose a projective resolution $P_\bullet\to V$. Centrality makes multiplication by $z$ a homomorphism of this complex. It lifts the scalar map $\chi(z)\operatorname{id}_V$. The map $\chi(z)\operatorname{id}_{P_\bullet}$ is another lift of the same map, so the two lifts are chain homotopic. After taking $\GL_e$-coinvariants, multiplication by $z$ is the identity in every degree. Consequently, on homology, $\operatorname{id}_{H_j(\GL_e,V)} =\chi(z)\operatorname{id}_{H_j(\GL_e,V)}$. Since $\chi(z)-1\ne0$, this forces $H_j(\GL_e,V)=0$ for $j=0,1$. The long exact sequences of the coefficient filtration and (\ref{H1}) give $H_j(G,J_{e,s}\otimes\sigma)=0$ for every $e>0$ and $j=0,1$.

Put $J_{\partial,s}:=J_s/J_{0,s}$. The long exact sequences of its orbit filtration give $H_j(G,J_{\partial,s}\otimes\sigma)=0$ for $j=0,1$. In the exact sequence
\[
H_1(G,J_{\partial,s}\otimes\sigma) \longrightarrow (J_{0,s}\otimes\sigma)_G \longrightarrow (J_s\otimes\sigma)_G \longrightarrow (J_{\partial,s}\otimes\sigma)_G \longrightarrow0,
\]
the first and fourth terms vanish. The middle map is therefore an isomorphism. All maps are $H$-equivariant. For admissible $\sigma$, Lemma~\ref{homlemma} identifies precomposition with this map with the restriction in (\ref{res}).
\end{proof}

Let $\sigma$ be an admissible representation of $G$ of finite length. For $e>0$ and $z\in F^\times$, equation~(\ref{Eaction}) shows that the action of $zI_e$ on $\calE_{e,s}$ is its action on $\calE_{e,0}$ multiplied by $|z|^{es}$. Let $\{\eta_{e,r}\}_r$ be the scalar characters on the successive quotients of the filtration of $\calE_{e,0}$ in Proposition~\ref{spectrum}. Let $\Xi_e(\sigma)$ be the set of characters of $zI_e$ on the irreducible constituents of $r_{U_e^G}(\sigma)$. Define
\begin{equation}\label{badset}
  \begin{split}
    \mathcal B_e(\tau,\sigma) &:=\bigcup_{\substack{r,\ \mu\in\Xi_e(\sigma)\\
    (\eta_{e,r}\mu)|_{\mathcal O_F^\times}=1}} \left\{X\in\bbC^\times: X^e=(\eta_{e,r}(\varpi)\mu(\varpi))^{-1}\right\},\\
    \mathcal B(\tau,\sigma) &:=\bigcup_{e>0}\mathcal B_e(\tau,\sigma).
  \end{split}
\end{equation}
At the even orthogonal endpoint, we include the characters from both signed terms. There are finitely many such characters by Proposition~\ref{spectrum} and the finite length of $r_{U_e^G}(\sigma)$. Each equation for $X^e$ has at most $e$ solutions. Thus $\mathcal B(\tau,\sigma)$ is finite. The index $e$ runs over the nonopen orbits in Theorem~\ref{geom}. Enumerate the characters in $\Xi_e(\sigma)$ as $\mu_{e,k}$.

\begin{theorem}
Let $\sigma$ be an admissible representation of $G$ of finite length. If $q_F^{-s}\notin\mathcal B(\tau,\sigma)$, there is a natural $H$-equivariant isomorphism
\[
  \mathbf I_s(\tau,\sigma)\simeq \Ind_{Q_{ni}}^H \bigl(\Delta(\tau,i)|\det|^s\boxtimes\sigma^\iota\bigr).
\]
The map from the right hand side to the left hand side is induced by inclusion of the open orbit term.
\end{theorem}

\begin{proof}
Let $X:=q_F^{-s}$. By Proposition~\ref{spectrum} and finite length of $r_{U_e^G}(\sigma)$, the representation $\calE_{e,s}\otimes r_{U_e^G}(\sigma)$ has a finite filtration on whose successive quotients $zI_e$ acts by $\eta_{e,r}(z)\mu_{e,k}(z)|z|^{es}$. By (\ref{badset}), each of these characters is nontrivial when $X\notin\mathcal B(\tau,\sigma)$. Lemma~\ref{exactness} applies to these filtrations. It gives an $H$-equivariant isomorphism $(J_{0,s}\otimes\sigma)_G \xrightarrow{\ \sim\ }(J_s\otimes\sigma)_G =\mathbf I_s(\tau,\sigma)$. Proposition~\ref{integral} identifies the left hand side with the induced representation in the theorem. This proves (\ref{induced}).
\end{proof}

For every smooth $H$-module $\Pi$, precomposition with this isomorphism and (\ref{hom2}) give (\ref{hom1}). Under the open orbit identification, precomposition is the restriction map in (\ref{res}). The resulting isomorphism of $\Hom$ spaces is natural in $\Pi$. Conversely, a map of smooth $H$-modules is an isomorphism if precomposition induces a bijection on maps to every smooth $H$-module $\Pi$. Thus the two formulations are equivalent.

The condition on $s$ is sufficient. The identity may also hold when $q_F^{-s}\in\mathcal B(\tau,\sigma)$.

\begin{corollary}\label{cusp}
Let $\sigma$ be a smooth representation of $G$. Suppose that $r_{U_e^G}(\sigma)=0$ for every $P_e$ with $e>0$, including both signed endpoint terms in type $D$. Then (\ref{induced}) holds for every $s$. If $\sigma$ is admissible, (\ref{hom1}) holds as well. In particular, the assertion holds for every irreducible supercuspidal representation of $G$, provided $G$ is not the split group $\SO_2$.
\end{corollary}

\begin{proof}
The coefficients in Lemma~\ref{exactness} vanish by hypothesis. The lemma therefore identifies $(J_s\otimes\sigma)_G$ with $(J_{0,s}\otimes\sigma)_G$. Proposition~\ref{integral} now gives (\ref{induced}).

For a supercuspidal representation, Jacquet modules for proper parabolic subgroups vanish. Each $P_e$ with $e>0$ is proper unless $G$ is the split group $\SO_2$. This proves the assertion for supercuspidal representations.
\end{proof}

\begin{corollary}\label{taucusp}
Assume that $\tau$ is supercuspidal, $n\ge2$, and $n>d$. Then $J_s=J_{0,s}$ for every $s$. Consequently, (\ref{induced}) holds for every $s$ and every smooth $G$-module $\sigma$, including when $G=\SO_2$. For admissible $\sigma$, this also gives (\ref{hom1}).
\end{corollary}

\begin{proof}
The Speh representation $\Delta(\tau,m+i)$ is a subquotient of parabolic induction from $m+i$ unramified twists of $\tau$. Ordinary Jacquet functors are exact. By the geometric lemma, a nonzero Jacquet term cannot divide one of these $n$-dimensional supercuspidal blocks between two target Levi blocks. Such a division would give a proper Jacquet module of $\tau$, which is zero. Thus every target Levi block in a nonzero term has dimension divisible by $n$.

The first block of the inner Jacquet module in (\ref{transitivity}) has dimension $ni+e$. It follows that $J_{V_{ni+m,m^{n-1}},\psi_e}(\Delta(\tau,m+i))\ne0$ implies $n\mid e$. Since $1\le e\le d<n$, every nonopen coefficient is zero. The same argument applies to both signed endpoint terms in type $D$. The orbit filtration therefore gives $J_s=J_{0,s}$. Proposition~\ref{integral} proves the asserted identity.
\end{proof}

\begin{remark}\label{exception}
The identity can fail at an exceptional parameter. In the symplectic case, take $n=1$, $m=2$, $i=0$, and $\tau=\sigma=\one$. Then $\bfS=\Sp_4(F)$ and $G=H=\Sp_2(F)$. Let $\nu:=|\cdot|$. For $K=G,H$, let $B_K$ be its standard Borel subgroup and put $I_K(\chi):=\Ind_{B_K}^K\chi$, with normalized induction.

The two orbits are $e=0,1$. The stabilizer of the closed orbit is $B_G\times B_H$. Its torus elements have the form $(\diag(a,a^{-1}),\diag(b,b^{-1}))$, with $a,b\in F^\times$. Siegel normalization gives the character $|ab|^{s+3/2}$ on its Levi. Since $\delta_{B_K}^{1/2}(a)=|a|$, its contribution is $J_{1,s}\simeq I_G(\nu^{s+1/2})\boxtimes I_H(\nu^{s+1/2})$. Formula (\ref{Be}) and $r_{U_{B_G}}(\one)\simeq \nu^{-1}$ give $(J_{1,1/2})_G\simeq I_H(\nu)$. The closed orbit term is the quotient in the filtration by the two orbits. Right exactness of coinvariants therefore gives $(J_{1/2})_G\twoheadrightarrow(J_{1,1/2})_G \simeq I_H(\nu)$. The representation $I_H(\nu)$ is infinite dimensional. The right hand side of (\ref{induced}) is $\one_H$ in this example. Hence (\ref{induced}) fails at $s=1/2$.
\end{remark}

\subsection{The metaplectic case}

In this subsection, we assume that the groups are symplectic. Put $\widetilde{\bfS}:=\widetilde\Sp_{2n(m+i)}(F)$, $\widetilde G:=\widetilde\Sp_m(F)$, and $\widetilde H:=\widetilde\Sp_{2ni+m}(F)$. We use the normalized double covers, canonical unipotent splittings, and genuine inducing notation from Section~1 of \cite{GS21}. Let $\mu_2:=\{\pm1\}$ be the central kernel of each cover. We write cover elements as $(g,\epsilon)$ in the normalized cocycle model. A representation is genuine if $(1,\epsilon)$ acts by the scalar $\epsilon$.

Let $\widetilde Q_{n(m+i)}$ be the full inverse image of the Siegel parabolic $Q_{n(m+i)}^{\bfS}$ in $\widetilde{\bfS}$. Similarly, let $\widetilde Q_{ni}$ be the full inverse image of $Q_{ni}^H$ in $\widetilde H$.

We denote the quadratic Hilbert symbol by $(\ ,\ )_F$. Let $x:\Sp_r(F)\to F^\times/F^{\times2}$ be the square class function used in the normalized cocycle of Section~1 of \cite{GS21}. We suppress the even dimension $r$, which is determined by the argument. For a Siegel parabolic element with Levi component $a$, its value is $\det a$ modulo $F^{\times2}$. In the coordinates of (1.28)--(1.29) of \cite{GS21}, the stabilizer map is $t^{(2)}((g,\epsilon),(h,\epsilon')) :=\bigl(t(g,h),\epsilon\epsilon'(x(g),x(h))_F\bigr)$. In particular, $t^{(2)}((1,\epsilon),(1,\epsilon'))=(1,\epsilon\epsilon')$. We use the lift of $\iota$ in (1.39) of \cite{GS21}. In these conventions, $\widetilde\boxtimes$ denotes the covered tensor product. The diagonal $\mu_2$ acts trivially on the tensor product of the two genuine Levi factors, since its action is $\epsilon^2=1$.

Let $\gamma_\psi$ be the normalized Weil factor. For elements $p,p'$ of a Siegel parabolic with Levi components $a,a'$, the normalized cocycle gives $(p,\epsilon)(p',\epsilon') =(pp',\epsilon\epsilon'(\det a,\det a')_F)$. The canonical unipotent splitting is included in this formula. In particular, our Siegel inducing representation $\Delta(\tau,m+i)\gamma_\psi|\det|^s$ has action $(a,\epsilon)\longmapsto \epsilon\gamma_\psi(\det a)|\det a|^s\Delta(\tau,m+i)(a)$ for $a\in\GL_{n(m+i)}(F)$. Let $\sigma$ be a smooth genuine representation of $\widetilde G$. We set $\sigma^\iota(\widetilde g):=\sigma(\widetilde g^\iota)$. For $a\in\GL_{ni}$ and $\widetilde b=(b,\epsilon)\in\widetilde G$, put $m(a,\widetilde b):=\bigl(\diag(a,b,a^*),\epsilon\bigr)$. In the conventions of (3.48)--(3.49) of \cite{GS21}, the representation $\Delta(\tau,i)\gamma_\psi|\det|^s \widetilde\boxtimes\sigma^\iota$ has action
\begin{equation}
m(a,\widetilde b)\longmapsto \gamma_\psi(\det a)(\det a,x(b))_F|\det a|^s \Delta(\tau,i)(a)\otimes\sigma^\iota(\widetilde b). \label{metaction}
\end{equation}
For $i=0$, the general linear factor is trivial, and this action is $\sigma^\iota$. Put
\[
  \begin{aligned}
    \widetilde\rho_s&:=\Ind_{\widetilde Q_{n(m+i)}}^{\widetilde{\bfS}}(\Delta(\tau,m+i)\gamma_\psi|\det|^s),\\
    \widetilde J_s&:=J_{U,\psi_U}(\widetilde\rho_s).
  \end{aligned}
\]

\paragraph{The Jacquet modules.}

The canonical splittings lift the unipotent subgroups in the linear calculation to the cover. These splittings are described in the paragraph preceding (1.5) of \cite{GS21}. They identify the root exchange maps with those in the linear case.

The covered parabolic subgroups are full inverse images of the linear parabolic subgroups. Hence the orbit decomposition is unchanged. The geometric lemma gives a filtration indexed as in Theorem~\ref{geom}. We denote its successive quotients by $\widetilde J_{e,s}$. Its first term $\widetilde J_{0,s}$ corresponds to the open orbit. Choose quotient measures compatible with projection to the linear groups. The Jacobians then agree with those in the linear case.

For $e>0$, let $\widetilde P_e$ and $\widetilde R_e$ be the full inverse images of $P_e$ and $R_e$ in $\widetilde G$ and $\widetilde H$, respectively. Let $\widetilde\GL_e$ be the inverse image of the standard $\GL_e$ factor of the Levi of $P_e$. For Jacquet modules of $\sigma$ along $U_e^G$, we use the canonical splitting of $U_e^G$ and the normalized Jacquet action. For $\widetilde p\in\widetilde P_e$ and $\widetilde r\in\widetilde R_e$, write $p$ and $r$ for their images. Put $\widehat S_e:=\{(\widetilde p,\widetilde r):(p,r)\in S_e\}$ and $\mu_2^{(1)}:=\{((1,\epsilon),(1,1)):\epsilon=\pm1\}$. Pullback by $t^{(2)}$ makes $\widetilde J_s$ genuine in each factor. Since $\sigma$ is genuine, the first central sign acts trivially on $\widetilde J_s\otimes\sigma$. Thus the action of $\widetilde G$ descends to $G=\Sp_m(F)$. The action of $\widetilde H$ remains genuine. Define $\overline S_e:=\widehat S_e/\mu_2^{(1)} \simeq \{(p,\widetilde r):(p,r)\in S_e\}$. Put $P_e^0:=\GL_e\ltimes U_e^G$. Projection to the second factor gives the exact sequence
\begin{equation}
  1\longrightarrow P_e^0 \longrightarrow\overline S_e \longrightarrow\widetilde R_e\longrightarrow1. \label{Stilde}
\end{equation}

Fix $e>0$. Let $Q:=Q_{n(m+i)}^{\bfS}$, and let $\widetilde Q$ be its full inverse image in $\widetilde{\bfS}$. Set $\lambda_s:=\delta_Q^{1/2} \Delta(\tau,m+i)\gamma_\psi|\det|^s$. We regard $\lambda_s$ as a genuine representation of $\widetilde Q$ by inflation from its Levi subgroup.

Choose a lift $\widetilde w_e$ of $w_e$. On $\widetilde w_e^{-1}\widetilde Q\widetilde w_e$, define $({}^{\widetilde w_e}\lambda_s)(x) :=\lambda_s(\widetilde w_e x\widetilde w_e^{-1})$. Put $U_{e,0}:=U\cap w_e^{-1}Qw_e$, as in Proposition~\ref{fibre}. The canonical splitting identifies $U_{e,0}$ with a subgroup of $\widetilde{\bfS}$. For $\widehat s\in\widehat S_e$, the element $t^{(2)}(\widehat s)$ lies in $\widetilde w_e^{-1}\widetilde Q\widetilde w_e$. It normalizes $U_{e,0}$ and preserves $\psi_U|_{U_{e,0}}$. Thus $\widehat S_e$ acts on $J_{U_{e,0},\psi_U|_{U_{e,0}}}({}^{\widetilde w_e}\lambda_s)$. The change of variables contributes the same factor $\mathfrak j_e$ as in (\ref{measure}). We regard this factor and the modulus characters as characters of the covering groups through their projections to the linear groups.

\begin{proposition}\label{fibre2}
The unnormalized and normalized covered coefficients are
\begin{equation}
  \begin{split}
    \widetilde W_{e,s} &:=\mathfrak j_e\, J_{U_{e,0},\psi_U|_{U_{e,0}}}({}^{\widetilde w_e}\lambda_s),\\
    \widetilde\calE_{e,s}&:=\delta_{S_e}^{-1/2}\widetilde W_{e,s}.
  \end{split}
  \label{metfibre}
\end{equation}
Their $\widehat S_e$-actions are independent of the lift $\widetilde w_e$. They depend on $s$ only through $|\det|^s$.
\end{proposition}

\begin{proof}
With the canonical unipotent splittings, the coefficient calculation in Proposition~\ref{fibre} applies to $\lambda_s$. This gives (\ref{metfibre}) with the $\widehat S_e$-action defined above. The two lifts of $w_e$ differ by a central sign. They therefore induce the same conjugation action. Only the factor $|\det|^s$ depends on $s$.
\end{proof}

\begin{lemma}\label{metboundary}
Let $\sigma$ be a smooth genuine representation of $\widetilde G$. The first central sign acts trivially on $\widetilde W_{e,s}\otimes\sigma$. Let $\overline W_{e,s}(\sigma)$ be the resulting representation of $\overline S_e$. Consider the tensor product
\begin{equation}
  \widetilde\calE_{e,s}\otimes r_{U_e^G}(\sigma). \label{Bes}
\end{equation}
The central sign of the diagonal $\widetilde\GL_e$ acts trivially on this tensor product. Thus the action descends to $\GL_e$. The action of $\widetilde R_e$ is genuine. There are isomorphisms
\begin{align}
  \widetilde J_{e,s}\otimes\sigma &\simeq\ind_{\overline S_e}^{G\times\widetilde H} \overline W_{e,s}(\sigma), \label{cind}\\
  (\widetilde J_{e,s}\otimes\sigma)_G &\simeq\Ind_{\widetilde R_e}^{\widetilde H} \bigl(\widetilde\calE_{e,s}\otimes r_{U_e^G}(\sigma)\bigr)_{\GL_e}. \label{metcoinv}
\end{align}
\end{lemma}

\begin{proof}
We tensor the model for the $e$-th orbit with $\sigma$. The first central sign acts trivially on $\widetilde W_{e,s}\otimes\sigma$. Passing to the quotient by $\mu_2^{(1)}$ gives (\ref{cind}).

The kernel of (\ref{Stilde}) is $P_e^0$, as in the linear calculation. The central sign acts trivially on the tensor product of the two representations of $\widetilde\GL_e$. The tensor product therefore descends to $\GL_e$. The canonical splittings carry Theorem~\ref{triviality} to the covered coefficient, so $U_e^G$ acts trivially on $\widetilde W_{e,s}$. The integral (\ref{int2}) and the modulus calculation in the proof of Theorem~\ref{boundary} give (\ref{metcoinv}).
\end{proof}

\begin{theorem}\label{layers2}
The module $\widetilde J_s$ has a finite orbit filtration indexed as in the linear symplectic case. Let $\widetilde J_{0,s}$ be the term corresponding to the open orbit. For every smooth genuine $\widetilde G$-module $\sigma$, we have
\[
  (\widetilde J_{0,s}\otimes\sigma)_G \simeq\Ind_{\widetilde Q_{ni}}^{\widetilde H} \bigl(\Delta(\tau,i)\gamma_\psi|\det|^s \widetilde\boxtimes\sigma^\iota\bigr).
\]
For $e>0$, the coinvariants $(\widetilde J_{e,s}\otimes\sigma)_G$ are given by (\ref{metcoinv}).
\end{theorem}

\begin{proof}
The orbit filtration follows from the geometric lemma and Proposition~\ref{fibre2}. The two genuine central actions cancel on the open term. Proposition~\ref{integral} therefore applies with the Levi action (\ref{metaction}), as in (3.48) of \cite{GS21}. Lemma~\ref{metboundary} gives the other terms.
\end{proof}

\paragraph{Proof of the identity.}

For $e>0$, we consider the square scalar lifts $\widetilde{\mathbf z}_e(t):=(t^2I_e,1)$, with $t\in F^\times$. To compute their action, we first conjugate them by the orbit representative. Put $c_e:=m-2e$ and, for $z\in F^\times$, set $g_e(z):=\diag(zI_e,I_{c_e},z^{-1}I_e)$, $A_e(z):=t(g_e(z),1)$, and $B_e(z):=w_eA_e(z)w_e^{-1}$. Both $A_e(z)$ and $B_e(z)$ lie in the ambient Siegel parabolic. The matrix $A_e(z)$ belongs to its Levi subgroup. The determinant of its Levi component is $z^e$. The determinant of the Levi component of $B_e(z)$ is also $z^e$, by (\ref{detell}). For a lift $\widetilde w_e$, let $\epsilon_e(t)\in\mu_2$ denote the central sign of $\widetilde w_e(A_e(t),1)\widetilde w_e^{-1}$. Then $\widetilde w_e(A_e(t),1)\widetilde w_e^{-1} =(B_e(t),\epsilon_e(t))$. On squaring this equality, the multiplication on each side contributes the same Hilbert symbol $(t^e,t^e)_F$. Also $\epsilon_e(t)^2=1$. Cancelling these factors gives $\widetilde w_e(A_e(t^2),1)\widetilde w_e^{-1} =(B_e(t^2),1)$. Thus conjugation introduces no sign on the square scalar lifts. Their determinants are squares, so $\gamma_\psi(t^{2e})=1$. The mixed Hilbert symbols with the other Levi factors are also one. The lifts therefore commute with all the remaining Levi actions and satisfy $\widetilde{\mathbf z}_e(t)\widetilde{\mathbf z}_e(u) =\widetilde{\mathbf z}_e(tu)$. The inducing character $|\det|^s$ has value $|\det(t^2I_e)|^s=|t|^{2es}$.

The proof of Proposition~\ref{spectrum} gives a finite filtration of $\widetilde\calE_{e,0}$ on whose successive quotients $\widetilde{\mathbf z}_e(t)$ acts by a character $\eta_{e,r}^\square(t)$ of $F^\times$. The middle factors give the same finite dimensional quotients as in that proof. All additional cocycle and Weil factors are trivial on $\widetilde{\mathbf z}_e(t)$. The Levi actions commute with these central lifts. Hence the filtration can be chosen stable under the Levi actions.

Assume that $\sigma$ is admissible of finite length. Let $\Xi_{e,\sigma}^\square$ be the finite set of characters by which $\widetilde{\mathbf z}_e(t)$ acts on the irreducible constituents of $r_{U_e^G}(\sigma)$. Define
\begin{equation}
  \mathcal B_e^{(2)}(\tau,\sigma):= \bigcup_{\substack{r,\,\mu\in\Xi_{e,\sigma}^\square\\
  (\eta_{e,r}^\square\mu)|_{\mathcal O_F^\times}=1}} \left\{X\in\bbC^\times: X^{2e}=(\eta_{e,r}^\square(\varpi)\mu(\varpi))^{-1}\right\}. \label{square}
\end{equation}
The character lists are finite. Each equation for $X^{2e}$ has at most $2e$ solutions. Thus $\mathcal B_e^{(2)}(\tau,\sigma)$ is finite.

\begin{theorem}
\label{metmain}
Let $\sigma$ be an admissible genuine representation of $\widetilde G$ of finite length. If $q_F^{-s}\notin\bigcup_{e>0}\mathcal B_e^{(2)}(\tau,\sigma)$, there is a natural $\widetilde H$-equivariant isomorphism
\begin{equation}\label{metind}
  (\widetilde J_s\otimes\sigma)_G \simeq\Ind_{\widetilde Q_{ni}}^{\widetilde H} \bigl(\Delta(\tau,i)\gamma_\psi|\det|^s \widetilde\boxtimes\sigma^\iota\bigr).
\end{equation}
The map from the right hand side to the left hand side is induced by inclusion of $\widetilde J_{0,s}$.
\end{theorem}

\begin{proof}
The coefficient (\ref{Bes}) has a finite filtration on whose successive quotients $\widetilde{\mathbf z}_e(t)$ acts by $\eta_{e,r}^\square(t)\mu(t)|t|^{2es}$, with $\mu\in\Xi_{e,\sigma}^\square$. By (\ref{square}), these characters are nontrivial under the hypothesis of the theorem. We apply the homology calculation in the proof of Lemma~\ref{exactness} to (\ref{cind}). The coefficient is identified in Lemma~\ref{metboundary}. We obtain
\[
  \begin{aligned}
    &H_j(G,\widetilde J_{e,s}\otimes\sigma)\\
    &\qquad\simeq\Ind_{\widetilde R_e}^{\widetilde H} H_j\bigl(\GL_e,\widetilde\calE_{e,s}\otimes r_{U_e^G}(\sigma)\bigr), \qquad j=0,1.
  \end{aligned}
\]
The right hand side vanishes by the central character argument in Lemma~\ref{exactness}. The long exact sequences used in that lemma therefore give $(\widetilde J_{0,s}\otimes\sigma)_G \xrightarrow{\ \sim\ }(\widetilde J_s\otimes\sigma)_G$. Theorem~\ref{layers2} identifies the left hand side with the induced representation in (\ref{metind}).
\end{proof}

The first central sign acts trivially on $\widetilde J_s\otimes\sigma$. Its $G$-coinvariants therefore agree with its $\widetilde G$-coinvariants. Lemma~\ref{homlemma} gives the equivalent identity
\[
  \begin{aligned}
    &\Hom_{\widetilde G\times\widetilde H} (\widetilde J_s,\sigma^\vee\boxtimes\Pi)\\
    &\quad\simeq\Hom_{\widetilde H}\!\left( \Ind_{\widetilde Q_{ni}}^{\widetilde H} \bigl(\Delta(\tau,i)\gamma_\psi|\det|^s \widetilde\boxtimes\sigma^\iota\bigr),\Pi\right)
  \end{aligned}
\]
for every smooth genuine $\widetilde H$-module $\Pi$. The isomorphism is natural in $\Pi$. It is given by restriction to $\widetilde J_{0,s}$ and the open orbit calculation.

If $r_{U_e^G}(\sigma)=0$ for every $e>0$, the same argument gives (\ref{metind}) for every $s$.

\section{Twisted Jacquet modules over finite fields}\label{Fq}

Let $k:=\mathbf F_q$ have odd characteristic $p$. In this section, all representations are finite dimensional over $\bbC$. Group symbols denote groups of $k$-points. Induction is ordinary induction of finite groups.

Let $\pi$ be a representation of a finite group. Let $V$ be a subgroup and $\chi$ a character of $V$. Write $J_{V,\chi}(\pi)$ for the quotient by the span of $\pi(v)x-\chi(v)x$, for $v\in V$ and $x\in\pi$. The averaging operator $e_\chi:=\frac1{|V|}\sum_{v\in V}\chi(v)^{-1}\pi(v)$ identifies this quotient with the $\chi$-isotypic subspace. Since $e_\chi$ is an idempotent, this quotient functor is exact.

Fix either the symplectic family $\mathcal H_r:=\Sp_r$ or the split special orthogonal family $\mathcal H_r:=\SO_r^{\mathrm{split}}$. Fix integers $m\geq2$, $n\geq1$, and $i\geq0$. Assume that $m$ is even in the symplectic case. Put $d:=\lfloor m/2\rfloor$ and $N:=n(m+i)$, and assume $n>d$. In particular, $n\geq2$. Put $\bfS:=\mathcal H_{2N}(k)$, $G:=\mathcal H_m(k)$, and $H_i:=\mathcal H_{2ni+m}(k)$. For a group $K$ in the chosen family, let $Q_a^K$ be its standard parabolic with general linear Levi factor $\GL_a(k)$. Put $Q_0^K:=K$. Let $Q_N:=Q_N^{\bfS}$ be the Siegel parabolic of $\bfS$. In the orthogonal case, $Q_N$ selects one of the two orbits of maximal isotropic subspaces under $\bfS$. We retain the signed endpoint parabolics of Proposition~\ref{evenSO}. Set $\GL_0(k):=\{1\}$. Fix a nontrivial additive character $\psi$ of $k$. The matrices for the chosen family in Section~\ref{matrices}, now taken over $k$, define $U$, $\psi_U$, and $t:G\times H_i\longrightarrow\bfS$.

The Levi embeddings are given by (\ref{Levi}) over $k$. In odd orthogonal dimension, the source group preserves $w_m$ and the target group preserves $w'_{2ni+m}$. Their residual factors are identified by $b\mapsto K_mbK_m^{-1}$. The identity ${}^tK_mw'_mK_m=-w_m$ holds over $k$, since $p$ is odd. Let $C_m:=J_0^{-1}K_m$ be as in (\ref{iota}). For a representation $\sigma$ of $G$, put $b^\iota:=C_mbC_m^{-1}$ and $\sigma^\iota(b):=\sigma(b^\iota)$. The entries of $J_0$ and $K_m$ are $0$, $1$, and $-1$. These matrices therefore give the same residual identification and stabilizer relation over $k$.

Let $\tau$ be an irreducible cuspidal representation of $\GL_n(k)$. For $c\geq1$, let $\Delta(\tau,c)$ be the finite Speh representation of $\GL_{nc}(k)$. It is the constituent of smallest degree in the parabolic induction of $c$ copies of $\tau$. Its partition parameter is $(1^c)$ in the convention of Section~2.4 of \cite{CZ26}. Put $\Delta(\tau,0):=\one$. Define
\[
  \rho_i:=\Ind_{Q_N}^{\bfS}\Delta(\tau,m+i),\qquad J_i:=J_{U,\psi_U}(\rho_i).
\]
The action on $J_i$ is through $t$. For a representation $\sigma$ of $G$, put
\begin{equation}\label{rhoi}
I_i(\sigma):=\Ind_{Q_{ni}^{H_i}}^{H_i} \bigl(\Delta(\tau,i)\boxtimes\sigma^\iota\bigr).
\end{equation}
The inducing representations in these formulas are inflated from the indicated Levi subgroups. For $i=0$, formula~(\ref{rhoi}) means $I_0(\sigma)\simeq\sigma^\iota$ under the fixed identification $H_0\simeq G$.

\subsection{The identity}

\begin{theorem}\label{finitecase}
Assume $p>2n(m+i)$. For every representation $\sigma$ of $G$, there is an isomorphism of $H_i$-representations
\begin{equation}\label{Ji}
  \Hom_G(\sigma^\vee,J_i)\simeq I_i(\sigma).
\end{equation}
The action of $h\in H_i$ on the left is $(h\cdot T)(v):=h\cdot T(v)$.
\end{theorem}

Finite dimensionality identifies $\Hom_G(\sigma^\vee,J_i)$ with $(J_i\otimes\sigma)^G$. Averaging over $G$ identifies the latter with $(J_i\otimes\sigma)_G$. Thus the theorem is equivalent to
\begin{equation}\label{coinv}
  (J_i\otimes\sigma)_G\simeq I_i(\sigma).
\end{equation}
It also gives, naturally in every $H_i$-module $\Pi$,
\[
  \Hom_{G\times H_i}(J_i,\sigma^\vee\boxtimes\Pi) \simeq\Hom_{H_i}(I_i(\sigma),\Pi)
\]
by Lemma~\ref{homlemma}.

We use the following two Jacquet module formulas for the finite Speh representation. Write $J_V:=J_{V,1}$. Let $a_1,\ldots,a_r$ be positive integers with $a_1+\cdots+a_r=nc$. If every $a_j$ is divisible by $n$, put $c_j:=a_j/n$. Then
\begin{equation}\label{HC}
J_{V_{a_1,\ldots,a_r}}\Delta(\tau,c)\simeq
  \begin{cases}
    \Delta(\tau,c_1)\boxtimes\cdots\boxtimes\Delta(\tau,c_r), &a_j=nc_j\text{ for all }j,\\
    0,    &n\nmid a_j\text{ for some }j.
  \end{cases}
\end{equation}
Cuspidality makes this Jacquet module zero unless every $a_j$ is divisible by $n$. When every $a_j$ is divisible by $n$, Theorems~2.2 and~2.3 of \cite{AT17}, together with Frobenius reciprocity, identify its multiplicities with the Littlewood--Richardson coefficients for the column $(1^c)$. Only the columns $(1^{c_1}),\ldots,(1^{c_r})$ occur, with multiplicity one.

For $a,r\geq1$ and $v\in V_{a^r}$, define $\psi_{a,r}(v):=\psi^{-1}\left(\sum_{j=1}^{r-1}\tr v_{j,j+1}\right)$. Let $\omega_\tau$ be the central character of $\tau$. The rectangular model in Theorem~2.1 and Section~2.4 of \cite{CZ26} gives
\begin{equation}\label{rect}
J_{V_{a^n},\psi_{a,n}}\Delta(\tau,a) \simeq\omega_\tau\circ\det
\end{equation}
as a representation of the diagonal $\GL_a(k)$. For $i>0$, extend $\psi_{m,n}$ to $V_{ni,m^n}$ by making the row corresponding to the block of size $ni$ trivial. For $i=0$, the character remains $\psi_{m,n}$ on $V_{m^n}$. For $i>0$, we first take the ordinary Jacquet module for the blocks $(ni,nm)$. Formula~(\ref{HC}) gives $\Delta(\tau,i)\boxtimes\Delta(\tau,m)$. Taking the rectangular Jacquet module in the second factor gives the formula below. For $i=0$, it follows directly from (\ref{rect}).
\begin{equation}\label{open0}
J_{V_{ni,m^n},\psi_{m,n}}\Delta(\tau,m+i) \simeq\Delta(\tau,i)\boxtimes\one_G.
\end{equation}
The diagonal $G$ acts trivially because $\det g=1$ for $g\in G$.

\begin{proof}[Proof of Theorem~\ref{finitecase}]
The finite Mackey decomposition for $Q_N\backslash\bfS/(t(G\times H_i)U)$ gives a direct sum of spaces of functions supported on double cosets. Taking $J_{U,\psi_U}$ preserves this direct sum. The matrix reductions in Lemma~\ref{intersection} are valid over $k$. Witt extension holds over $k$; see Theorem~7.4 of \cite{T92}. In the symplectic case, nondegenerate alternating spaces of a fixed dimension are isometric over $k$. The $G\times H_i$-orbits on the central maximal isotropic subspaces are indexed by $e=0,\ldots,d$.

For the orthogonal groups, let $W$ and $Y$ be the defining quadratic spaces of $G$ and $H_i$. For a maximal isotropic subspace $X$ of $W\perp Y$, put $E:=X\cap W$ and $F:=X\cap Y$. Take the orthogonal complements $E^\perp$ in $W$ and $F^\perp$ in $Y$. Put $e:=\dim E$ and $f:=\dim F$. The projections of $X$ into $W$ and $Y$ lie in $E^\perp$ and $F^\perp$. Their kernels are $F$ and $E$, respectively. Since $\dim X=m+ni$, this gives $m+ni-f\leq m-e$ and $m+ni-e\leq m+2ni-f$. Hence $f=ni+e$. Both projections are therefore onto $E^\perp$ and $F^\perp$, respectively. After quotienting by $E\oplus F$, both projections are isomorphisms. The isotropy of $X$ therefore identifies $X/(E\oplus F)$ with the graph of an anti-isometry from $E^\perp/E$ to $F^\perp/F$. Both quotient spaces have dimension $c:=m-2e$. For even $c$, their split forms are anti-isometric by $\diag(I_{c/2},-I_{c/2})$. For odd $c$, the matrix $K_c$ is an anti-isometry between their forms. Thus these graphs exist over $k$. Witt extension makes the product of the full orthogonal groups transitive for each $e$.

The determinant argument in the proof of Proposition~\ref{evenSO} applies over $k$, since $p$ is odd. In the chosen ambient orbit, it gives one $G\times H_i$-orbit when $c>0$ and two signed orbits when $c=0$. The latter case occurs only when $m=2d$ and $e=d$. In particular, the open orbit is unique in all three families.

The inducing representation is trivial on the rectangular unipotent subgroups used in Lemma~\ref{intersection}. The same calculation of the character gives the conditions $\kappa_j=0$ and $\kappa_\infty=0$. The other exclusions in Proposition~\ref{rzero} use a nilpotent matrix whose $n$-th power is nonzero. Lemma~\ref{compare} supplies the required finite field vanishing. Let $J_{i,e}$ denote the remaining contribution indexed by $e$. In the even orthogonal case, $J_{i,d}$ is replaced by the two terms $J_{i,d,+}$ and $J_{i,d,-}$. Propositions~\ref{stab} and \ref{fibre} give
\[
J_i\simeq
  \begin{cases}
    \displaystyle\bigoplus_{e=0}^d J_{i,e}, &G=\Sp_{2d}(k)\text{ or }\SO_{2d+1}^{\mathrm{split}}(k),\\[6pt]
    \displaystyle\bigoplus_{e=0}^{d-1}J_{i,e}\oplus J_{i,d,+}\oplus J_{i,d,-}, &G=\SO_{2d}^{\mathrm{split}}(k).
  \end{cases}
\]
The matrix calculations in these propositions hold over $k$ with the fixed residual Levi identifications. All modulus characters are trivial. The coefficient of $J_{i,e}$ is the twisted Jacquet module in (\ref{Ce}), with the group and character defined over $k$ as follows. Consider the decomposition $k^N\simeq W_0\oplus M_1\oplus\cdots\oplus M_n$, where $W_0:=k^{ni}$ and each $M_j:=k^m$ has its standard basis. For $v\in V_{ni+m,m^{n-1}}$, let $x_j:M_{j+1}\to M_j$ be its indicated block, for $1\leq j<n$. Put $a_e:=\diag(I_{m-e},0_e)$ and $\psi_e(v):=\psi^{-1}\left(\tr(a_ex_1)+ \sum_{j=2}^{n-1}\tr x_j\right)$. The character is trivial on the block with row $W_0$ and column $M_2$. The coefficient is $J_{V_{ni+m,m^{n-1}}(k),\psi_e}(\Delta(\tau,m+i))$, as in (\ref{Ce}). If this coefficient is zero, then $J_{i,e}=0$. At the even orthogonal endpoint, both signed terms use $J_{V_{ni+m,m^{n-1}}(k),\psi_d}(\Delta(\tau,m+i))$, as in Proposition~\ref{evenSO}.

Let $e>0$. Put $a:=m-e$, $b:=e$, and $a_0:=ni+e$. These replace the block dimensions $q$, $p$, and $d_0$, respectively, in Theorem~\ref{coarse}. The outer Jacquet functor below acts on the last two factors of $\GL_{a_0}(k)\times\GL_{na}(k)\times\GL_{(n-1)b}(k)$. The root exchanges in the proof of Theorem~\ref{coarse} give
\begin{equation}\label{finiteJe}
  \begin{aligned}
    &J_{V_{ni+m,m^{n-1}}(k),\psi_e}(\Delta(\tau,m+i))\\
    &\qquad\simeq J_{V_{a^n}\times V_{b^{n-1}},\,\psi_{a,n}\boxtimes\psi_{b,n-1}} \left(J_{V_{a_0,na,(n-1)b}}\Delta(\tau,m+i)\right).
  \end{aligned}
\end{equation}
Over $k$, each root exchange pairs two rectangular matrix spaces $\mathcal X$ and $\mathcal Y$ by $(X,Y)\mapsto\psi(\tr(XY))$. The trace pairing is nondegenerate, so
\[
  \sum_{Y\in\mathcal Y}\psi(\tr(XY))
  =\begin{cases}|\mathcal Y|,&X=0,\\0,&X\ne0.\end{cases}
\]
The two finite summation operators in opposite directions therefore compose to a nonzero scalar multiple of the identity. They give the required exchange isomorphism.

For $D_0$ and $D_2$, the matrix calculations in Lemma~\ref{endpoint} and the $D_2$ part of Lemma~\ref{D2D3} apply, respectively. Their nontrivial characters give nilpotent matrices with nonzero $n$-th power. Lemma~\ref{compare} gives the required vanishing over $k$. Finite character decomposition therefore makes $D_0$ and $D_2$ act trivially.

Finally, use the groups ${}^1V$, ${}^2V$, and $R$ in the proof of Theorem~\ref{coarse}, over $k$ and with block dimensions $a_0,a,b$. Before taking the ${}^2V$-Jacquet module, adjoin the whole row $R$. An extending nontrivial character is represented by $0\ne A\in\Mat_{(n-1)b\times a}(k)$ whose last $b$ rows are zero. The enlarged group is $V_{a_0,a^n,(n-1)b}$. The character is supported on adjacent blocks. Let $N_A$ be the nilpotent matrix representing this character. Its adjacent identity blocks give $N_A^n\ne0$. Lemma~\ref{compare} makes every such Jacquet module zero. Finite character decomposition makes $R$ act trivially before taking the ${}^2V$-Jacquet module. Every change of variables is a bijection of finite matrix spaces. The scalar factors from summation are nonzero. The resulting block group and character are those in (\ref{transitivity}), with the indicated block dimensions. This proves (\ref{finiteJe}).

If $J_{V_{ni+m,m^{n-1}}(k),\psi_e}(\Delta(\tau,m+i))\ne0$, formula~(\ref{HC}) forces $n\mid a_0$. Since $a_0=ni+e$, it follows that $n\mid e$. No integer $e$ with $1\leq e\leq d$ has this property because $n>d$. The same argument makes both signed terms zero at the even orthogonal endpoint. Thus every contribution outside the open orbit vanishes as a $G\times H_i$-representation. Hence $J_i=J_{i,0}$.

For the open contribution, put $Q:=Q_{ni}^{H_i}$. Let $U_Q$ be its unipotent radical. For $a\in\GL_{ni}(k)$ and $b\in G$, write $m_Q(a,b):=\diag(a,K_mbK_m^{-1},a^*)$. The star operation is that of Section~\ref{matrices}. The stabilizer of the open representative is
\[
S_0=\{(b^\iota,m_Q(a,b)u): a\in\GL_{ni}(k),\ b\in G,\ u\in U_Q\}.
\]
In the coefficient $J_{V_{ni+m,m^{n-1}}(k),\psi_0}(\Delta(\tau,m+i))$, the first block is $W_0\oplus M_1$, with $\dim W_0=ni$ and $\dim M_1=m$. We adjoin the missing block $\Hom(M_1,W_0)$ as in Proposition~\ref{openSpeh}. A nonzero character on that block is represented by a nonzero map $B:W_0\to M_1$. Let $N_B\in\mathfrak{gl}_N(k)$ have block $B:W_0\to M_1$ and identity blocks $M_j\to M_{j+1}$ for $1\leq j<n$. All other blocks are zero. The restriction $N_B^n|_{W_0}:W_0\to M_n$ has matrix $B$ in the chosen bases. Hence $N_B^n\ne0$. Lemma~\ref{compare} excludes every such character. Thus the missing block acts trivially. For $i=0$ this block is absent.

Formula~(\ref{open0}) identifies the resulting coefficient with $\Delta(\tau,i)$. The image of $U_Q$ lies in the unipotent group just adjoined, with trivial character. Define the representation $\eta_i$ of $S_0$ by $\eta_i(b^\iota,m_Q(a,b)u):=\Delta(\tau,i)(a)$. Consequently, $J_{i,0}\simeq\Ind_{S_0}^{G\times H_i}\eta_i$. The proof of Proposition~\ref{integral}, with integration replaced by summation and all modulus characters equal to one, gives (\ref{coinv}). The averaging identification above then gives (\ref{Ji}).
\end{proof}

\subsection{Irreducible constituents}

Let $\operatorname{Irr}(K)$ denote the isomorphism classes of irreducible complex representations of a finite group $K$. Proposition~\ref{constituents} describes the irreducible constituents of (\ref{rhoi}) for every odd $q$. The characteristic bound in Theorem~\ref{finitecase} is not needed.

\begin{proposition}\label{constituents}
Let $\sigma,\sigma'\in\operatorname{Irr}(G)$. If $\tau\not\simeq\tau^\vee$, then $I_i(\sigma)$ is irreducible. If $\tau\simeq\tau^\vee$, then $I_i(\sigma)$ is the direct sum of exactly $i+1$ pairwise nonisomorphic irreducible representations. In either case, $I_i(\sigma)$ and $I_i(\sigma')$ have no common irreducible constituent when $\sigma\not\simeq\sigma'$.
\end{proposition}

\begin{proof}
The case $i=0$ follows from $I_0(\sigma)\simeq\sigma^\iota$. Assume $i>0$. Choose a cuspidal pair $(L_G,\theta)$ for the Harish--Chandra series containing $\sigma^\iota$. Here $L_G$ is a Levi subgroup of $G$, and $\theta$ is an irreducible cuspidal representation of $L_G$. Put $M:=\GL_{ni}(k)\times G$, $L:=\GL_n(k)^i\times L_G$, and $\eta:=\tau^{\boxtimes i}\boxtimes\theta$. For a cuspidal pair $(L',\eta')$ in a finite reductive group $K$, let $W_K(L',\eta')$ denote the stabilizer of $\eta'$ in the relative Weyl group of the corresponding algebraic Levi subgroup. Set $W_0:=W_G(L_G,\theta)$. Every general linear factor of $L_G$ has size at most $d<n$. A Weyl element normalizing $L$ therefore cannot exchange a new $\GL_n$-block with a general linear block of $L_G$. Permuting the new blocks preserves $\eta$. Reversing a new block replaces $\tau$ by $\tau^\vee$. In the symplectic case, this reversal can be represented on the new symplectic block and acts trivially on the space of $G$.

For the orthogonal groups, consider the determinant of a block reversal. If $\tau\simeq\tau^\vee$, then $n$ is even because $n>1$; see Section~3 of \cite{LS20}.

On a new hyperbolic block with dual bases, reversal is represented by $\bigl(\begin{smallmatrix}0&I_n\\I_n&0\end{smallmatrix}\bigr)$. Its determinant is $(-1)^n$. When $\tau$ is selfdual, this determinant is $1$ because $n$ is even. The reversal therefore belongs to the special orthogonal group and acts as the identity on the defining space of $G$. When $\tau$ is not selfdual, no reversal of a new block preserves $\eta$. Permutations of the new hyperbolic blocks also have determinant $1$. Consequently, the orthogonal groups give the same relative Weyl groups as the symplectic groups. Let $\mathfrak S_i$ be the symmetric group on $i$ letters. Put $A_i:=(\mathbb Z/2\mathbb Z)^i$ and $B_i:=A_i\rtimes\mathfrak S_i$, where $\mathfrak S_i$ permutes the factors of $A_i$. It follows that
\[
W_M(L,\eta)\simeq \mathfrak S_i\times W_0,\qquad W_{H_i}(L,\eta)\simeq
  \begin{cases}
    \mathfrak S_i\times W_0,&\tau\not\simeq\tau^\vee,\\
B_i\times W_0,&\tau\simeq\tau^\vee.
  \end{cases}
\]
We choose the relative Weyl group parameters compatibly with the product $M=\GL_{ni}(k)\times G$. Let $\nu_i\in\operatorname{Irr}(\mathfrak S_i)$ and $E_\sigma\in\operatorname{Irr}(W_0)$ correspond to $\Delta(\tau,i)$ and $\sigma^\iota$. For $E\in\operatorname{Irr}(W_{H_i}(L,\eta))$, let $\Pi_E$ be the corresponding irreducible representation of $H_i$. The comparison theorem, Theorem~3.2.7 of \cite{GM20}, gives
\[
  \begin{aligned}
    &\dim\Hom_{H_i}(\Pi_E,I_i(\sigma))\\
    &\qquad=\dim\Hom_{W_{H_i}(L,\eta)} \left(E,\Ind_{W_M(L,\eta)}^{W_{H_i}(L,\eta)} (\nu_i\boxtimes E_\sigma)\right).
  \end{aligned}
\]
Compatibility with induction through intermediate Levi subgroups allows us to use the chosen product parameters on $M$. The Speh representation occurs with multiplicity one in the induction of $i$ copies of $\tau$ by Section~2.4 of \cite{CZ26}. The comparison theorem therefore gives $\dim\nu_i=1$.

If $\tau\not\simeq\tau^\vee$, the two relative Weyl groups are equal. The induced parameter in the multiplicity formula is then the irreducible representation $\nu_i\boxtimes E_\sigma$. Thus $I_i(\sigma)$ is irreducible.

Suppose $\tau\simeq\tau^\vee$. The induced parameter is
\[
  \Ind_{\mathfrak S_i\times W_0}^{B_i\times W_0} (\nu_i\boxtimes E_\sigma) \simeq \left(\Ind_{\mathfrak S_i}^{B_i}\nu_i\right)\boxtimes E_\sigma.
\]
The representation $\Ind_{\mathfrak S_i}^{B_i}\one$ is the permutation representation on $A_i$. Fourier transformation gives distinct $A_i$-weight lines indexed by subsets $T\subseteq\{1,\ldots,i\}$. The group $\mathfrak S_i$ permutes these lines. For each $a=0,\ldots,i$, the sum of the lines with $|T|=a$ is irreducible. Indeed, an invariant subspace is a sum of $A_i$-weight lines, and $\mathfrak S_i$ is transitive on the subsets of size $a$. These subrepresentations have disjoint $A_i$-weights for different $a$. There are therefore exactly $i+1$ irreducible constituents, each with multiplicity one. Every one dimensional representation of $\mathfrak S_i$ extends to $B_i$ through the quotient $B_i\to\mathfrak S_i$. Tensoring by this extension gives the same conclusion for $\nu_i$.

Within one cuspidal series of $G$, distinct $\sigma$ have distinct parameters $E_\sigma$. Their constituents in $H_i$ are therefore disjoint.

Two cuspidal pairs of $G$ cannot become conjugate after adding the new $\GL_n$-blocks unless they were already conjugate in $G$. To see this, a conjugating Weyl element must preserve the collection of blocks of size $n$. In the orthogonal case, its action on the new hyperbolic blocks has determinant $1$, by the permutation and reversal calculation above. Its restriction to the remaining quadratic space therefore has determinant $1$. Thus its restriction to the remaining space belongs to the Weyl group of $G$ in every case. The uniqueness of cuspidal support now separates different series.
\end{proof}

Assume the characteristic bound in Theorem~\ref{finitecase}. For $\sigma\in\operatorname{Irr}(G)$, the multiplicity space of $\sigma^\vee$ in $J_i$ is $\Hom_G(\sigma^\vee,J_i)$. Theorem~\ref{finitecase} identifies this $H_i$-module with $I_i(\sigma)$. Semisimplicity and the evaluation maps therefore give
\begin{equation}\label{decomposition}
J_i\simeq\bigoplus_{\sigma\in\operatorname{Irr}(G)} \sigma^\vee\boxtimes I_i(\sigma).
\end{equation}
Proposition~\ref{constituents} shows that $J_i$ is multiplicity free as a $G\times H_i$-representation. For $\sigma\in\operatorname{Irr}(G)$ and $\pi\in\operatorname{Irr}(H_i)$,
\[
  \dim\Hom_{G\times H_i}(J_i,\sigma^\vee\boxtimes\pi) =
  \begin{cases}
    1,&\pi\text{ is a constituent of }I_i(\sigma),\\
    0,&\text{otherwise}.
  \end{cases}
\]
Each irreducible constituent on the $H_i$ side determines $\sigma$ uniquely. If $i=0$, formula~(\ref{decomposition}) becomes $J_0\simeq\bigoplus_{\sigma\in\operatorname{Irr}(G)} \sigma^\vee\boxtimes\sigma^\iota$. Formula~(\ref{decomposition}) requires $p>2n(m+i)$. For a fixed field $k$, the proof therefore gives this formula only for the integers $i$ satisfying this bound.

The finite field results include $G=\SO_2^{\mathrm{split}}(k)$. Here $d=1$ and $n>1$, so the divisibility argument eliminates both endpoint terms. Thus (\ref{Ji}) also holds for $\SO_2$.

\appendix
\section{A finite field vanishing criterion}\label{comparison}

We retain $k=\mathbf F_q$, its characteristic $p$, and the cuspidal representation $\tau$ of $\GL_n(k)$ from Section~\ref{Fq}. Fix $c\geq1$ and put $D:=nc$. We order partitions of $D$ by dominance. For a partition $\lambda$ of $D$, let $\Gamma_\lambda$ denote the polarized generalized Gelfand--Graev representation of $\GL_D(k)$ attached to the nilpotent partition $\lambda$, with additive character $\psi^{-1}$. We use the normalization of Section~3.1 of \cite{AT17}.

The finite Speh representation satisfies
\begin{equation}\label{bound}
  \Hom_{\GL_D(k)}(\Gamma_\lambda,\Delta(\tau,c))=0 \qquad\text{unless }\lambda\leq(n^c).
\end{equation}
To prove (\ref{bound}), let $e_c$, $p_n$, $m_\mu$, and $s_\alpha$ be the elementary, power-sum, monomial, and Schur symmetric functions. Let $K_{\alpha\lambda}(q)$ be the Kostka--Foulkes polynomial. The brackets in $e_c[p_n]$ mean substitution of $x_j^n$ for $x_j$. Let $a_\alpha$ be the coefficient of $s_\alpha$ in the Schur expansion of $e_c[p_n]$. Then $e_c[p_n]=\sum_{\alpha\vdash D}a_\alpha s_\alpha$. The formula in Proposition~5.7 of \cite{AT17} and its proof gives
\[
  \dim\Hom_{\GL_D(k)}(\Gamma_\lambda,\Delta(\tau,c)) =\sum_{\alpha\vdash D}a_\alpha K_{\alpha\lambda}(q).
\]
By the definition of $e_c[p_n]$, $e_c[p_n]=\sum_{j_1<\cdots<j_c}x_{j_1}^n\cdots x_{j_c}^n =m_{(n^c)}$. The expansion of $s_\alpha$ in monomial symmetric functions is triangular in dominance order, with diagonal coefficient $1$. Its inverse is triangular in the same order. Consequently, $a_\alpha=0$ unless $\alpha\leq(n^c)$. Also $K_{\alpha\lambda}(q)=0$ unless $\lambda\leq\alpha$. A nonzero summand in the multiplicity formula therefore requires $\lambda\leq\alpha\leq(n^c)$. This proves (\ref{bound}).

The matrix calculations in Section~\ref{Fq} use characters of unipotent subgroups defined by block decompositions. Lemma~\ref{compare} deduces the vanishing of these twisted Jacquet modules from (\ref{bound}). Put $\mathfrak g:=\mathfrak{gl}_D(k)$. A rational grading means an assignment of rational labels $s_1,\ldots,s_D$ to a basis of $k^D$. The matrix entry in row $a$ and column $b$ has degree $s_a-s_b$. The labels are not reduced modulo $p$. Let $\mathfrak g_j^s$ be the space of matrices of degree $j$. We write $\mathfrak g_{\ge a}^s:=\bigoplus_{j\ge a}\mathfrak g_j^s$, with analogous notation for strict inequalities.

\begin{lemma}\label{compare}
Assume $p>2D$. Let $f\in\mathfrak{gl}_D(k)$ have degree $-2$ for a rational grading $s$. Put $\mathfrak u_s:=\mathfrak g_{\geq1}^s$ and $\omega_f(X,Y):=\tr(f[X,Y])$. Choose a maximal isotropic subspace $\mathfrak p_s$ of $\mathfrak u_s$ for $\omega_f$. Let $\exp$ denote the truncated matrix exponential $\exp X=\sum_{r=0}^{D-1}X^r/r!$. Set $P_s:=\exp(\mathfrak p_s)$ and $\chi_f(\exp X):=\psi^{-1}(\tr(fX))$. If $f^n\ne0$, then
\[
J_{P_s,\chi_f}\Delta(\tau,c)=0.
\]
\end{lemma}

\begin{proof}
Since $f$ has degree $-2$, the space $\mathfrak g_{>1}^s$ lies in the radical of $\omega_f$ on $\mathfrak u_s$. Every maximal isotropic subspace contains this radical. Hence $\mathfrak g_{>1}^s\subseteq\mathfrak p_s$. Since $[\mathfrak u_s,\mathfrak u_s]\subseteq\mathfrak g_{>1}^s$, $\mathfrak p_s$ is a Lie algebra. Its nilpotence class is at most $D-1$. Order the basis by decreasing $s$-labels. Then every matrix in $\mathfrak u_s$ is strictly upper triangular. Every product of $D$ such matrices is therefore zero. For $X,A\in\mathfrak u_s$, the exponential and logarithm are the finite polynomials $\exp X:=\sum_{r=0}^{D-1}\frac{X^r}{r!}$ and $\log(I+A):=\sum_{r=1}^{D-1}\frac{(-1)^{r+1}A^r}{r}$. Their denominators are invertible because $p>2D$. These polynomials are inverse to one another. For $X,Y\in\mathfrak p_s$, put $Z:=\log(\exp X\exp Y)$. The Baker--Campbell--Hausdorff formula gives $Z=X+Y+\tfrac12[X,Y]+R$, where $R\in\mathfrak g_{>2}^s$. Every term in $R$ is a product of at least three factors from $\mathfrak u_s$. Since $\mathfrak g_{\ge2}^s\subseteq\mathfrak p_s$, we have $Z\in\mathfrak p_s$. Together with $(\exp X)^{-1}=\exp(-X)$, this proves that $P_s$ is a subgroup. Isotropy gives $\tr(f[X,Y])=0$. The grading gives $\tr(fR)=0$. Thus $\tr(fZ)=\tr(fX)+\tr(fY)$, which proves $\chi_f(\exp X\exp Y)=\chi_f(\exp X)\chi_f(\exp Y)$.

Since $f$ lowers the grading, it is nilpotent. Let $\lambda$ be its Jordan partition. Choose a Jordan basis whose vectors are homogeneous for $s$. On a chain of length $a$, assign the labels $a-1,a-3,\ldots,1-a$. Let $h$ denote this grading. Put $z:=s-h$. The $z$-label is constant on each Jordan chain. Thus $f$ has $z$-degree zero. For rational $t\in[0,1]$, define $s_t:=h+tz$, $\mathfrak u_t:=\mathfrak g_{\geq1}^{s_t}$, $\mathfrak v_t:=\mathfrak g_{>1}^{s_t}$, and $\mathfrak w_t:=\mathfrak g_1^{s_t}$. Put $\mathfrak g^f:=\{X\in\mathfrak g:[X,f]=0\}$. The trace pairing gives
\[
R_t:=\operatorname{rad}(\omega_f|_{\mathfrak u_t}) =\mathfrak v_t+(\mathfrak g^f\cap\mathfrak w_t).
\]
To determine the $h$-degrees in $\mathfrak g^f$, consider chains $e_0,\ldots,e_{a-1}$ and $v_0,\ldots,v_{b-1}$, where $fe_j=e_{j+1}$ and $fv_j=v_{j+1}$ until the last vector. A map commuting with $f$ is determined by the image of $e_0$. Its homogeneous components have $e_0\mapsto v_r$ with $r\geq\max(0,b-a)$, because $f^a v_r=0$. Such a component has $h$-degree $(b-1-2r)-(a-1)=b-a-2r\leq0$. Thus $\mathfrak g^f$ has only nonpositive $h$-degrees. The radical formula at $t=0$ now shows that $\omega_f$ is nondegenerate on $\mathfrak g_1^h$. Since $f$ has $z$-degree zero, the form pairs opposite $z$-degrees. Its restriction to $\mathfrak g_1^h\cap\mathfrak g_0^z$ is therefore nondegenerate. Choose a maximal isotropic subspace $\mathfrak m$ for this restriction. Put $\mathfrak l_t:=\mathfrak m+ (\mathfrak u_t\cap\mathfrak g_{<0}^z)+R_t$ and $\mathfrak r_t:=\mathfrak m+ (\mathfrak u_t\cap\mathfrak g_{>0}^z)+R_t$. Modulo $R_t$, the form pairs every positive $z$-degree space nondegenerately with the opposite negative $z$-degree space. Its zero $z$-degree space is $\mathfrak g_1^h\cap\mathfrak g_0^z$. The choice of $\mathfrak m$ therefore makes $\mathfrak l_t$ and $\mathfrak r_t$ maximal isotropic. They are Lie subalgebras because they contain $R_t$, which contains $[\mathfrak u_t,\mathfrak u_t]$. On their exponential subgroups, the character is $\exp X\mapsto\psi^{-1}(\tr(fX))$, again denoted by $\chi_f$. The commutator pairs $\mathfrak l_t/(\mathfrak l_t\cap\mathfrak r_t)$ with $\mathfrak r_t/(\mathfrak l_t\cap\mathfrak r_t)$ nondegenerately. Summing against this pairing gives an intertwining operator between the two induced representations. The analogous sum in the opposite direction gives an intertwining operator. The composition of these operators is multiplication by the size of either quotient. This scalar is nonzero. The same finite Fourier calculation as in (\ref{finiteJe}) therefore gives
\begin{equation}\label{polarization}
  \Ind_{\exp\mathfrak l_t}^{\GL_D(k)}\chi_f \simeq\Ind_{\exp\mathfrak r_t}^{\GL_D(k)}\chi_f.
\end{equation}

Let $0=t_0<\cdots<t_r=1$ list the endpoints and all values of $t$ for which a matrix unit of nonzero $z$-degree has $s_t$-degree one. Each such matrix unit reaches degree one at most once. There are finitely many matrix units, so this list is finite. Fix two successive values $a:=t_j$ and $b:=t_{j+1}$. Then $\mathfrak r_a\subseteq\mathfrak l_b$. Indeed, a vector of degree one with positive $z$-degree moves into $\mathfrak v_b$. A vector of $\mathfrak v_a$ can leave $\mathfrak v_b$ only by reaching degree one at $b$. Its $z$-degree is then negative, so it lies in $\mathfrak l_b$. A vector of $\mathfrak g^f\cap\mathfrak w_a$ has positive $z$-degree when $a>0$, since its $h$-degree is nonpositive. It remains in $R_b$. This intersection is zero at $a=0$. Finally, $\mathfrak m$ belongs to every $\mathfrak l_b$.

For subgroups $B_1\subseteq B_2\subseteq\GL_D(k)$, let $\chi$ be a character of $B_2$. For $F\in\Ind_{B_1}^{\GL_D(k)}(\chi|_{B_1})$, put $(\mathcal A F)(g):= \sum_{b\in B_1\backslash B_2}\chi(b)^{-1}F(bg)$. The summand is independent of the chosen representative of $B_1b$. Replacing $b$ by $bb_0$ gives $(\mathcal A F)(b_0g)=\chi(b_0)(\mathcal A F)(g)$ for $b_0\in B_2$. Thus $\mathcal A$ maps to $\Ind_{B_2}^{\GL_D(k)}\chi$. On that subspace it equals multiplication by $[B_2:B_1]$. This scalar is nonzero, so $\mathcal A$ is surjective.

We combine these surjections with (\ref{polarization}) for $j=0,\ldots,r-1$. At $t=0$, the grading is $h$. By the definition of $\Gamma_\lambda$, the induced representations in (\ref{polarization}) at $t=0$ are isomorphic to $\Gamma_\lambda$. This gives a surjection
\[
  \Gamma_\lambda\longrightarrow \Ind_{P_s}^{\GL_D(k)}\chi_f.
\]
The induced representations for different maximal isotropic choices at either endpoint are identified by the same finite Fourier argument. If $f^n\ne0$, then $\lambda_1>n$, so $\lambda\not\leq(n^c)$. Applying $\Hom_{\GL_D(k)}(-,\Delta(\tau,c))$ to this surjection gives
\[
  \Hom_{\GL_D(k)} \left(\Ind_{P_s}^{\GL_D(k)}\chi_f,\Delta(\tau,c)\right) \lhook\joinrel\longrightarrow \Hom_{\GL_D(k)}(\Gamma_\lambda,\Delta(\tau,c))=0.
\]
Frobenius reciprocity identifies the source with the $\chi_f$-isotypic subspace of $\Delta(\tau,c)$. Averaging identifies that subspace with $J_{P_s,\chi_f}\Delta(\tau,c)$. This proves the vanishing. The interpolation between gradings follows Lemmas~3.2.1 and 3.2.6--3.2.7 of \cite{GGS17}. All averaging maps used here are over finite groups.
\end{proof}

For a character supported on adjacent blocks of a standard unipotent radical, we give successive blocks labels decreasing by two. Let $f$ represent this character under the trace pairing. Its nonzero blocks lie on the opposite lower adjacent blocks. Then $f$ has degree $-2$. Also $\mathfrak g_1^s=0$, and $\exp\mathfrak g_{\geq2}^s$ is that radical. Lemma~\ref{compare} therefore applies directly to every character supported on adjacent blocks used in Section~\ref{Fq}.

\end{document}